\documentclass[10pt,11pt]{article}
\usepackage{amssymb}
\usepackage[dvips]{graphicx}
\usepackage{amsmath}

\providecommand{\U}[1]{\protect\rule{.1in}{.1in}}
\providecommand{\U}[1]{\protect\rule{.1in}{.1in}}
\newtheorem{theorem}{Theorem}[subsection]

\newtheorem{corollary}[theorem]{Corollary}

\newtheorem{example}[theorem]{Example}
\newtheorem{examples}[theorem]{Examples}

\newtheorem{lemma}[theorem]{Lemma}

\newtheorem{proposition}[theorem]{Proposition}

\newenvironment{proof}[1][Proof]{\textbf{#1.} }{\ \rule{0.5em}{0.5em}}

\def \N{\mathbb{N}}

\def \Z{\mathbb{Z}}
\def \R{\mathbb{R}}

\def \C{\mathcal C}

\begin{document}

\title{Random walks in Weyl chambers and crystals}
\date{December 14, 2010}
\author{C\'{e}dric Lecouvey, Emmanuel Lesigne and Marc Peign\'{e}}

\maketitle

\begin{abstract}
We use Kashiwara crystal basis theory to associate a random walk $\mathcal{W}$
to each irreducible representation $V$ of a simple Lie algebra. This is
achieved by endowing the crystal attached to $V$ with a (possibly non uniform)
probability distribution compatible with its weight graduation. We then prove
that the generalized Pitmann transform defined in \cite{BBOC1} for similar
random walks with uniform distributions yields yet a Markov chain. When the
representation is minuscule, and the associated random walk has a drift in the
Weyl chamber, we establish that this Markov chain has the same law as
$\mathcal{W}$ conditionned to never exit the cone of dominant weights. For 
the defining representation $V$ of $\mathfrak{gl_{n}}$, we notably recover the
main result of \cite{OC1}. At the heart of our proof is a quotient version of
a renewal theorem that we state in the context of general random walks in a
lattice. This theorem also have applications in representation theory since it
permits to precise the behavior of some outer multiplicities for large
dominant weights.
\end{abstract}
\vfill
Laboratoire de Math\'ematiques et Physique Th\'eorique (UMR CNRS 6083)\\
Universit\'e Fran\c{c}ois-Rabelais, Tours \\
F\'ed\'eration de Recherche Denis Poisson - CNRS\\
Parc de Grandmont, 37200 Tours, France.
\\

\noindent
{cedric.lecouvey@lmpt.univ-tours.fr}\\
{emmanuel.lesigne@lmpt.univ-tours.fr}\\
{marc.peigne@lmpt.univ-tours.fr}
\newpage
\tableofcontents

\section{Introduction}

The purpose of the article is to study some interactions between
representation theory of simple Lie algebras over $\mathbb{C}$ and certain
random walks defined on lattices in Euclidean spaces.\ We provide both
results on random walks conditionned to never exit a cone and identities
related to asymptotic representation theory.

The well-known ballot walk in $\mathbb{R}^{n}$ appears as a particular
case of the random walks we consider here.\ Let $B=(\varepsilon _{1},\ldots
,\varepsilon _{n})$ be the standard basis of $\mathbb{R}^{n}$.\ The {ballot
walk} can be defined as the Markov chain {$(\mathcal{W}_{\ell }=X_{1}+\cdots
+X_{\ell })_{\ell \geq 1}$ where $(X_{k})_{k\geq 1}$ is a sequence of
independent and identically distributed random variables taking values in
the base $B$.} Our main motivation is to generalize results due to O'Connell 
\cite{OC1},\cite{OC2} given the law of the random walk $\mathcal{W}=(%
\mathcal{W}_{\ell })_{\ell \geq 1}$ conditioned to never exit the Weyl
chamber $\overline{C}=\{x=(x_{1},\ldots ,x_{n})\in \mathbb{R}^{n}\mid
x_{1}\geq \cdots \geq x_{n}\}.\;$This is achieved in \cite{OC1} by
considering first a natural transformation $\mathfrak{P}$ which associates
to any path with steps in $B$ a path in the Weyl chamber $\overline{C}$,
next by checking that the image of the random walk $\mathcal{W}_{\ell }$ by
this transformation is a Markov chain and finally by establishing that this
Markov chain has the same law as $\mathcal{W}$ conditioned to never exit $%
\overline{C}$.\ The transformation $\mathfrak{P}$ is based on the
Robinson-Schensted correspondence which maps the words on the ordered
alphabet $\mathcal{A}_{n}=\{1<\cdots <n\}$ (regarded as finite paths in $%
\mathbb{R}^{n}$) on pairs of semistandard tableaux.\ It can be reinterpreted
in terms of Kashiwara's crystal basis theory \cite{Kashi} (or equivalently
in terms of the Littelmann path model).\ Each path of length $\ell $ with
steps in $B$ is then interpreted as a vertex in the crystal $B(\omega
_{1})^{\otimes \ell }$ (see \S\ref{subsec_crys} for basics on crystals)
corresponding to the $\ell $-tensor power of the defining representation of $%
\mathfrak{sl}_{n}$. The transformation $\mathfrak{P}$ associates to each
vertex $b\in B(\omega _{1})^{\otimes \ell }$ the highest weight vertex of $%
B(b)$, where $B(b)$ denotes the connected component of $B(\omega _{1})^{\otimes \ell }$
containing $b$.

Let $\mathfrak{g}$ be a simple Lie algebra over $\mathbb{C}$ with weight
lattice $P\subset\mathbb{R}^{N}$ and dominant weights $P_{+}$ (see Section 
\ref{Sec_RT} for some background on root systems and representation theory).
Write also $\overline{C}$ for the corresponding Weyl chamber. For any
dominant weight $\delta\in P_{+},$ let $V(\delta)$ be the
(finite-dimensional) irreducible representation of $\mathfrak{g}$ of highest
weight $\delta$ and denote by $B(\delta)$ its Kashiwara crystal graph.\ One
can then consider the transformation $\mathfrak{P}$ which maps any  $b\in
B(\delta)^{\otimes\ell}$ on  the highest weight vertex $\mathfrak{P}(b)$ of 
 the connected component
 $B(b)$ of the graph $B(\delta)^{\otimes\ell}$
containing the vertex $b$. This simply means that $\mathfrak{P}(b)$ is the
source vertex of $B(b)$ considered as an oriented graph.$\;$This
transformation was introduced in \cite{BBOC1}. As described in \cite{BBOC1},
it can be interpreted as a generalization of the Pitman transform for
one-dimensional paths.

In order to define a random walk from the tensor powers $B(\delta)^{\otimes\ell}$, we
need first to endow $B(\delta)$ with a probability distribution $%
(p_{a})_{a\in B(\delta)}$. Contrary to \cite{BBOC1} where the random walks
considered are discrete version of Brownian motions, we will consider non
uniform distributions on $B(\delta)$ and use random walks with drift in $C$%
.\ Once we have defined a probability on $B(\delta),$ it suffices to
consider the product probability on $B(\delta)^{\otimes\ell}$: the
probability associated to the vertex $b=a_{1}\otimes\cdots\otimes
a_{\ell}\in B(\delta)^{\otimes\ell}$ is then $p_{b}=p_{a_{1}}\cdots
p_{a_{\ell}}$. In order to use Kashiwara crystal basis theory, it is also
natural to impose that the distributions we consider on the crystal $%
B(\delta)^{\otimes\ell}$ are compatible with their weight graduation $%
\mathrm{wt}$, that is $p_{b}=p_{b^{\prime}}$ whenever $\mathrm{wt}(b)=%
\mathrm{wt}(b^{\prime}).\;$Since $\mathrm{wt}(b)=\mathrm{wt}(a_{1})+\cdots+%
\mathrm{wt}(a_{\ell})$, the two conditions $p_{b}=p_{a_{1}}\cdots
p_{a_{\ell}}$ and $p_{b}=p_{b^{\prime}}$ whenever $\mathrm{wt}(b)=\mathrm{wt}%
(b^{\prime})$ essentially impose that our initial distribution on $B(\delta)$
should be exponential with respect to the weight graduation. To be more
precise, recall that $B(\delta)$ is an oriented and colored graph with
arrows $a\overset{i}{\rightarrow}a^{\prime}$ labelled by the simple roots $%
\alpha _{i},i=1,\ldots,n$, of $\mathfrak{g}$.\ Moreover, we have then $%
\mathrm{wt}(a^{\prime})=\mathrm{wt}(a)-\alpha_{i}$.\ We thus associate to
each simple root $\alpha_{i}$ a positive real number $t_{i}$ and consider
only probability distributions on $B(\delta)$ such that $a\overset{i}{%
\rightarrow}a^{\prime}$ implies   $p_{a^{\prime}}=p_{a}\times t_{i}$.\
These distributions and their corresponding product on $%
B(\delta)^{\otimes\ell}$ are then compatible with the weight graduation as
required. The uniform distribution corresponds to the case when $%
t_{1}=\cdots=t_{n}=1$.

The tensor power $B(\delta)^{\otimes\mathbb{N}}$ (defined as the projective
limit of the crystals $B(\delta)^{\otimes\ell}$) can then also be endowed
with a probability distribution of product ype.\ We define the random variable $\mathcal{W%
}_{\ell}$ on $B(\delta)^{\otimes\mathbb{N}}$ by $\mathcal{W}_{\ell}(b)=%
\mathrm{wt}(b(\ell))$ where $b(\ell)=\otimes_{k=1}^{\ell}a_{k}\in
B(\delta )^{\otimes\ell}$ for any $b=\otimes_{k=1}^{+\infty}a_{k}\in
B(\delta )^{\otimes\mathbb{N}}$. This yields a random walk $\mathcal{W}=(\mathcal{W%
}_{\ell })_{\ell\geq1}$ on $P$ whose transition matrix $\Pi_{\mathcal{W}}$
can be easily computed (see \ref{transW}).\ Our next step is to introduce the
random process $\mathcal{H=}(\mathcal{H}_{\ell})_{\ell\geq1}$ on $P_{+}$
such that $\mathcal{H}_{\ell}(b)=\mathrm{wt}(\mathfrak{P}(b(\ell)))$ for any 
$b\in B(\delta)^{\otimes\mathbb{N}}$.\ We prove in Theorem \ref{Th_main}
that $\mathcal{H}$ is a Markov chain and compute its transition matrix $\Pi_{%
\mathcal{H}}$ in terms of the Weyl characters of the irreducible
representations $V(\lambda),\lambda\in P_{+}.$

Write $\Pi _{\mathcal{W}}^{\overline{C}}$ for the restriction of $\Pi _{%
\mathcal{W}}$ to the Weyl chamber $\overline{C}$.\ Proposition \ref{Th_psi}
states that $\Pi _{\mathcal{H}}$ can be regarded as a Doob $h$-transform of $%
\Pi _{\mathcal{W}}^{\overline{C}}$ if and only if $\delta $ is a minuscule
weight.\ As in \cite{OC1} and \cite{OC2}, we use a theorem of Doob (Theorem 
\ref{Th_Doob}) to compare the law of $\mathcal{W}$ conditioned to stay in $%
\overline{C}$ with the law of $\mathcal{H}$.\ Nevertheless the proofs in 
\cite{OC1} and \cite{OC2} use, as a key argument, asymptotic behaviors of
the outer multiplicities in $V(\mu )\otimes V(\delta )^{\otimes \ell }$ when 
$\delta $ is the defining representation of $\mathfrak{sl}_{n}$.\ These
limits seem very difficult to obtain for any minuscule representation by
purely algebraic means.\ To overcome this problem, we establish Theorems \ref
{tll-q} and \ref{ren-q} which can be seen as quotient Local Limit Theorem
and quotient Renewal Theorem for large deviations of random walks
conditioned to stay in a cone.\ They hold in the general context of random
walks in a lattice with drift $m$ in the interior $C$ of a fixed closed cone 
$\overline{C}.$ This yields Theorem \ref{Th_coincide} equating the law of $%
\mathcal{W}$ conditioned to never exit $\overline{C}$ with that of $\mathcal{%
H}$ provided $\delta $ is a minuscule representation and $m=\mathbb{E}(%
\mathcal{W}_{1})$ belongs to $C$. 

The quotient Local Limit Theorem has also some
applications in representation theory since it provides the asymptotic
behavior for the outer multiplicities in $V(\mu )\otimes V(\delta )^{\otimes
\ell }$ when $\delta $ is a minuscule representation (Theorem \ref{Th_Asympt}%
).

The paper is organized as follows. Sections \ref{Sec_RT} and \ref{Sec_Markov}
are respectively devoted to basics on representation theory of Lie algebras
and Markov chains.\ In Section \ref{Sec_Renewal}, we state and prove the
probability results (Theorems \ref{tll-q} and \ref{ren-q}). Section \ref
{Sec_crystl_RW} details the construction of the random walk $\mathcal{W}$.\
The transition matrix of the Markov chain $\mathcal{H}$ is computed in
Section \ref{Sec_MarkovCains}. The main results of Section \ref{Sec_restric}
are Theorem \ref{Th_coincide} and Corollary \ref{Cor_StayinC}. This corollary gives an
explicit formula for the probability that $\mathcal{W}$ remains forever in $%
\overline{C}$. Finally in Section \ref{Sec_miscell}, we study random walks
defined from non irreducible representations and prove Theorem \ref
{Th_Asympt}.

\section{Background on representation theory}

\label{Sec_RT} We recall in the following paragraphs some classical
background on representation theory of simple Lie algebras that we shall
need in the sequel.\ For a complete review, the reader is referred to \cite%
{Bour} or \cite{Hal}.

\subsection{Root systems}

Let $\mathfrak{g}$ be a simple Lie algebra over $\mathbb{C}$ and $\mathfrak{g%
}=\mathfrak{g_{+}}\oplus\mathfrak{h}\oplus\mathfrak{g_{-}}$ a triangular
decomposition.\ We shall follow the notation and convention of \cite{Bour}.\
According to the Cartan-Killing classification, $\mathfrak{g}$ is
characterized (up to isomorphism) by its root system. This root system is
realized in an Euclidean space $\mathbb{R}^{N}$ with standard basis $%
B=(\varepsilon _{1},\ldots,\varepsilon_{N}).\;$We denote by $%
\Delta_{+}=\{\alpha_{i}\mid i\in I\}$ the set of simple roots of $\mathfrak{g%
},$ by $R_{+}$ the (finite) set of positive roots and let $n=\mathrm{card}%
(I)\;$for the rank of $\mathfrak{g}$. The root lattice of $\mathfrak{g}$ is
the integral lattice $Q=\bigoplus_{i=1}^{n}\mathbb{Z\alpha}_{i}.$ Write $%
\omega_{i},i=1,\ldots,n$ for the fundamental weights associated to $%
\mathfrak{g}$. The weight lattice associated to $\mathfrak{g}$ is the
integral lattice $P=\bigoplus_{i=1}^{n}\mathbb{Z\omega}_{i}.$ It can be
regarded as an integral sublattice of $\mathfrak{h}_{\mathbb{R}}^{\ast}$. We
have $\dim(P)=\dim(Q)=n$ and $Q\subset P$.

The cone of dominant weights for $\mathfrak{g}$ is obtained by considering
the positive integral linear combinations of the fundamental weights, that
is 
\begin{equation*}
P_{+}=\bigoplus_{i=1}^{n}\mathbb{N\omega}_{i}.
\end{equation*}
The corresponding Weyl chamber is the cone $C=\bigoplus_{i=1}^{n}\mathbb{R}%
_{>0}\mathbb{\omega}_{i}$. We also introduce its closure $\overline{C}%
=\bigoplus_{i=1}^{n}\mathbb{R}_{\geq0}\mathbb{\omega}_{i}$. In type $A$, we
shall use the weight lattice of $\mathfrak{gl}_{n}$ rather than that of $%
\mathfrak{sl}_{n}$ for simplicity.\ We also introduce the Weyl group $W$ of $%
\mathfrak{g}$ which is the group generated by the reflections $s_{i}$
through the hyperplanes perpendicular to the simple root $\alpha
_{i},i=1,\ldots,n$. Each $w\in W$ may be decomposed as a product of the $%
s_{i},i=1,\ldots,n.\;$All the minimal length decompositions of $w$ have the
same length $l(w)$.\ For any weight $\beta$, the orbit $W\cdot\beta$ of $%
\beta$ under the action of $W$ intersects $P_{+}$ in a unique point. We
define a partial order on $P$ by setting $\mu\leq\lambda$ if $%
\lambda-\mu$ belongs to $Q_{+}=\bigoplus_{i=1}^{n}\mathbb{N\alpha}_{i}$.

\begin{examples}
\ 

\begin{enumerate}
\item For $\mathfrak{g=gl}_{n},$ $P_{+}=\{\lambda =(\lambda _{1},\ldots
,\lambda _{n})\in \mathbb{Z}^{n}\mid \lambda _{1}\geq \cdots \geq \lambda
_{n}\geq 0\}$ and $W$ coincides with the permutation group on $\mathcal{A}%
_{n}=\{1,\ldots ,n\}.$ We have $N=n$. For any $\sigma \in W$ and any $%
v=(v_{1},\ldots ,v_{n}),$ $\sigma \cdot v=(v_{\sigma (1)},\ldots ,v_{\sigma
(n)}).$

\item For $\mathfrak{g=sp}_{2n}$, $P_{+}=\{\lambda =(\lambda _{1},\ldots
,\lambda _{n})\in \mathbb{Z}^{n}\mid \lambda _{1}\geq \cdots \geq \lambda
_{n}\geq 0\}$ and $W$ coincides with the group of signed permutations on $%
\mathcal{C}_{n}=\{\overline{n},\ldots ,\overline{1},1,\cdots ,n\}$, i.e. the
permutations $w$ of $\mathcal{B}_{n}$ such that $\overline{w(x)}=w(\overline{%
x})$ for any $x\in \mathcal{B}_{n}$ (here we write $\overline{x}$ for $-x$
so that $\overline{\overline{x}}=x$). We have $N=n$.\ For any $\sigma \in W$
and any $v=(v_{1},\ldots ,v_{n}),$ $\sigma \cdot v=(\theta _{1}v_{\left|
\sigma (1)\right| },\ldots ,\theta _{n}v_{\left| \sigma (1)\right| })$ where 
$\theta _{i}=1$ if $\sigma (i)$ is positive and $\theta _{i}=-1$ otherwise.

\item For $\mathfrak{g=so}_{2n+1},$ $P_{+}=\{\lambda =(\lambda _{1},\ldots
,\lambda _{n})\in \mathbb{Z}^{n}\sqcup (\mathbb{Z}+\frac{1}{2})^{n}\mid
\lambda _{1}\geq \cdots \geq \lambda _{n}\geq 0\}.$ The Weyl group $W$ is
the same as for $\mathfrak{sp}_{2n}$.

\item For $\mathfrak{g=so}_{2n},$ $P_{+}=\{\lambda =(\lambda _{1},\ldots
,\lambda _{n})\in \mathbb{Z}^{n}\sqcup (\mathbb{Z}+\frac{1}{2})^{n}\mid
\lambda _{1}\geq \cdots \geq \lambda _{n-1}\geq \left| \lambda _{n}\right|
\geq 0\}.$ The Weyl group is the subgroup of the signed permutations on $%
\mathcal{B}_{n}$ switching and even number of signs.
\end{enumerate}
\end{examples}

\subsection{Highest weight modules}

\label{subsec_hwmod}

Let $U(\mathfrak{g})$ be the enveloping algebra associated to $\mathfrak{g}.$
Each finite dimensional $\mathfrak{g}$ (or $U(\mathfrak{g})$)-module $M$
admits a decomposition in weight spaces 
\begin{equation*}
M=\bigoplus_{\mu\in P}M_{\mu}
\end{equation*}
where 
\begin{equation*}
M_{\mu}:=\{v\in M\mid h(v)=\mu(h)v\text{ for any }h\in\mathfrak{h}\}
\end{equation*}
and $P$ is identified with a sublattice of $\mathfrak{h}_{\mathbb{R}}^{\ast}$%
. In particular, $(M\oplus M^{\prime})_{\mu}=M_{\mu}\oplus M_{\mu}^{\prime}$%
.\ The Weyl group $W$ acts on the weights of $M$ and for any $\sigma\in W$,
we have $\dim M_{\mu}=\dim M_{\sigma\cdot\mu}.$ The character of $M$ is the
Laurent polynomial in the variables $x_{1},\ldots,x_{N}$%
\begin{equation*}
\mathrm{char}(M)(x):=\sum_{\mu\in P}\dim(M_{\mu})x^{\mu}
\end{equation*}
where $\dim(M_{\mu})$ is the dimension of the weight space $M_{\mu}$ and $%
x^{\mu}=x_{1}^{\mu_{1}}\cdots x_{N}^{\mu_{N}}$ with $(\mu_{1},\ldots,\mu
_{N})$ the coordinates of $\mu$ on the standard basis $\{\varepsilon
_{1},\cdots,\varepsilon_{N}\}.$

The irreducible finite dimensional representations of $\mathfrak{g}$ are
labelled by the dominant weights. For each dominant weight $\lambda\in
P_{+}, $ let $V(\lambda)$ be the irreducible representation of $\mathfrak{g}$
associated to $\lambda$. The category $\mathcal{C}$ of finite dimensional
representations of $\mathfrak{g}$ over $\mathbb{C}$ is semisimple: each
module decomposes into irreducible components. The category $\mathcal{C}$ is
equivariant to the (semisimple) category of finite dimensional $U(\mathfrak{g%
})$-modules (over $\mathbb{C}$). Any finite dimensional $U(\mathfrak{g})$%
-module $M$ decomposes as a direct sum of irreducible 
\begin{equation*}
M=\bigoplus_{\lambda\in P_{+}}V(\lambda)^{\oplus m_{M,\lambda}}
\end{equation*}
where $m_{M,\lambda}$ is the multiplicity of $V(\lambda)$ in $M$. Here we
slightly abuse the notation by also denoting by $V(\lambda)$ the irreducible
f.d. $U(\mathfrak{g})$-module associated to $\lambda.$

When $M=V(\lambda)$ is irreducible, we set 
\begin{equation*}
s_{\lambda}(x):=\mathrm{char}(M)(x)=\sum_{\mu\in P}K_{\lambda,\mu}x^{\mu}
\end{equation*}
with $\dim(M_{\mu})=K_{\lambda,\mu}.$ Then $K_{\lambda,\mu}\neq0$ only if $%
\mu\leq\lambda$. The characters can be computed from the Weyl character
formula 
\begin{equation}
s_{\lambda}(x)=\frac{\sum_{w\in W}\varepsilon(w)x^{w(\lambda+\rho)-\rho}}{%
\prod_{\alpha\in R_{+}}(1-x^{-\alpha})}.  \label{WVF}
\end{equation}
where $\varepsilon(w)=(-1)^{\ell(w)}$ and $\rho=\frac{1}{2}\sum_{\alpha\in
R_{+}}\alpha_{i}$ is the half sum of positive roots.

Given $\delta,\mu$ in $P_{+}$ and a nonnegative integer $\ell$, we define
the tensor multiplicities $f_{\lambda/\mu,\delta}^{\ell}$ by 
\begin{equation}
V(\mu)\otimes V(\delta)^{\otimes\ell}\simeq\bigoplus_{\lambda\in
P_{+}}V(\lambda)^{\oplus f_{\lambda/\mu,\delta}^{\ell}}.  \label{def_f}
\end{equation}
For $\mu=0$, we set $f_{\lambda,\delta}^{\ell}=f_{\lambda/0,\delta}^{\ell}$.
When there is no risk of confusion, we write simply $f_{\lambda/\mu}^{\ell}$
(resp. $f_{\lambda}^{\ell}$) instead of $f_{\lambda/\mu,\delta}^{\ell}$
(resp. $f_{\lambda,\delta}^{\ell}$). We also define the
multiplicities $m_{\mu,\delta}^{\lambda}$ by 
\begin{equation}
V(\mu)\otimes
V(\delta)\simeq\bigoplus_{\mu\leadsto\lambda}V(\lambda)^{\oplus
m_{\mu,\delta}^{\lambda}}  \label{step}
\end{equation}
where the notation $\mu\leadsto\lambda$ means that $\lambda\in P_{+}$ and $%
V(\lambda)$ appears as an irreducible component of $V(\mu)\otimes V(\delta)$. We have in particular $%
m_{\mu,\delta}^{\lambda}=f_{\lambda/\mu,\delta}^{1}$.\ 

\begin{lemma}
\label{Lem_identity}Consider $\mu \in P_{+}.$ Then for any $\beta \in P,$ we
have the relation 
\begin{equation*}
\sum_{\lambda \in P_{+},\mu \leadsto \lambda }m_{\mu ,\delta }^{\lambda
}K_{\lambda ,\beta }=\sum_{\gamma \in P}K_{\mu ,\gamma }K_{\delta ,\beta
-\gamma }.
\end{equation*}
\end{lemma}

\begin{proof}
According to (\ref{step}), one gets 
\begin{equation*}
s_{\mu }(x)\times s_{\delta }(x)=\sum_{\mu \leadsto \lambda }\sum_{\beta \in
P}m_{\mu ,\delta }^{\lambda }K_{\lambda ,\beta }x^{\beta }.
\end{equation*}
On the other hand 
\begin{equation*}
s_{\mu }(x)\times s_{\delta }(x)=\sum_{\gamma \in P}K_{\mu ,\gamma
}x^{\gamma }\sum_{\xi \in P}K_{\delta ,\xi }x^{\xi }=\sum_{\beta \in
P}\sum_{\gamma \in P}K_{\mu ,\gamma }K_{\delta ,\beta -\gamma }x^{\beta }.
\end{equation*}
By comparing both expressions, we derive the expected relation.
\end{proof}

\subsection{Minuscule representations}

\label{subsec_minus}

The irreducible representation $V(\delta)$ is said minuscule when the orbit $%
W\cdot\delta$ of the highest weight $\delta$ under the action of the Weyl group $%
W$ contains all the weights of $V(\delta).$ In that case, the dominant weight $\delta$ is also called
minuscule. The minuscule weights are fundamental weights and each weight
space in $V(\delta)$ has dimension $1$. They are given by the following
table 
\begin{equation*}
\begin{tabular}{|l|l|l|l|}
\hline
type & minuscule weights & $N$ & decomposition on $B$ \\ \hline
$A_{n}$ & $\omega_{i},i=1,\ldots,n$ & $n+1$ & $\omega_{i}=\varepsilon
_{1}+\cdots+\varepsilon_{i}$ \\ \hline
$B_{n}$ & $\omega_{n}$ & $n$ & $\omega_{n}=\frac{1}{2}(\varepsilon_{1}+%
\cdots+\varepsilon_{n})$ \\ \hline
$C_{n}$ & $\omega_{1}$ & $n$ & $\omega_{1}=\varepsilon_{1}$ \\ \hline
$D_{n}$ & $\omega_{1},\omega_{n-1},\omega_{n}$ & $n$ & $\omega_{1}=%
\varepsilon_{1},\omega_{n+t}=\frac{1}{2}(\varepsilon_{1}+\cdots
+\varepsilon_{n})+t\varepsilon_{n}$, $t\in\{-1,0\}$ \\ \hline
$E_{6}$ & $\omega_{1},\omega_{6}$ & $8$ & $\omega_{1}=\frac{2}{3}%
(\varepsilon_{8}-\varepsilon_{7}-\varepsilon_{6}),\omega_{6}=\frac{1}{3}%
(\varepsilon_{8}-\varepsilon_{7}-\varepsilon_{6})+\varepsilon_{5}$ \\ \hline
$E_{7}$ & $\omega_{7}$ & $8$ & $\omega_{7}=\varepsilon_{6}+\frac{1}{2}%
(\varepsilon_{8}-\varepsilon_{7}).$ \\ \hline
\end{tabular}%
\end{equation*}
Here,  the four infinite families $A_{n},B_{n},C_{n}$ and $D_{n}$
correspond respectively to the classical Lie algebras $\mathfrak{g=gl}_{n},\mathfrak{g=so}%
_{2n+1},\mathfrak{g=sp}_{2n}$ and $\mathfrak{g=so}_{2n}$. When 
$\delta$ is minuscule, one may check that each nonzero multiplicity $%
m_{\mu,\delta}^{\lambda}$ in (\ref{step}) is equal to $1$.

\bigskip

\noindent\textbf{Remark:} We will also need the following classical
properties of minuscule representations.

\begin{enumerate}
\item If $\delta $ is minuscule, then for any $\lambda ,\mu \in P_{+},$ 
one gets $
K_{\delta ,\lambda -\mu }=m_{\mu ,\delta }^{\lambda }\in \{0,1\}$.

\item If $\delta $ is not a minuscule weight, there exists a dominant weight 
$\kappa \neq \delta $ such that $\kappa $ is a weight for $V(\delta ), $
that is $K_{\delta ,\kappa }\neq 0$. This follows from the fact that $%
V(\delta )$ contains a weight $\beta $ which does not belong to $W\cdot
\delta $. Then we can take $\kappa =P_{+}\cap W\cdot \beta $.
\end{enumerate}

\subsection{Paths in a weight lattice}

\label{subsec_paths}

Consider $\delta$ a dominant weight associated to the Lie algebra $\mathfrak{%
g}$.\ We denote by $Z_{\gamma}(\delta,\ell)$ the set of paths of length $\ell
$ in the weight lattice $P$ starting at $\gamma\in P$ with steps the set of
weights of $V(\delta),$ that is the weights $\beta$ such that $%
V(\delta)_{\beta}\neq\{0\}$. When $\gamma=0$, we write for short $%
Z(\delta,\ell)=Z_{0}(\delta,\ell)$.\ A path in $Z_{\gamma}(\delta,\ell)$ can
be identified with a sequence $(\mu^{(0)},\mu^{(1)}\ldots,\mu^{(\ell)})$ of
weights of $V(\delta)$ such that $\mu^{(0)}=\gamma.\;$The position at the $k$%
-th step corresponds to the weight $\mu^{(k)}.$ When $\delta$ is a minuscule
weight, the paths of $Z_{\gamma}(\delta,\ell)$ will be called minuscule. We
will see in \S\ \ref{subsec_crys} that the paths in $Z(\delta,\ell)$ can
then be identified with the vertices of the $\ell$-th tensor product of the
Kashiwara crystal graph associated to the representation $V(\delta)$.

Assume $\gamma\in P_{+}.$ We denote by $Z_{\gamma}^{+}(\delta,\ell)$ the set
of paths of length $\ell$ which never exit the closed Weyl chamber $%
\overline{C}$. Since $\overline{C}$ is convex, this is equivalent to say
that $\mu^{(k)}\in P_{+}$ for any $k=0,\ldots,\ell$.

\begin{examples}
\ 

\begin{enumerate}
\item For $\mathfrak{g=gl}_{n}$ and $\delta =\omega _{1},$ the representation $V(\omega _{1})$
has dimension $n$.\ The weights of $V(\omega _{1})$ are the vectors $%
\varepsilon _{1},\ldots ,\varepsilon _{n}$ of the standard basis $B$.\ The
paths $Z(\omega _{1},\ell )$ are those appearing in the classical ballot
problem.

\item For $\mathfrak{g=sp}_{2n}$ and $\delta =\omega _{1}$, the representation  $V(\omega _{1})$
has dimension $2n$.\ The weights of $V(\omega _{1})$ are the vectors $\pm
\varepsilon _{1},\ldots ,\pm \varepsilon _{n}$.\ In particular a path of $%
Z_{\gamma }(\omega _{1},\ell )$ can return to its starting point $\gamma $.

\item For $\mathfrak{g=so}_{2n+1}$ and $\delta =\omega _{1}$, the representation $V(\omega _{1})
$ has dimension $2n+1$.\ The weights of $V(\omega _{1})$ are $0$ and the
vectors $\pm \varepsilon _{1},\ldots ,\pm \varepsilon _{n}$.\ The
representation $V(\omega _{1})$ is not minuscule$.$

\item For $\mathfrak{g=s0}_{2n+1}$ and $\delta =\omega _{n},$ the representation  $V(\delta )$
has dimension $2^{n}$.\ The coordinates of the weights appearing in $%
V(\delta )$ on the standard basis are of the form $(\beta _{1},\ldots ,\beta
_{n})$ where for any $i=1,\ldots ,n,$ $\beta _{i}=\pm \frac{1}{2}.$
\end{enumerate}
\end{examples}

\noindent\textbf{Remark:} In type $A,$ the coordinates on the standard basis
of the weights appearing in any representation are always nonnegative.\ This
notably implies that a path in $Z_{\gamma}(\delta,\ell)$ can attain a fixed
point of $\mathbb{R}^{N}$ at most one time.\ This special property does
not hold in general for the paths of $Z_{\gamma}(\delta,\ell)$ in types
other than type $A.\;$Indeed, the sign of the weight coordinates can be
modified under the action of the Weyl group. In particular $%
Z_{\gamma}(\delta,\ell)$ contains paths of length $\ell>0$ which return to $%
\gamma$. This phenomenon introduces some complications in the probabilistic
estimations which follows.

\section{Background on Markov chains}

\label{Sec_Markov}

\subsection{Markov chains and conditioning}

Consider a probability space $(\Omega,\mathcal{T},\mathbb{P})$ and a
countable set $M$. Let $Y=(Y_{\ell})_{\ell\geq1}$ be a sequence of random
variables defined on $\Omega$ with values in $M$. The sequence $Y$ is a
Markov chain when 
\begin{equation*}
\mathbb{P}(Y_{\ell+1}=y_{\ell+1}\mid Y_{\ell}=y_{\ell},\ldots,Y_{1} =y_{1})=%
\mathbb{P}(Y_{\ell+1}=y_{\ell+1}\mid Y_{\ell}=y_{\ell}) 
\end{equation*}
for any any $\ell\geq1$ and any $y_{1},\ldots,y_{\ell},y_{\ell+1}\in M$. The
Markov chains considered in the sequel will also be assumed time
homogeneous, that is $\mathbb{P}(Y_{\ell+1}=y_{\ell+1}\mid
Y_{\ell}=y_{\ell})=\mathbb{P} (Y_{\ell}=y_{\ell+1}\mid Y_{\ell-1}=y_{\ell})$
for any $\ell\geq2$.\ For all $x,y$ in $M$, the transition probability from $%
x$ to $y$ is then defined by 
\begin{equation*}
\Pi(x,y)=\mathbb{P}(Y_{\ell+1}=y\mid Y_{\ell}=x) 
\end{equation*}
and we refer to $\Pi$ as the transition matrix of the Markov chain $Y$. The
distribution of $Y_1$ is called the initial distribution of the chain $Y$.
It is well known that the initial distribution and the transition
probability determine the law of the Markov chain and that given a probability distribution and a transition matrix on M, there exists an associated Markov chain.\newline

An example of Markov chain is given by random walk on a group. Suppose that 
$M$ has a group structure and that $\nu$ is a probability measure on $M$;
the random walk of law $\nu$ is the Markov chain with transition
probabilities $\Pi(x,y)=\nu\left(x^{-1}y\right)$; this Markov chain starting
at the neutral element of $M$ can be realized has $(X_1X_2\ldots
X_\ell)_{\ell\geq1}$ where $(X_\ell)_{\ell\geq1}$ is a sequence of
independent and identically distributed random variables, with law $\nu$.%
\newline

Let $Y$ be a Markov chain on $(\Omega,\mathcal T, \mathbb P)$, whose initial distribution has
full support,  i.e.  $ \mathbb P (Y_1=x)>0$ for any $x\in M$. Let $\mathcal{C}$ be a nonempty subset of $M$ and
consider the event $S=(Y_{\ell}\in \C$ for any $\ell\geq1)$. Assume that $%
\mathbb{P}(S\mid Y_{1}=\lambda)>0$ for all $\lambda\in \C$. This implies
that $\mathbb{P}[S]>0$, and we can consider the conditional probability $%
\mathbb{Q}$ relative to this event: $\mathbb{Q}[\cdot]=\mathbb{P}[\cdot|S]$.

It is easy to verify that, under this new probability $\mathbb{Q}$, the
sequence $(Y_{\ell})$ is still a Markov chain, with values in $\mathcal{C}$,
and with transitions probabilities given by 
\begin{equation}
\mathbb{Q}[Y_{\ell+1}=\lambda\mid Y_{\ell}=\mu]=\mathbb{P}[Y_{\ell+1}
=\lambda\mid Y_{\ell}=\mu]\frac{\mathbb{P}[S\mid Y_{1}=\lambda]} {\mathbb{P}%
[S\mid Y_{1}=\mu]}  \label{reco}
\end{equation}
We will denote by $Y^\C$ this Markov chain and  by $\Pi^\mathcal{C}$ the restriction of the transition matrix 
$\Pi$ to the entries which belong to  $\C$ (in other words  
$\displaystyle 
\Pi^\mathcal{C}=\left(\Pi(\lambda,\mu)\right)_{\lambda,\mu\in \mathcal{C}}). 
$

\subsection{Doob $h$-transforms}

\label{sub_sec_Doobh}

A \emph{substochastic matrix} on the countable set $M$ is a map $\Pi:M\times
M\rightarrow\lbrack0,1]$ such that $\sum_{y\in M} \Pi(x,y)\leq1$ for any $%
x\in M.\;$
If 
$\Pi,\Pi^{\prime}$ are substochastic matrices on $M$, we define their
product $\Pi\times\Pi^{\prime}$ as the substochastic matrix given by the
ordinary product of matrices: 
\begin{equation*}
\Pi\times\Pi^{\prime}(x,y)=\sum_{z\in M}\Pi(x,z)\Pi^{\prime}(z,y). 
\end{equation*}

The matrix $\Pi^\C$ defined in the previous subsection is an example of
substochastic matrix.\newline

A function $h:M\rightarrow\mathbb{R}$ is \emph{harmonic} for the
substochastic transition matrix $\Pi$ when we have $\sum_{y\in
M}\Pi(x,y)h(y)=h(x)$ for any $x\in M$. Consider a strictly positive harmonic
function $h$. We can then define the Doob transform of $\Pi$ by $h$
 (also called the $h$-transform of $\Pi$) 
setting 
\begin{equation*}
\Pi_{h}(x,y)=\frac{h(y)}{h(x)}\Pi(x,y). 
\end{equation*}

We then have $\sum_{y\in M}\Pi_{h}(x,y)=1$ for any $x\in M.\;$Thus $\Pi_{h}$
can be interpreted as the transition matrix for a certain Markov chain.

An example is given in the second part of the previous subsection (see
formula (\ref{reco})): the state space is now $\mathcal{C}$, the
substochastic matrix is $\Pi^\mathcal{C}$ and the harmonic function is $h_{%
\mathcal{C}}(\lambda):=\mathbb{P}[S\mid Y_{1}=\lambda]$; the transition
matrix $\Pi^\mathcal{C}_{h_\mathcal{C}}$ is the transition matrix of the
Markov chain $Y^\C$.

\subsection{Green function and Martin kernel}

Let $\Pi$ be a substochastic matrix on the set $M$. Its Green function is
defined as the series 
\begin{equation*}
\Gamma(x,y)=\sum_{\ell\geq0}\Pi^{\ell}(x,y). 
\end{equation*}
(If $\Pi$ is the transition matrix of a Markov chain, $\Gamma(x,y)$ is the
expected value of the number of passage at $y$ of the Markov chain starting
at $x$.)

Assume there exists $x^{\ast}$ in $M$ such that $0<\Gamma(x^{\ast},y)<\infty$
for any $y\in M$. Fix such a point $x^{\ast}$.\ The Martin kernel associated
to $\Pi$ (with reference point $x^{\ast}$) is then defined by 
\begin{equation*}
K(x,y)=\frac{\Gamma(x,y)}{\Gamma(x^{\ast},y)}. 
\end{equation*}

Consider a positive harmonic function $h$ and denote by $\Pi_{h}$  the $h$%
-transform of $\Pi$. Consider the Markov chain $Y^{h}=\left(Y^{h}_\ell%
\right)_{\ell\geq1}$ starting at $x^{\ast}$ and  whose transition matrix is $\Pi_{h}$.

\begin{theorem}
\label{Th_Doob}(Doob) Assume that there exists a function $f:M\rightarrow 
\mathbb{R}$ such that for all $x\in M$, $\lim_{\ell\rightarrow+\infty
}K(x,Y_{\ell}^{h})=f(x)$ almost surely. Then there exists a positive real
constant $c$ such that $f=ch$.
\end{theorem}

For the sake of completeness, we detail a proof of this Theorem in the
Appendix.

\section{Quotient renewal theorem for a random walk in a cone}

\label{Sec_Renewal} The purpose of this section is to establish a renewal
theorem for a random walk forced to stay in a cone. We state this
theorem in the weak form of a quotient theorem : see Theorem \ref{ren-q}.
This result is a key ingredient in our proof of Theorem \ref{Th_coincide} ; it  is a purely probabilistic result whose proof can be read
independently of the reminder of the article.

We begin by the statement and proof of a quotient local limit theorem for a
random walk forced to stay in a cone (Theorem \ref{tll-q}). This local
limit theorem is easier to establish  than the renewal one  and the ideas of its
proof will be reinvested in the proof of Theorem \ref{ren-q}.

\subsection{Probability to stay in a cone}

\label{stayC} Let $(X_{\ell})_{{\ell}\geq 1}$ be a sequence of random
variables in an Euclidean space ${\mathbb{R}}^n$, independent and
identically distributed, defined on a probability space $(\Omega,{\mathcal{T}%
},\mathbb{P})$. We assume that these variables have moment of order 1 and
denote by $m$ their common mean. Let us denote by $(S_{\ell})_{{\ell}\geq 0}$
the associated random walk defined by $S_0=0$ and $S_{\ell}:=
X_1+\cdots+X_{\ell}$. We consider a cone $\mathcal{C}$ in $\mathbb{R}^n$. We
assume it contains an open convex sub-cone $\mathcal{C}_0$ such that $m\in%
\mathcal{C}_0$ and $\mathbb{P}[X_\ell\in \mathcal{C}_0] >0.$ First, one gets the

\begin{lemma}
\label{psc} 
\begin{equation*}
\mathbb{P}\left[ \forall {\ell} \geq 1, \ S_{\ell} \in \mathcal{C}\right] >0.
\end{equation*}
\end{lemma}

\begin{proof}
It suffices to prove the lemma for $\mathcal{C}_0$, that is, we can assume
(and we will in the sequel) that $\mathcal{C}$ is open and convex. Fix $a$
in $\mathcal{C}$ such that $\mathbb{P}\left[X_\ell\in B(a,\varepsilon)\right]%
>0$, for any $\varepsilon>0$. Such an element does exist since the cone is
charged by the law of $X$.

By the strong law of large numbers, the sequence $\left(\frac1{\ell}
S_{\ell}\right)$ converges almost surely to $m$. Therefore, almost surely,
one gets $S_{\ell}\in\mathcal{C}$ for any large enough ${\ell}$, that is 
\begin{equation*}
\mathbb{P}\left[\exists L, \forall \ell\geq L, S_\ell\in\mathcal{C}\right]%
=\lim_{L\to+\infty} \mathbb{P}\left[\forall {\ell}\geq L,\ S_{\ell}\in%
\mathcal{C}\right]=1.
\end{equation*}

For any $x\in\mathbb{R}^n$, there exists $k\in\mathbb{N}$ such that $x+ka\in%
\mathcal{C}$. Thus 
\begin{equation*}
\left(\forall {\ell}\geq L,\ S_{\ell}\in\mathcal{C}\right)=\bigcup_{k\geq0}%
\left((\forall {\ell}\geq L,\ S_{\ell}\in\mathcal{C}) \text{ and } (\forall {%
\ell}<L, \ S_{\ell}+ka\in\mathcal{C})\right) 
\end{equation*}
and therefore 
$\displaystyle 
\left(\forall {\ell}\geq L,\ S_{\ell}\in\mathcal{C}\right)\subset\bigcup_{k%
\geq0}\left(\forall {\ell}, \ S_{\ell}+ka\in\mathcal{C}\right) $
since the cone is stable under addition. Hence, there exists $k\geq0$ such that 
$\mathbb{P}\left[\forall {\ell}, \ S_{\ell}+ka\in\mathcal{C}\right]>0.$ Fix
such a $k$ and   $\delta>0$ such that $B(a,\delta)\subset\mathcal{C}$. If $%
X_{\ell}\in B(a,\frac\delta{k+1})$ for any ${\ell}\leq k+1$, then $%
S_{k+1}-ka\in B(a,\delta)$. We have 
\begin{multline*}
\mathbb{P}\left[ \forall {\ell} \geq 1, \ S_{\ell} \in \mathcal{C}\right]\geq
\\
\mathbb{P}\left[\left((\forall {\ell}\leq k+1,\ X_{\ell}\in B(a,\frac\delta{%
k+1})\right)\text{ and }\left(\forall {\ell}>k+1,\ S_{\ell}-S_{k+1}+ka\in%
\mathcal{C}\right)\right]= \\
\left(\mathbb{P}\left[X\in B(a,\frac\delta{k+1})\right]\right)^{k+1}\mathbb{P%
}\left[\forall {\ell},\ S_{\ell}+ka\in\mathcal{C}\right]>0.
\end{multline*}
\end{proof}

\subsection{Local Limit Theorem}

We now assume that the support $S_\mu$ of the law $\mu$ of the random variables 
$X_l$ is a subset of $\Z^n$. We suppose that $\mu$ is {\it adapted} on $\Z^n$, which means that the group $\langle S_\mu\rangle$ generated by the elements of $S_\mu$ is equal to $\Z^n$.

The local limit theorem precises the   behavior as $\ell \to +\infty$ of the probability $\mathbb{P}[S_\ell=g]$ for $g \in \Z^n$.
 It thus appears a phenomenon   of ``periodicity''. For instance, for the classical random walk to the closest neighbor on $\Z$, the walk $(S_\ell)_{\ell\geq 1}$ may visit an odd site only at an odd time ; this is due to the fact that in this case the support of $\mu$   is included in $2\Z+1$, that is   a coset of a proper subgroup of $\Z$. 
 
 Here is a classical definition : we say that the law $\mu$ is aperiodic if there is no proper subgroup $G$ of $\langle S_\mu\rangle$ and vector $b\in\langle S_\mu\rangle$ such that $S_\mu\subset b+G$. One may easily check that if the law $\mu$ is aperiodic then translating $\mu$ by $-b$ with $b\in S_\mu$ leads to a new law $\mu'$ which is aperiodic and such that $G=\langle S_{\mu'}\rangle$ is a proper subgroup of $\langle S_\mu\rangle$.
  
  Let us give some various examples :

%

\medskip

Examples : \newline
\textbf{1.} Denote by $(e_i)_{1\leq i\leq n}$ the canonical basis of $
\mathbb{R}^n$. If the variables $X_{\ell}$ take their values in the set $
\{e_i\mid 1\leq i\leq n\}\cup\{0\}$, if $\mathbb{P}[X_\ell=e_i]>0$ for any $i
$ and if $\mathbb{P}[X_\ell=0]>0$, then 
$\langle S_\mu\rangle=\Z^n$ and the law is aperiodic. \newline
\textbf{2.} If the variables $X_{\ell}$ take their values in $\{\pm e_i\mid
1\leq i\leq n\}$, if $\mathbb{P}[X_\ell=e_i]>0$ for any $i$ and if $\mathbb{P%
}[X_\ell=-e_i]>0$ for at least one $i$, then we can take $b=e_1$ and $%
G=\{(x_1,x_2,\ldots,x_n)\in\mathbb{Z}^n\mid x_1+x_2+\ldots+x_n\text{ is even}%
\}$ is the subgroup of $\mathbb{Z}^n$ generated by the $e_i+e_j, 1\leq
i,j\leq n$. \newline
\textbf{3.} If the variables $X_{\ell}$ take their values in $\{e_i\mid
1\leq i\leq n\}$ and if $\mathbb{P}[X_\ell=e_i]>0$ for any $i$, then we can
take $b=e_1$ and $G$ is the $n-1$-dimensional subgroup generated by the
vectors $e_i-e_1$, $2\leq i\leq n$.\newline

The translation by $-b$ thus permits to limit ourselves to aperiodic random
walks in a group of suitable dimension $d$, (with possibly $d<n$). By replacing the random walk $%
(S_{\ell})$ by the random walk $(S_{\ell}-{\ell}b)$, we can therefore
restrict ourselves to an adapted aperiodic random walk in the discrete group 
$G$.

We denote   by $d$ the dimension of the group $G$ once this
translation is performed. We also assume that the random walk admits a
moment of order $2$. We write $m$ for the mean vector of $X$ and $\Gamma$ for
the covariance matrix. A classical form of the local limit theorem is the
following: 
\begin{equation*}
\lim_{{\ell}\to+\infty}\sup_{g\in G}\left|\frac{{\ell}^{d/2}}{v(G)}\mathbb{P}%
[S_{\ell}=g]-\mathcal{N}_\Gamma\left(\frac1{\sqrt{{\ell}}}(g-{\ell}%
m)\right)\right|=0, 
\end{equation*}
where $\mathcal{N}_\Gamma$ is the Gaussian density 
\begin{equation*}
\mathcal{N}_\Gamma(x)=(2\pi)^{-d/2} (\det{\Gamma})^{-1/2}\exp\left(-\frac12%
\, x\cdot\Gamma^{-1}\cdot{}^tx\right)\ ,\quad x\in\mathbb{R}^d, 
\end{equation*}
and $v(G)$ is the volume of an elementary cell of $G$. This result gives an
equivalent of $\mathbb{P}[S_{\ell}=g_{\ell}]$ \ providing that $g_{\ell}-{%
\ell}m=O(\sqrt {\ell})$. The local limit theorem for large deviations (see
the original article \cite{richter} or the classical book \cite{Ib-Lin})
yields and equivalent when $g_{\ell}-{\ell}m=o({\ell})$. It takes a
particularly simple form when $g_{\ell}-{\ell}m=o({\ell}^{2/3})$.

\begin{theorem}
\label{TLL-R} Assume the random variables $X_{\ell}$ have an exponential
moment, that is, there exists $t>0$ such that $\mathbb{E}\left[%
\exp(t|X_\ell|)\right]<+\infty$. Let $(a_{\ell})$ be a sequence of real
numbers such that $\lim a_{\ell} {\ell}^{-2/3}=0$. Then, when ${\ell}$ tends
to infinity, we have 
\begin{equation*}
\mathbb{P}[S_{\ell}=g]\sim v(G) {\ell}^{-d/2}\mathcal{N}_\Gamma\left(\frac1{%
\sqrt{{\ell}}}(g-{\ell}m)\right) 
\end{equation*}
uniformly in $g\in G$ such that $\|g-{\ell}m\|\leq a_{\ell}$.
\end{theorem}

Under the hypotheses of Theorem \ref{TLL-R}, we derive the following
equivalent. Given $(g_{\ell}), (h_{\ell})$ two sequences in $G$ such that $%
\lim {\ell}^{-2/3}\|g_{\ell}-{\ell}m\|=0$ and $\lim {\ell}%
^{-1/2}\|h_{\ell}\|=0$, we have 
\begin{equation}  \label{tll-equ}
\mathbb{P}[S_{\ell}=g_{\ell}+ h_{\ell}]\sim \mathbb{P}[S_{\ell}=g_{\ell} ].
\end{equation}
Observe that there exists a stronger version of this result where the
exponent $2/3$ is replaced by $1$ but the equivalent obtained is more
complicated, and we will not need it in the present article.

\subsection{A quotient LLT for the random walk restricted to a cone}

We assume that the hypotheses of the previous Subsection are satisfied.

\begin{theorem}
\label{tll-q}Assume the random variables $X_{\ell}$ are almost surely
bounded.
Let $(g_{\ell}), (h_{\ell})$ be two sequences in $G$ and $\alpha<2/3$ such
that $\lim {\ell}^{-\alpha}\|g_{\ell}-{\ell}m\|=0$ and $\lim {\ell}%
^{-1/2}\|h_{\ell}\|=0$. Then, when ${\ell}$ tends to infinity, we have 
\begin{equation*}
\mathbb{P}\left[S_1\in\mathcal{C},\ldots,S_{\ell}\in\mathcal{C}%
,S_{\ell}=g_{\ell}+h_{\ell}\right]\sim \mathbb{P}\left[S_1\in\mathcal{C}%
,\ldots,S_{\ell}\in\mathcal{C},S_{\ell}=g_n\right]. 
\end{equation*}
\end{theorem}

The following lemma will play a crucial role; in order to keep a direct way
to our principal results, we postpone its proof in an appendix.

\begin{lemma}
\label{lem-marc} Assume the random variables $X_{\ell}$ are almost surely
bounded. Let $\alpha\in]1/2,2/3[$. If the sequence $({\ell}^{-\alpha}\|g_{\ell}-{\ell}%
m\|)$ is bounded, then there exists $c>0$ such that, for all large enough $%
\ell$, 
\begin{equation*}
\mathbb{P}\left[S_1\in\mathcal{C},\ldots,S_{\ell}\in\mathcal{C}%
,S_{\ell}=g_{\ell}\right]\geq\exp\left(-c\ell^\alpha\right). 
\end{equation*}
\end{lemma}

In the sequel, we will use the following notation : if $(u_\ell)$ and $%
(v_\ell)$ are two real sequences, we write $u_\ell\preceq v_\ell$ when there
exists a constant $\kappa>0$ such that  
$u_\ell\leq\kappa\, v_\ell$  for all  large enough $\ell$.

\begin{proof}[Proof of Theorem \protect\ref{tll-q}]
Fix a real number $\beta$ such that $\frac12<\alpha<\beta<\frac23$ and set $%
b_{\ell}=[{\ell}^\beta]$. Let $\delta>0$ be such that $B(m,\delta)\subset%
\mathcal{C}$. Set
$
B_{{\ell}}=B({\ell}m,{\ell}\delta).
$

For any ${\ell}\geq1$ we have $B_{\ell}\subset \mathcal{C}$. We are going to
establish that 
\begin{equation}  \label{Q_n}
Q_{\ell}:= {\frac{\mathbb{P}\left[ S_1\in \mathcal{C}, \ldots, S_{\ell} \in 
\mathcal{C}, S_{\ell}=g_{\ell}\right] }{\mathbb{P}\left[ S_1\in \mathcal{C},
\ldots, S_{b_{\ell}} \in \mathcal{C}, S_{b_{\ell}}\in B_{b_{\ell}},
S_{\ell}=g_{\ell}\right] }}\ \to 1.
\end{equation}
By the Cramer-Chernoff large deviations inequality, there exists $%
c=c(\delta)>0$ such that 
\begin{equation*}
\mathbb{P}\left[S_{b_{\ell}}\notin B_{b_{\ell}}\right]=\mathbb{P}\left[%
\left\|S_{b_{\ell}}-b_{\ell}m\right\|\geq b_{\ell}\delta\right]%
\leq\exp(-cb_{\ell})\preceq\exp(-cl^\beta). 
\end{equation*}
By Lemma \ref{lem-marc}, there also exists $c^{\prime }>0$ such that 
\begin{equation*}
\mathbb{P}\left[ S_1\in \mathcal{C}, \ldots, S_{{\ell}} \in \mathcal{C},
S_{\ell}=g_{\ell}\right]\geq \exp(-c^{\prime }{\ell}^\alpha). 
\end{equation*}
Since $\alpha<\beta$, we thus have 
\begin{equation*}
\mathbb{P}\left[S_{b_{\ell}}\notin B_{b_{\ell}}\right]=o\left(\mathbb{P}%
\left[ S_1\in \mathcal{C}, \ldots, S_{{\ell}} \in \mathcal{C},
S_{\ell}=g_{\ell}\right]\right). 
\end{equation*}
This gives 
\begin{equation}  \label{corrun}
\mathbb{P}\left[ S_1\in \mathcal{C}, \ldots, S_{\ell} \in \mathcal{C},
S_{b_{\ell}}\in B_{b_{\ell}}, S_{\ell}=g_{\ell}\right]\sim \mathbb{P}\left[
S_1\in \mathcal{C}, \ldots, S_{{\ell}} \in \mathcal{C}, S_{\ell}=g_{\ell}%
\right],
\end{equation}
and in particular for any large enough ${\ell}$ 
\begin{equation}  \label{corrde}
\mathbb{P}\left[ S_1\in \mathcal{C}, \ldots, S_{b_{\ell}} \in \mathcal{C},
S_{b_{\ell}}\in B_{b_{\ell}}, S_{\ell}=g_{\ell}\right]\geq \frac12 \mathbb{P}%
\left[ S_1\in \mathcal{C}, \ldots, S_{{\ell}} \in \mathcal{C},
S_{\ell}=g_{\ell}\right].
\end{equation}
Now set 
\begin{equation*}
U_{\ell}:= {\frac{ \mathbb{P}\left[ S_1\in \mathcal{C}, \ldots, S_{{\ell}}
\in \mathcal{C}, S_{b_{\ell}}\in B_{b_{\ell}}, S_{\ell}=g_{\ell}\right] }{%
\mathbb{P}\left[ S_1\in \mathcal{C}, \ldots, S_{b_{\ell}} \in \mathcal{C},
S_{b_{\ell}}\in B_{b_{\ell}}, S_{\ell}=g_{\ell}\right] }}\ . 
\end{equation*}
Write $\varepsilon={\rm dist}(m,\mathcal{C}^c)$, and recall that $\varepsilon>0$. So,
when the walk goes out the cone, its distance to the point $km$ is at least $%
k\varepsilon$.

We have, 
\begin{equation*}
0\leq 1-U_{\ell}  \leq {\frac{\displaystyle 
\sum_{k\geq1}  \mathbb{P}\left[S_{b_{\ell}+k}\notin \mathcal{C}\right]  }{%
\mathbb{P}\left[ S_1\in \mathcal{C}, \ldots, S_{b_{\ell}} \in \mathcal{C},
S_{b_{\ell}}\in B_{b_{\ell}}, S_{\ell}=g_{\ell}\right]  }}\ .  
\end{equation*}
Thus by (\ref{corrde}) and our choice of $\varepsilon$, we have for any
sufficiently large ${\ell}$,  
\begin{equation}  \label{uene}
|1-U_{\ell}|\leq 2\,{\frac{ \sum_{k\geq b_{\ell}} \mathbb{P}\left[ \Vert
S_k-km\Vert \geq k\varepsilon \right] }{\mathbb{P}\left[ S_1\in \mathcal{C},
\ldots, S_{{\ell}} \in \mathcal{C}, S_{\ell}=g_{\ell}\right] }}\ .
\end{equation}

We deduce from the large deviations inequality that there exists a constant $%
c^{\prime \prime }=c^{\prime \prime }(\varepsilon)>0$ such that for any $k$, 
\begin{equation*}
\mathbb{P}\left[\|S_k-km\|\geq k\varepsilon\right]\leq\exp(-c^{\prime \prime
}k). 
\end{equation*}
Therefore 
\begin{equation*}
\sum_{k\geq b_{\ell}}  \mathbb{P}\left[ \Vert S_k-km\Vert \geq k\varepsilon %
\right]\preceq\exp(-c^{\prime \prime }b_{\ell})\preceq\exp(-c^{\prime \prime
}{\ell}^\beta).  
\end{equation*}
Since $\alpha<\beta$, Lemma \ref{lem-marc} and (\ref{uene}) implies that $%
\lim U_{\ell}=1.$ Together with (\ref{corrun}) this proves (\ref{Q_n}).
Moreover, the same result holds if we replace $g_{\ell}$ by $%
g_{\ell}+h_{\ell}$ for $\lim {\ell}^{-\alpha}\|g_{\ell}+h_{\ell}-{\ell}m\|=0$. 
To achieve the proof of Theorem \ref{tll-q}, it now suffices to establish
that 
\begin{equation*}
{\frac{\mathbb{P}\left[ S_1\in \mathcal{C}, \ldots, S_{b_{\ell}} \in 
\mathcal{C}, S_{b_{\ell}}\in B_{b_{\ell}}, S_{\ell}=g_{\ell}+h_{\ell}\right] 
}{\mathbb{P}\left[ S_1\in \mathcal{C}, \ldots, S_{b_{\ell}} \in \mathcal{C},
S_{b_{\ell}}\in B_{b_{\ell}}, S_{\ell}=g_{\ell}\right]  }}\ \to 1. 
\end{equation*}
Since the increments of the random walk are independent and stationnary, we have 
\begin{multline*}
\mathbb{P}\left[ S_1\in \mathcal{C}, \ldots, S_{b_{\ell}} \in \mathcal{C},
S_{b_{\ell}}\in B_{b_{\ell}}, S_{\ell}=g_{\ell}\right] \\
= \sum_{x\in B_{b_{\ell}}\cap G} \mathbb{P}\left[S_{{\ell}%
-b_{\ell}}=g_{\ell}-x\right]\times \mathbb{P}\left[S_1\in \mathcal{C},
\ldots, S_{b_{\ell}} \in \mathcal{C}, S_{b_{\ell}}=x\right].
\end{multline*}
This leads to the theorem since by (\ref{tll-equ}) we have 
\begin{equation*}
\mathbb{P}\left[S_{{\ell}-b_{\ell}}=g_{\ell}-x\right]\sim\mathbb{P}\left[S_{{%
\ell}-b_{\ell}}=g_{\ell}+h_{\ell}-x\right], 
\end{equation*}
uniformly in $x\in B_{b_{\ell}}$. (Indeed $\|g_{\ell}-x-({\ell}%
-b_{\ell})m\|\preceq {\ell}^\beta\delta$, uniformly in $x\in B_{b_{\ell}}$.)
\end{proof}

\subsection{Renewal theorem}

Assume now that the random variables $X_{\ell}$ take values in a discrete
subgroup of $\mathbb{R}^n$, and denote by $G$ the group generated by their
law $\mu$. We assume that $G$ linearly generates the whole space $\mathbb{R}%
^n$. Denote $U$ the associated renewal measure defined by $U:=\sum_{j\geq0}
\mu^{*j}$. Equivalently, we have for any $g\in G$, 
\begin{equation*}
U(g)=\sum_{j=0}^{+\infty} \mathbb{P}[S_j=g].
\end{equation*}

Let us first insist  that there is no hypothesis of aperiodicity in the following statement~; this is due to the fact that the quantity  $U(g)$ represents the expected number of visits of  $g$ by the whole random walk and not only at a precise time, like  the local limit theorem does.  In particular, the law of the variables $X_l$ may be supported by a proper coset of $G$ with  no consequences on the behavior of $U(g)$ as $g\to \infty$.

We assume that $m:=\mathbb{E}[X_{\ell}]$ is nonzero. The renewal theorem
tells us that, when $g$ tends to infinity in the direction $m$, we have 
\begin{equation*}
U(g)\sim\frac{v(G)}{\sigma}(2\pi)^{-(n-1)/2} \|m\|^{(n-3)/2}\|g\|^{-(n-1)/2}
\end{equation*}
where $\sigma^2$ is the determinant of the covariance matrix associated to
the orthogonal projection of $X_{\ell}$ on the hyperplan orthogonal to $m$.
More precisely, we have the following theorem.

Let $(e_1,e_2,\ldots, e_{n-1})$ be an orthonormal basis of the hyperplan $%
m^\perp$. If $x\in\mathbb{R}^n$, denote by $x^{\prime }$ its orthogonal
projection on $m^\perp$ expressed in this basis (here $x^{\prime }$ is
regarded as a row vector). Finally let $B$ be the covariance matrix of the
random vector $X^{\prime }_{\ell}$.

Recall that $\mathcal{N}_B$ is the $(n-1)$-dimensional Gaussian density
given by 
\begin{equation*}
\mathcal{N}_B(x^{\prime})=(2\pi)^{-(n-1)/2} (\det B)^{-1/2}\exp\left(-\frac12\,
x^{\prime}\cdot B^{-1}\cdot{}^tx^{\prime }\right). 
\end{equation*}

\begin{theorem}
\label{ren-CW} We assume the random variables $X_{\ell}$ have an exponential
moment. Let $(a_{\ell})$ be a sequence of real numbers such that $\lim
a_{\ell} {\ell}^{-2/3}=0$. Then, when $\ell$ goes to infinity, we have 
\begin{equation*}
U(g)\sim \frac{v(G)}{\|m\|} {\ell}^{-(n-1)/2}\,\mathcal{N}_{B}\left(\frac1{%
\sqrt {\ell}}g^{\prime }\right)
\end{equation*}
uniformly in $g\in G$ such that $\|g-{\ell}m\|\leq a_{\ell}$.
\end{theorem}

This theorem has been proved by H. Carlsson and S. Wainger in the case of
absolutely continuous distribution (\cite{CW}). We did not find in the
literature the lattice distribution version we state here. A detailed proof
of this version is given in \cite{note_LLP}.

Under the hypotheses of Theorem \ref{ren-CW}, we see in particular that if $%
(g_{\ell}), (h_{\ell})$ are two sequences in $G$ such that $\lim {\ell}%
^{-2/3}\|g_{\ell}-{\ell}m\|=0$ and $\lim {\ell}^{-1/2}\|h_{\ell}\|=0$, then 
\begin{equation}  \label{ren-equ}
U(S_{\ell}=g_{\ell}+ h_{\ell})\sim U(S_{\ell}=g_{\ell} ).
\end{equation}

\subsection{A quotient renewal theorem for the random walk restricted to a
cone}

\begin{theorem}
\label{ren-q} Assume the random variables $X_{\ell}$ are almost surely
bounded. 
Let $(g_{\ell}), (h_{\ell})$ two sequences in $G$ such that there exists $%
\alpha<2/3$ with $\lim {\ell}^{-\alpha}\|g_{\ell}-{\ell}m\|=0$ and $\lim {%
\ell}^{-1/2}\|h_{\ell}\|=0$. Then, when ${\ell}$ tends to infinity, we have 
\begin{equation*}
\sum_{j\geq1}\mathbb{P}\left[S_1\in\mathcal{C},\ldots,S_j\in\mathcal{C}%
,S_j=g_{\ell}+h_{\ell}\right]\sim \sum_{j\geq1}\mathbb{P}\left[S_1\in%
\mathcal{C},\ldots,S_j\in\mathcal{C},S_j=g_{\ell}\right]. 
\end{equation*}
\end{theorem}

\begin{proof}
Fix a real number $\beta$ such that $\frac12<\alpha<\beta<\frac23$ and set $%
b_{\ell}=[{\ell}^\beta]$. Let $\delta>0$ be such that $B(m,\delta)\subset%
\mathcal{C}$ and set 
$
B_{{\ell}}=B({\ell}m,{\ell}\delta).
$
For any ${\ell}\geq1$, we have $B_{\ell}\subset \mathcal{C}$.

We are going to prove that 
\begin{equation}  \label{R_n}
\frac{\sum_{j\geq1}\mathbb{P}\left[ S_1\in \mathcal{C}, \ldots, S_{b_{\ell}}
\in \mathcal{C}, S_{b_{\ell}}\in B_{b_{\ell}}, S_j=g_{\ell}\right] }{{%
\sum_{j\geq1}\mathbb{P}\left[ S_1\in \mathcal{C}, \ldots, S_j \in \mathcal{C}%
, S_j=g_{\ell}\right] }}\ \to 1\quad\text{when ${\ell}\to+\infty$.}
\end{equation}

Observe first that there exists $c_1>0$ such that $\mathbb{P}[S_j=g_{\ell}]=0
$ if $j<c_1{\ell}$ since the support of $\mu$ is bounded.

We prove first that 
\begin{equation}  \label{T_n}
T_{\ell}:= {\frac{\sum_{j\geq c_1{\ell}}\mathbb{P}\left[ S_1\in \mathcal{C},
\ldots, S_{j} \in \mathcal{C}, S_{b_{\ell}}\in B_{b_{\ell}}, S_j=g_{\ell}%
\right] }{\sum_{j\geq c_1{\ell}}\mathbb{P}\left[ S_1\in \mathcal{C}, \ldots,
S_j \in \mathcal{C}, S_j=g_{\ell}\right] }}\ \to 1\quad\text{when ${\ell}%
\to+\infty$}.
\end{equation}
We have 
\begin{equation*}
0\leq1- T_{\ell}\leq {\frac{\sum_{j\geq c_1{\ell}}\mathbb{P}\left[ S_1\in 
\mathcal{C}, \ldots, S_{j} \in \mathcal{C}, S_{b_{\ell}}\notin B_{b_{\ell}},
S_j=g_{\ell}\right] }{\sum_{j\geq c_1{\ell}}\mathbb{P}\left[ S_1\in \mathcal{%
C}, \ldots, S_j \in \mathcal{C}, S_j=g_{\ell}\right].  }}
\end{equation*}
Write $N_{\ell}$ and $D_{\ell}$ for the numerator and the denominator of the
previous fraction. We have 
\begin{equation*}
D_{\ell}\geq \mathbb{P}\left[ S_1\in \mathcal{C}, \ldots, S_{\ell} \in 
\mathcal{C}, S_{\ell}=g_{\ell}\right]. 
\end{equation*}
By Lemma \ref{lem-marc}, there exists $c^{\prime }>0$ such that $%
D_{\ell}\geq \exp(-c^{\prime }{\ell}^\alpha). $ Since the support of $\mu$
is bounded, there exists $c_2>0$ such that $|S_{b_{\ell}}|\leq c_2 {\ell}%
^\beta$. This gives 
\begin{multline*}
N_{\ell} \leq \sum_{j\geq c_1{\ell}}\mathbb{P}\left[ S_{b_{\ell}}\notin
B_{b_{\ell}}, S_j=g_{\ell}\right] \leq \sum_{j\geq c_1{\ell}}\mathbb{P}\left[
S_{b_{\ell}}\notin B_{b_{\ell}}\right]\max_{|x|\leq c_2 {\ell}^\beta} 
\mathbb{P}[x+S_{j-b_{\ell}}=g_{\ell}] \\
\leq \mathbb{P}\left[ S_{b_{\ell}}\notin B_{b_{\ell}}\right] \sum_{j\geq c_1{%
\ell}}\sum_{|x|\leq c_2 {\ell}^\beta} \mathbb{P}[x+S_{j-b_{\ell}}=g_{\ell}]%
\leq \mathbb{P}\left[ S_{b_{\ell}}\notin B_{b_{\ell}}\right]\sum_{|x|\leq
c_2 {\ell}^\beta} U(g_{\ell}-x).
\end{multline*}
The cardinality of $B(0,c_2{\ell}^\beta)\cap G$ is $O({\ell}^{n\beta})$ and
the function $U$ is uniformly bounded on $G$. This gives 
\begin{equation*}
N_{\ell}\preceq \mathbb{P}\left[ S_{b_{\ell}}\notin B_{b_{\ell}}\right]
\times {\ell}^{n\beta}. 
\end{equation*}
By the large deviations inequality, there exists $c_3>0$ such that $\mathbb{P%
}\left[ S_{b_{\ell}}\notin B_{b_{\ell}}\right]\leq\exp(-c_3b_{\ell})$.
Finally the polynomial term is absorbed by the exponential term and we
obtain 
\begin{equation*}
N_{\ell}\preceq \exp(-c_4{\ell}^\beta). 
\end{equation*}
The estimates we have obtained on $D_{\ell}$ et $N_{\ell}$ clearly yields (%
\ref{T_n}).

Next, we prove that 
\begin{equation}  \label{T'_n}
T^{\prime }_{\ell}:= {\frac{\sum_{j\geq c_1{\ell}}\mathbb{P}\left[ S_1\in 
\mathcal{C}, \ldots, S_{j} \in \mathcal{C}, S_{b_{\ell}}\in B_{b_{\ell}},
S_j=g_{\ell}\right] }{\sum_{j\geq c_1{\ell}}\mathbb{P}\left[ S_1\in \mathcal{%
C}, \ldots, S_{b_{\ell}} \in \mathcal{C}, S_{b_{\ell}}\in B_{b_{\ell}},
S_j=g_{\ell}\right]}}\ \to 1\quad\text{when ${\ell}\to+\infty$}.
\end{equation}
We have 
\begin{equation*}
0\leq1- T^{\prime }_{\ell}\leq {\frac{\sum_{j\geq c_1{\ell}}\mathbb{P}\left[
S_1\in \mathcal{C}, \ldots, S_{b_{\ell}} \in \mathcal{C}, S_{b_{\ell}}\in
B_{b_{\ell}}, \exists k\in\{1,\ldots,j-b_{\ell}\} \,S_{b_{\ell}+k}\notin%
\mathcal{C }, S_j=g_{\ell}\right] }{\sum_{j\geq c_1{\ell}}\mathbb{P}\left[
S_1\in \mathcal{C}, \ldots, S_{b_{\ell}} \in \mathcal{C}, S_{b_{\ell}}\in
B_{b_{\ell}}, S_j=g_{\ell}\right]  }}\ .
\end{equation*}
Write $N^{\prime }_{\ell}$ and $D^{\prime }_{\ell}$ for the numerator and the
denominator of the previous fraction.  
Let $\varepsilon>0$ be as in the proof of Theorem \ref{tll-q}.
Since  $[S_k\notin \mathcal{C}] \subset [\left\|S_{k}- km\right\|\geq
k\varepsilon]$, one gets
\begin{align*}
N^{\prime }_{\ell}&\leq\sum_{j\geq c_1{\ell}}\mathbb{P}\left[ \exists
k\in\{1,\ldots,j-b_{\ell}\} \,S_{b_{\ell}+k}\notin\mathcal{C }, S_j=g_{\ell}%
\right] \\
&\leq \sum_{k\geq1}\ \sum_{j\geq \max(c_1{\ell},k+b_{\ell})}\mathbb{P}\left[
S_{b_{\ell}+k}\notin\mathcal{C }, S_j=g_{\ell}\right] \\
&\leq \sum_{k\geq1}\ \sum_{j\geq k+b_{\ell}}\ \sum_{y\notin \mathcal{C}} 
\mathbb{P}\left[S_{b_{\ell}+k}=y\right]\mathbb{P}\left[S_{j-b_{\ell}-k}=g_{%
\ell}-y\right] \\
&\leq \sum_{k\geq1}\ \sum_{y\notin \mathcal{C}} \mathbb{P}\left[%
S_{b_{\ell}+k}=y\right]\ U(g_{\ell}-y) \\
&\leq \sum_{k\geq1}\ \sum_{y\notin \mathcal{C}} \mathbb{P}\left[%
S_{b_{\ell}+k}=y\right]\ \max_{g\in G}U(g) \\
&\leq\sum_{k\geq b_{\ell}}\mathbb{P}\left[S_k\notin \mathcal{C }\right]\
\max_{g\in G}U(g) \\
& \leq\sum_{k\geq b_{\ell}}\mathbb{P}\left[\|S_k-km\|\geq k\varepsilon %
\right]\ \max_{g\in G}U(g).
\end{align*}

Using the large deviations estimate, we obtain as in the proof of Theorem %
\ref{tll-q}   $$N^{\prime }_{\ell}\preceq \exp(-c{\ell}^\beta).$$ 

Let us now look at the denominator $D'_\ell$ ; one gets
\begin{equation*}
D^{\prime }_{\ell}\geq\sum_{j\geq c_1{\ell}}\mathbb{P}\left[ S_1\in \mathcal{%
C}, \ldots, S_j \in \mathcal{C}, S_{b_{\ell}}\in B_{b_{\ell}}, S_j=g_{\ell}%
\right]. 
\end{equation*}
Since $T_{\ell}\to1$, we obtain for large enough ${\ell}$,
\begin{equation*}
D^{\prime }_{\ell}\geq\frac12\sum_{j\geq c_1{\ell}}\mathbb{P}\left[ S_1\in 
\mathcal{C}, \ldots, S_j \in \mathcal{C}, S_j=g_{\ell}\right]=D_{\ell}. 
\end{equation*}
So  $D^{\prime }_{\ell}\geq\frac12 \exp(-c^{\prime }{\ell}^\alpha)$
which leads to  (\ref{T'_n}) ; the convergence (\ref{R_n}) follows, combining  (\ref{T_n})
and (\ref{T'_n}). 

The limit (\ref{R_n}) also holds when we replace $g_{\ell}$
by $g_{\ell} +h_{\ell}$. Therefore, to prove the theorem, it suffices to
check that 
\begin{equation*}
{\frac{\sum_{j\geq c_1{\ell}}\mathbb{P}\left[ S_1\in \mathcal{C}, \ldots,
S_{b_{\ell}} \in \mathcal{C}, S_{b_{\ell}}\in B_{b_{\ell}},
S_j=g_{\ell}+h_{\ell}\right] }{\sum_{j\geq c_1{\ell}}\mathbb{P}\left[ S_1\in 
\mathcal{C}, \ldots, S_{b_{\ell}} \in \mathcal{C}, S_{b_{\ell}}\in
B_{b_{\ell}}, S_j=g_{\ell}\right]  }}\ \to 1. 
\end{equation*}
We write the denominator of this fraction as
\begin{equation*}
\sum_{j\geq c_1{\ell}}\ \sum_{y\in B_{b_{\ell}}}\mathbb{P}\left[S_1\in 
\mathcal{C}, \ldots, S_{b_{\ell}} \in \mathcal{C}, S_{b_{\ell}}=y\right]%
\times \mathbb{P}\left[S_{j-b_{\ell}}=g_{\ell}-y\right], 
\end{equation*}
which is equal to 
\begin{equation*}
\sum_{y\in B_{b_{\ell}}}\mathbb{P}\left[S_1\in \mathcal{C}, \ldots,
S_{b_{\ell}} \in \mathcal{C}, S_{b_{\ell}}=y\right]\times U(g_{\ell}-y). 
\end{equation*}
This is sufficient to  conclude since, by(\ref{ren-equ}), we have 
$
U(g_{\ell}+h_{\ell}-y)\sim U(g_{\ell}-y), 
$
uniformly in $y\in B_{b_{\ell}}$.
\end{proof}

\section{Crystals and random walks}

\label{Sec_crystl_RW}

\subsection{Brief review on crystals}

\label{subsec_crys}We now recall some basics on Kashiwara crystals and
quantum groups.\ For a complete review, we refer to \cite{HK} and \cite%
{Kashi}.\ The quantum group $U_{q}(\mathfrak{g)}$ is a $q$-deformation of
the enveloping Lie algebra $U(\mathfrak{g}).$ To each dominant weight $%
\lambda\in P_{+}$ corresponds a unique (up to isomorphism) irreducible $%
U_{q}(\mathfrak{g})$-module $V_{q}(\lambda).\;$The representation theory of
the finite dimensional $U_{q}(\mathfrak{g})$-modules is essentially parallel
to that of $U(\mathfrak{g}).\;$In particular, any tensor product $%
V_{q}(\lambda)\otimes V_{q}(\mu)$ decomposes into irreducible components.\
The outer multiplicities so obtained are the same as those appearing in the
decomposition of $V(\lambda)\otimes V(\mu).$ Similarly, there exists a
relevant notion of weight subspaces in $V_{q}(\lambda)$ and for any $%
\beta\in P,$ one gets $\dim V_{q}(\lambda)_{\beta}=\dim
V(\lambda)_{\beta}=K_{\lambda,\beta}.$

To each irreducible module $V_{q}(\lambda)$ is associated its Kashiwara
crystal $B(\lambda).$ Formally $B(\lambda):=L(\lambda)/qL(\lambda)$ where $%
L(\lambda)$ is a particular lattice in $V_{q}(\lambda)$ over the ring 
\begin{equation*}
A(q):=\{F\in\mathbb{C}(q)\text{ without pole at }q=0\}.
\end{equation*}
It was proved by Kashiwara that $B(\lambda)$ has the structure of a colored
and oriented graph. This graph encodes many informations on the
representation $V_{q}(\lambda)$ (and thus also on $V(\lambda)$). In
particular, the crystal $B(\lambda)$ contains $\dim V_{q}(\lambda)=\dim
V(\lambda)$ vertices. Its arrows are labelled by the simple roots $%
\alpha_{i}\in\Delta_{+}$.\ The graph structure is obtained from the
Kashiwara operators $\tilde{f}_{i}$ and $\tilde{e}_{i}$, $i\in I$, which are
renormalizations of the action of the Chevalley generators $e_{i},f_{i}$ of $%
U_{q}(\mathfrak{g)}$.\ More precisely, we have an arrow $a\overset{i}{%
\rightarrow}b$ when $b=\tilde{f}_{i}(a)$ or equivalently $a=\tilde{e}%
_{i}(b). $ When there is no arrow $\overset{i}{\rightarrow}$ starting from $a$ (resp.
ending at $b$), we write $\tilde{f}_{i}(a)=0$ (resp. $\tilde{e}%
_{i}(b)=0$).

The notion of crystal can be extended to a category $\mathcal{O}_{\mathrm{int%
}}$ of $U_{q}(\mathfrak{g})$ modules containing the irreducible modules $%
V_{q}(\lambda)$ and stable by tensorization. The crystal associated to the
module $M\in\mathcal{O}_{\mathrm{int}}$ is unique up to isomorphism: given $%
B $ and $B^{\prime}$ two crystals associated to $M,$ there exists a
bijection $\psi:B\rightarrow B^{\prime}$ which commutes with the Kashiwara
operators $\tilde{f}_{i}$ and $\tilde{e}_{i}.$

Given any $b\in B$ and $i\in I$, we set $\varepsilon_{i}(b)=\max\{k\mid 
\tilde{e}_{i}^{k}(b)\neq0\}$ and $\varphi_{i}(b)=\max\{k\mid\tilde{f}%
_{i}^{k}(b)\neq0\}.$ The weight of the vertex $b\in B$ is then defined by 
\begin{equation}
\mathrm{wt}(b)=\sum_{i\in I}\mathrm{wt}(b)_{i}\,\omega_{i}\in P\text{ where }%
\mathrm{wt}(b)_{i}=\varphi_{i}(b)-\varepsilon_{i}(b)  \label{defWeightCrys}
\end{equation}
One can then prove that $\mathrm{wt}(\tilde{f}_{i}(b))=\mathrm{wt}%
(b)-\alpha_{i}$ for any $i\in I$ and any $b\in B$ such that $\tilde{f}%
_{i}(b)\neq0.\;$For any $\beta\in P,$ the dimension $\dim M_{\beta}$  of the
weight space $\beta$ in $M$  is the cardinality of the set $B_{\beta}$ of
vertices of weight $\beta$ in  the crystal $B$ associated to $M$. A vertex $%
b\in B$ is said to be of highest weight when $\varepsilon_{i}(b)=0$ for any $%
i\in I.\;$In that case, we have immediately that $\mathrm{wt}(b)\in P_{+}.$
Write $\mathrm{HW}(B)$ for the set of highest weight vertices in $B$.\ The
elements of $\mathrm{HW}(B)$ are in one-to-one correspondence with the
connected components of the crystal $B$. In particular the crystal $%
B(\lambda)$ is connected with a unique highest weight vertex of weight $%
\lambda$. For any $b\in B,$ we denote by $B(b)$ the connected component of $%
B $ containing $b$ and by $\mathrm{hw}(b)$ the highest weight vertex of $B(b)
$.

The two following properties of crystals will be essential for our purpose.

\begin{theorem}
\label{TH_Ka}Consider $M\in \mathcal{O}_{\mathrm{int}}$ and $B$ its crystal
graph.

\begin{enumerate}
\item The decomposition of the $U_{q}(\mathfrak{g})$-module $M$ in
irreducible components is given be the decomposition of $B$ in connected
components. More precisely, we have 
\begin{equation*}
M\simeq \bigoplus_{b\in \mathrm{HW}(B)}V_{q}(\mathrm{wt}(b)).
\end{equation*}

\item Consider $\lambda ,\mu \in P_{+}$ and $B(\lambda ),B(\mu )$ the
crystals associated to $V_{q}(\lambda )$ and $V_{q}(\mu ).$ The crystal
associated to $V_{q}(\lambda )\otimes V_{q}(\mu )$ is the crystal $B(\lambda
)\otimes B(\mu )$ whose set of vertices is the direct product of the sets of
vertices of $B(\lambda )$ and $B(\mu )$ and whose crystal structure is given
by the following rules 
\begin{equation}
\tilde{e}_{i}(u\otimes v)=\left\{ 
\begin{array}{l}
u\otimes \tilde{e}_{i}(v)\text{ if }\varepsilon _{i}(v)>\varphi _{i}(u) \\ 
\tilde{e}_{i}(u)\otimes v\text{ if }\varepsilon _{i}(v)\leq \varphi _{i}(u)%
\end{array}
\right. \text{ and }\tilde{f}_{i}(u\otimes v)=\left\{ 
\begin{array}{l}
\tilde{f}_{i}(u)\otimes v\text{ if }\varphi _{i}(u)>\varepsilon _{i}(v) \\ 
u\otimes \tilde{f}_{i}(v)\text{ if }\varphi _{i}(u)\leq \varepsilon _{i}(v)%
\end{array}
\right. .  \label{tens_crys}
\end{equation}
We thus have 
\begin{equation*}
\left\{ 
\begin{array}{l}
\varphi _{i}(u\otimes v)=\varphi _{i}(v)+\max \{\varphi _{i}(u)-\varepsilon
_{i}(v),0\}, \\ 
\varepsilon _{i}(u\otimes v)=\varepsilon _{i}(u)+\max \{\varepsilon
_{i}(v)-\varphi _{i}(u),0\}.%
\end{array}
\right.
\end{equation*}
In particular $u\otimes v\in \mathrm{HW}(B(\lambda )\otimes B(\mu ))$ if and
only if $u\in \mathrm{HW}(B(\lambda ))$ and $\varepsilon _{i}(v)\leq \varphi
_{i}(v)$ for any $i\in I.$
\end{enumerate}
\end{theorem}

\begin{example}
\label{ex_crys}In type $C_{3}$, the crystal corresponding to the minuscule
weight $\omega _{1}$ is 
\begin{equation*}
B(\omega _{1}):1\overset{1}{\rightarrow }2\overset{2}{\rightarrow }3\overset{%
3}{\rightarrow }\overline{3}\overset{2}{\rightarrow }\overline{2}\overset{1}{%
\rightarrow }\overline{1}.
\end{equation*}
The tensor power $B(\omega _{1})^{\otimes 2}$%
\begin{equation*}
\begin{tabular}{lllllllllll}
1$\otimes $1 & $\overset{1}{\rightarrow }$ & 2$\otimes $1 & $\overset{2}{%
\rightarrow }$ & 3$\otimes $1 & $\overset{3}{\rightarrow }$ & \={3}$\otimes $%
1 & $\overset{2}{\rightarrow }$ & \={2}$\otimes $1 & $\overset{1}{%
\rightarrow }$ & \={1}$\otimes $1 \\ 
&  & \ {\scriptsize 1}$\downarrow $ &  & \ {\scriptsize 1}$\downarrow $ &  & 
\ {\scriptsize 1}$\downarrow $ &  &  &  & \ {\scriptsize 1}$\downarrow $ \\ 
1$\otimes $2 &  & 2$\otimes $2 & $\overset{2}{\rightarrow }$ & 3$\otimes $2
& $\overset{3}{\rightarrow }$ & \={3}$\otimes $2 & $\overset{2}{\rightarrow }
$ & \={2}$\otimes $2 &  & \={1}$\otimes $2 \\ 
\ {\scriptsize 2}$\downarrow $ &  &  &  & \ {\scriptsize 2}$\downarrow $ & 
&  &  & \ {\scriptsize 2}$\downarrow $ &  & \ {\scriptsize 2}$\downarrow $
\\ 
1$\otimes $3 & $\overset{1}{\rightarrow }$ & 2$\otimes $3 &  & 3$\otimes $3
& $\overset{3}{\rightarrow }$ & \={3}$\otimes $3 &  & \={2}$\otimes $3 & $%
\overset{1}{\rightarrow }$ & \={1}$\otimes $3 \\ 
\ {\scriptsize 3}$\downarrow $ &  & \ {\scriptsize 3}$\downarrow $ &  &  & 
& \ {\scriptsize 3}$\downarrow $ &  & \ {\scriptsize 3}$\downarrow $ &  & \ 
{\scriptsize 3}$\downarrow $ \\ 
1$\otimes $\={3} & $\overset{1}{\rightarrow }$ & 2$\otimes $\={3} & $\overset%
{2}{\rightarrow }$ & 3$\otimes $\={3} &  & \={3}$\otimes $\={3} & $\overset{2%
}{\rightarrow }$ & \={2}$\otimes $\={3} & $\overset{1}{\rightarrow }$ & \={1}%
$\otimes $\={3} \\ 
\ {\scriptsize 2}$\downarrow $ &  &  &  & \ {\scriptsize 2}$\downarrow $ & 
&  &  & \ {\scriptsize 2}$\downarrow $ &  & \ {\scriptsize 2}$\downarrow $
\\ 
1$\otimes $\={2} & $\overset{1}{\rightarrow }$ & 2$\otimes $\={2} &  & 3$%
\otimes $\={2} & $\overset{3}{\rightarrow }$ & \={3}$\otimes $\={2} &  & \={2%
}$\otimes $\={2} & $\overset{1}{\rightarrow }$ & \={1}$\otimes $\={2} \\ 
&  & \ {\scriptsize 1}$\downarrow $ &  & \ {\scriptsize 1}$\downarrow $ &  & 
\ {\scriptsize 1}$\downarrow $ &  &  &  & \ {\scriptsize 1}$\downarrow $ \\ 
1$\otimes $\={1} &  & 2$\otimes $\={1} & $\overset{2}{\rightarrow }$ & 3$%
\otimes $\={1} & $\overset{3}{\rightarrow }$ & \={3}$\otimes $\={1} & $%
\overset{2}{\rightarrow }$ & \={2}$\otimes $\={1} &  & \={1}$\otimes $\={1}%
\end{tabular}%
\end{equation*}
admits three connected components with highest weight vertices 1$\otimes $1,1%
$\otimes $2 and 1$\otimes $\={1} of weights $2\omega _{1},\omega _{2}$ and $%
0.$ This gives the decomposition 
\begin{equation*}
V_{q}(\omega _{1})^{\otimes 2}\simeq V_{q}(2\omega _{1})\oplus V_{q}(\omega
_{2})\oplus V_{q}(0).
\end{equation*}
\end{example}

\noindent\textbf{Remark:} We have seen that $\dim M_{\beta}$ is equal to the
number of vertices of weight $\beta$ in $B$ the crystal of $M$. Similarly,
the number of highest weight vertices in $B$ with weight $\lambda$ gives the
multiplicity of $V(\lambda)$ in the decomposition of $M$ into its
irreducible components.

\bigskip

Consider $C$ and $C^{\prime}$ two connected components of the crystal $B$.
The components $C$ and $C^{\prime}$ are isomorphic when there exists a
bijection $\phi_{C,C^{\prime}}$ from $C$ to $C^{\prime}$ which commutes with
the action of the Kashiwara operators, that is $b_{1}\overset{i}{\rightarrow}%
b_{2}$ in $C$ if and only if $\phi_{C,C^{\prime}}(b_{1})\overset{i}{%
\rightarrow}\phi_{C,C^{\prime}}(b_{2})$ in $C^{\prime}$. In that case, the
isomorphism $\phi_{C,C^{\prime}}$ is unique since it must send the highest
weight vertex of $C$ on the highest weight vertex of $C^{\prime}$.

The following lemma is a straightforward consequence of Theorem \ref{TH_Ka}.

\begin{lemma}
\label{Lem_HWV}Assume $b^{0}=a_{1}^{0}\otimes \cdots \otimes a_{\ell }^{0}$
is a highest weight vertex of $B(\delta )^{\otimes \ell }$.

\begin{enumerate}
\item For any $k=1,\ldots ,\ell ,$ the vertex $b^{(k),0}:=a_{1}^{0}\otimes \cdots
\otimes a_{k}^{0}$ of $B(\delta )^{\otimes \ell }$ is a highest weight vertex of $B(\delta )^{\otimes k}$
and $\varepsilon _{i}(a_{k+1}^{0}\otimes \cdots \otimes a_{\ell }^{0})\leq
\varphi _{i}(b^{(k),0})$ for any $i\in I$.

\item Consider $b=a_{1}\otimes \cdots \otimes a_{\ell }$ a vertex of $%
B(b^{0}).$ Then, for any $k=1,\ldots ,\ell ,$ the vertex $b^{(k)}:=a_{1}\otimes \cdots
\otimes a_{k}$ belongs to $B(b^{(k),0})$.
\end{enumerate}
\end{lemma}

\bigskip

We will need the following proposition in Section \ref{Sec_restric}.

\begin{proposition}
\label{Prop_dec_skew}Consider $\lambda ,\mu \in P_{+}$. With the notation of 
\S\ \ref{subsec_hwmod}, the following properties hold.

\begin{enumerate}
\item $f_{\lambda /\mu }^{\ell }=\sum_{\kappa \in P_{+}}f_{\kappa }^{\ell
}m_{\kappa ,\mu }^{\lambda }$ for any $\ell \geq 1.$

\item Assume $(\lambda ^{(a)})_a$ is a sequence of weights of the form $\lambda
^{(a)}=am+o(a)$ with $m\in C$ and consider $\kappa \in P_{+}$. Then, for $a$
sufficiently large, the weight $\lambda ^{(a)}$ belongs to $P_{+}$ and $m_{\kappa ,\mu }^{\lambda
^{(a)}}=K_{\mu ,\lambda ^{(a)}-\kappa }$. Therefore 
\begin{equation}
f_{\lambda ^{(a)}/\mu }^{\ell }=\sum_{\kappa \in P_{+}}f_{\kappa }^{\ell
}K_{\mu ,\lambda ^{(a)}-\kappa }=\sum_{\gamma \in P}f_{\lambda ^{(a)}-\gamma
}^{\ell }K_{\mu ,\gamma }  \label{dec_skew_coef}
\end{equation}
for any $\ell \geq 1.$
\end{enumerate}
\end{proposition}

\begin{proof}
To prove 1, write 
\begin{equation*}
s_{\mu }(s_{\delta })^{\otimes \ell }=\sum_{\kappa }f_{\kappa }^{\ell
}s_{\kappa }s_{\mu }=\sum_{\kappa }\sum_{\lambda }f_{\kappa }^{\ell
}m_{\kappa ,\mu }^{\lambda }s_{\lambda }=\sum_{\lambda }f_{\lambda /\mu
}^{\ell }s_{\lambda }
\end{equation*}
where all the sums run over $P_{+}.$ The assertion immediately follows by
comparing the two last expressions.

For $a$ sufficiently large, we must have $\lambda ^{(a)}\in P_{+}$ for $%
\lambda ^{(a)}\sim am$ and $m\in C$. By Lemma \ref{Lem_HWV}, $m_{\kappa ,\mu
}^{\lambda ^{(a)}}$ is equal to the number of vertices of weight $\lambda
^{(a)}$ of the form $b_{\kappa }\otimes b$ where $b_{\kappa }$ is the
highest weight vertex of $B(\kappa )$ and $b\in B(\mu )$ verifies $%
\varepsilon _{i}(b)\leq \varphi _{i}(b_{\kappa })$ for any $i\in I$. In
particular $b$ has weight $\lambda ^{(a)}-\kappa $. Thus $m_{\kappa ,\mu
}^{\lambda ^{(a)}}\leq K_{\mu ,\lambda ^{(a)}-\kappa }.$

Now assume $b\in B(\mu )$ has weight $K_{\mu ,\lambda ^{(a)}-\kappa }$ so
that the vertex $b_{\kappa }\otimes b\in B(\kappa )\otimes B(\mu )$ has
dominant weight $\lambda ^{(a)}$. We have by 2 of Theorem \ref{TH_Ka} 
\begin{equation*}
\mathrm{wt}(b_{\kappa }\otimes b)_{i}=\varphi _{i}(b_{\kappa }\otimes
b)-\varepsilon _{i}(b_{\kappa }\otimes b)\leq \varphi _{i}(b_{\kappa
}\otimes b).
\end{equation*}
Since $\lambda ^{(a)}=am+o(a)$, the weight  $\mathrm{wt}(b_{\kappa
}\otimes b)_{i}$ tends to infinity with $a$, for any $i \in I$.\ Since $b\in B(\mu )$ and $\mu $
is fixed, 
\begin{equation*}
\varphi _{i}(b_{k}\otimes b)=\varphi _{i}(b)+\max \{\varphi _{i}(b_{\kappa
})-\varepsilon _{i}(b)\}>\varphi _{i}(b)
\end{equation*}
for $a$ sufficiently large. This means that $\varepsilon _{i}(b)\leq \varphi
_{i}(b_{\kappa })$ for such a $a$. So $b_{\kappa }\otimes b$ is a highest
weight vertex with dominant weight $\lambda ^{(a)}$. Therefore $m_{\kappa
,\mu }^{\lambda ^{(a)}}\geq K_{\mu ,\lambda ^{(a)}-\kappa }.$
\end{proof}

\subsection{Paths in weight lattices and crystals}

\label{subsec_paths_crys}

Let $\delta$ be a dominant weight and $B(\delta)$ the crystal of $%
V_{q}(\delta)$ $^($\footnote{%
For the sake of simplicity, we only consider in this section paths obtained
from tensor powers of irreducible modules. This hypothesis will be relaxed
in Section \ref{Sec_miscell}.}$^)$.$\;$By Theorem \ref{TH_Ka} and (\ref{def_f}),
we derive for any $\ell\geq1$ the decomposition of $B(\delta)^{\otimes\ell}$
in its irreducible components. 
\begin{equation*}
B(\delta)^{\otimes\ell}\simeq\bigsqcup_{\lambda\in P_{+}}B(\lambda)^{\oplus
f_{\lambda,\delta}^{\ell}}.
\end{equation*}
Let $\mathfrak{P}$ be the map defined on $B(\delta)^{\otimes\ell}$ which
associates to each vertex $b$ the highest weight vertex $\mathfrak{P}(b)$ of 
$B(b)$. The map $\mathfrak{P}$ can be interpreted as a Pitmann transform on
paths in the weight lattice following ideas essentially analogue to those
used in \cite{BBOC1}.

\bigskip

\noindent\textbf{Remark: }In \cite{OC1}, the transformation $\mathfrak{P}$
was computed by using Knuth insertion algorithm on semistandard tableaux.
One can prove that these semistandard tableaux yield simple parametrizations
of the crystals $B(\delta)$. There exist analogous notions of tableaux for
types $B_{n},C_{n},D_{n}$ and $G_{2}$ which similarly give a simple
parametrization of $B(\delta)$ for any dominant weight $\delta$. They were
introduced by Kashiwara and Nakashima in \cite{KN} for the classical types
and by Kang and Misra for type $G_{2}$ \cite{KM}.\ In \cite{Lec1}, \cite%
{lec2}, \cite{lec3} one describes combinatorial procedures on these tableaux
generalizing Knuth insertion algorithm. They also permit to compute the
transformation $\mathfrak{P}$ similarly to the original paper by O'Connell.\
The computation is then more efficient for it avoids the determination of a
path from $b$ to the highest weight vertex $\mathfrak{P}(b)$ of $B(b)$. We
do not pursue in this direction and refer to \cite{lec3} for a simple
exposition of these procedures in types $B_{n},C_{n},D_{n}$ and $G_{2}$.

\bigskip

To each vertex $b=a_{1}\otimes\cdots\otimes a_{\ell}$ in $B(\delta
)^{\otimes\ell}$ naturally corresponds a path in $Z(\delta,\ell),$ namely
the path $(\mu^{(1)},\ldots,\mu^{(\ell)})$ where for any $k=1,\ldots,\ell,$
we have $\mu^{(k)}=\sum_{j=1}^{k}\mathrm{wt}(a_{j}).\;$We shall denote by $%
\mathcal{R}$ the map 
\begin{equation}
\mathcal{R}:\left\{ 
\begin{array}{c}
B(\delta)^{\otimes\ell}\rightarrow Z(\delta,\ell) \\ 
b\mapsto(\mu^{(1)},\ldots,\mu^{(\ell)}).%
\end{array}
\right.  \label{defWT}
\end{equation}
The map $\mathcal{R}$ is surjective by definition of $Z(\delta,\ell)$. Write 
$\mathrm{HW}(B(\delta)^{\otimes\ell})$ for the set of highest weight
vertices in $B(\delta)^{\otimes\ell}.$ Then the image of the restriction $%
\mathcal{R}_{\mathrm{HW}}$ of $\mathcal{R}$ to $\mathrm{HW}%
(B(\delta)^{\otimes\ell})$ is a subset of $Z^{+}(\delta,\ell)$.\ Indeed, by
Lemma \ref{Lem_HWV}, if $b=a_{1}\otimes\cdots\otimes a_{\ell}$ belongs to $%
\mathrm{HW}(B(\delta )^{\otimes\ell})$, then $a_{1}\otimes\cdots\otimes
a_{k} $ belongs to $\mathrm{HW}(B(\delta)^{\otimes k})$ for any $%
k=1,\ldots,\ell$.

When each weight space in $V_{q}(\delta )$ has dimension $1$, the map $%
\mathcal{R}$ is bijective since the steps in the path of $Z(\delta ,\ell )$
are in one-to-one correspondence with the weights of $B(\delta ).$ We then
identify the paths $Z(\delta ,\ell )$ with the vertices of $B(\delta
)^{\otimes \ell }.\;$More precisely the path $(0,\mu ^{(1)},\ldots ,\mu
^{(\ell )})$ is identified with the vertex $b=a_{1}\otimes \cdots \otimes
a_{\ell }\in B(\delta )^{\otimes \ell }$ where for any $k=1,\ldots ,\ell ,$
the vertex $a_{k}$ is the unique one with weight $\mu ^{(k)}-\mu ^{(k-1)}$
in $B(\delta )$. This situation happens in particular when $\delta $ is a
minuscule weight.\ Indeed, since each weight in $V_{q}(\delta )$ belongs to
the orbit of $\delta $, each weight space has dimension $1$. We have the
following crucial property.

\begin{proposition}
\label{Prop_minuscule}The restriction $\mathcal{R}_{\mathrm{HW}}:\mathrm{HW}%
(B(\delta )^{\otimes \ell })\rightarrow Z^{+}(\delta ,\ell )$ is a
one-to-one correspondence for any $\ell \geq 1$ if and only if $\delta $ is
minuscule.
\end{proposition}

\begin{proof}
Assume $\delta $ is minuscule.\ We can then identify the paths in $%
Z^{+}(\delta ,\ell )$ with vertices of $B(\delta )^{\otimes \ell }$ as
explained above and $\mathcal{R}_{\mathrm{HW}}$ is injective.\ We thus have
to show that any vertex $b=a_{1}\otimes \cdots \otimes a_{\ell }\in B(\delta
)^{\otimes \ell }$ such that $\mathrm{wt}(a_{1}\otimes \cdots \otimes a_{k})$
is dominant for any $k=1,\ldots ,\ell $ is a highest weight vertex. We
proceed by induction on $\ell $. For $\ell =1,$ the highest weight vertex of 
$B(\delta )$ is $a_{1}$ for $B(\delta )$ contains a unique vertex of
dominant weight ($\delta $ being minuscule). Assume $b=a_{1}\otimes \cdots
\otimes a_{\ell }\in B(\delta )^{\otimes \ell }$ is such that $\mathrm{wt}%
(a_{1}\otimes \cdots \otimes a_{k})$ is dominant for any $k=1,\ldots ,\ell $%
. By the induction hypothesis, $b^{\flat }=a_{1}\otimes \cdots \otimes
a_{\ell -1}$ is of highest weight.\ According to assertion $(2)$ of Theorem %
\ref{TH_Ka}, it suffices to show that $\varepsilon _{i}(a_{\ell })\leq
\varphi _{i}(b^{\flat })$ for any $i\in I.$ But $\varepsilon _{i}(a_{\ell })\in \{0,1\} $ since $\delta $ is
minuscule. One can therefore
assume that $\varepsilon _{i}(a_{\ell })=1$.\ In this case, we have $\varphi
_{i}(a_{\ell })=0$ because $\delta $ is minuscule.\ The condition $\mathrm{wt%
}(b)\in P_{+}$ implies that $\varphi _{i}(b)\geq \varepsilon _{i}(b)$ for
any $i\in I$. Moreover, the condition $\mathrm{wt}(b)=\mathrm{wt}(b^{\flat
})+\mathrm{wt}(a_{\ell })$ can be written 
\begin{equation*}
\varphi _{i}(b)-\varepsilon _{i}(b)=\varphi _{i}(b^{\flat })-\varepsilon
_{i}(b^{\flat })+\varphi _{i}(a_{\ell })-\varepsilon _{i}(a_{\ell })=\varphi
_{i}(b^{\flat })-\varepsilon _{i}(b^{\flat })-1\text{ for any }i\in I.
\end{equation*}
Thus $\varphi _{i}(b^{\flat })\geq \varepsilon _{i}(b^{\flat })+1\geq
\varepsilon _{i}(a_{\ell })$ as expected.

Conversely, assume $\mathcal{R}_{\mathrm{HW}}:\mathrm{HW}(B(\delta
)^{\otimes \ell })\rightarrow $ $Z^{+}(\delta ,\ell )$ is a one-to-one
correspondence for any $\ell \geq 1$. If $\delta $ is not minuscule, there
exists a weight $\mu $ such that $\dim V_{q}(\delta )_{\mu }>0$ which is not
in the orbit of $\delta $ under the action of the Weyl group $W$.\ The orbit
of $\mu $ under $W$ intersects the cone of dominant weights $P_{+}$.
Therefore, we can assume $\mu \in P_{+}$.\ The crystal $B(\delta )$ thus contains a
vertex $b_{\mu }$ of dominant weight $\mu $ which is not a highest weight
vertex since $\mu \neq \delta .\;$Then the path in $P$ from $0$ to $\mu $
belongs to $Z^{+}(\delta ,1)$ but is not in the image of $\mathrm{HW}%
(B(\delta ))$ by $\mathcal{R}_{\mathrm{HW}}$ which contradicts our
assumption.
\end{proof}

\bigskip

\noindent\textbf{Remark:} There exist dominant weights $\delta$ which are
not minuscule but such that each weight spaces of $V(\delta)$ has dimension $%
1$. This is notably the case for $\delta=k\omega_{1},k\in\mathbb{N}$ in type 
$A_{n-1}$ and $\delta=\omega_{1}$ in type $B_{n}$.\ In that case, according
to the previous proposition, the map $\mathcal{R}_{\mathrm{HW}}$ does not
provide a bijection between highest weight vertices and paths in $P_{+}$.

\bigskip

Let $Z^{+}(\delta,\ell,\lambda)$ be the subset of $Z^{+}(\delta,\ell)$ of
paths starting at $0$ and ending at $\lambda\in P_{+}.\;$Theorem \ref{TH_Ka}
and the previous Proposition immediately yield the following corollary
equating the number of paths in $Z^{+}(\delta,\ell,\lambda)$ to an outer
multiplicity in the tensor power $V(\delta)^{\otimes\ell}.$

\begin{corollary}
Assume $\delta $ is minuscule. We have $\mathrm{card}(Z^{+}(\delta ,\ell
,\lambda ))=f_{\lambda ,\delta }^{\ell }$ where the outer multiplicity $%
f_{\lambda ,\delta }^{\ell }$ is defined by (\ref{def_f}).
\end{corollary}

As far as we are aware this equality was first established by Grabiber and
Mayard in \cite{GrMa} by using some previous works of Proctor \cite{Pr}.\
Crystal basis theory permits to derive a very short proof of this identity.

\subsection{Probability distribution on $B(\protect\delta)$}

\label{subsec_def_pa}

The aim of this paragraph is to endow $B(\delta)$ with a probability
distribution. We are going to associate to each vertex $a\in B(\delta)$ a
probability $p_{b}$ such that 
\begin{equation}
0<p_{a}<1\text{ and }\sum_{a\in B(\delta)}p_{a}=1.  \tag{I}  \label{CI}
\end{equation}
The probability distributions we consider are compatible with the weight
graduation, that is for any $a,a^{\prime}\in B(\delta)$, 
\begin{equation}
\mathrm{wt}(a)=\mathrm{wt}(a^{\prime})\Longrightarrow p_{a}=p_{a^{\prime}}. 
\tag{II}  \label{CII}
\end{equation}
We proceed as follows. Let $t=(t_{1},\ldots,t_{n})$ be a $n$-tuple of
positive reals (recall that $n$ is the rank of the root system considered).\
Let $a\in B(\delta)$.\ For any $\kappa\in Q_{+}$ such that $%
\kappa=\sum_{i=1}^{n}m_{i}\alpha_{i}$, we set $t^{[\kappa]}=t_{1}^{m_{1}}%
\cdots t_{n}^{m_{n}}$. Since $\mathrm{wt}(a)$ is a weight of $V_{q}(\delta)$%
, there exist nonnegative integers $m_{1},\ldots,m_{n}$ such that $\mathrm{wt%
}(a)=\delta-\sum_{i=1}^{n}m_{i}\alpha_{i}$.\ We can compute the sum $%
\Sigma_{\delta}(t)=\sum_{a\in B(\delta)}t^{[\mathrm{wt}(a)]}.$ The
probability distribution is then defined by 
\begin{equation}
p_{a}=\frac{t^{[\mathrm{wt}(a)]}}{\Sigma_{\delta}(t)}\text{ for any }a\in
B(\delta)\text{,}  \label{def-p(a)}
\end{equation}
so that we have $\sum_{a\in B(\delta)}p_{a}=1$.\ It is clear that conditions
(\ref{CI}) and (\ref{CII}) are satisfied. Note that if we have an arrow $a%
\overset{i}{\rightarrow}a^{\prime}$ in $B(\delta)$, then $%
p_{a^{\prime}}=p_{a}\times t_{i}$ since $\mathrm{wt}(a^{\prime})=\mathrm{wt}%
(a)-\alpha_{i}$.\ Also, the integer $m_{i}$ can easily be read on the
crystal $B(\delta).$ It corresponds to the number of arrows $\overset{i}{%
\rightarrow}$ in any path connected the highest weight vertex of $B(\delta)$
to $a$. Recall that these numbers are independent of the path considered
since we have the weight graduation $\mathrm{wt}$ on the crystals. The case
where $t_{i}=1$ for any $i=1,\ldots,n,$ corresponds to the uniform
probability distribution on $B(\delta)$ considered in \cite{BBOC1}.

\bigskip

\noindent\textbf{Remark:} Observe that our construction depends only on the
fixed $n$-tuple $t=(t_{1},\ldots,t_{n})$.

\bigskip 

When $t_{1}=\cdots =t_{n},$ one gets $p_{a}=\frac{1}{\dim V(\delta )}$ for any $a\in
B(\delta )$.\ We shall say in this case that the distribution is uniform. In
the rest of the paper, we assume the $n$-tuple $t=(t_{1},\ldots ,t_{n})$ is
fixed for each root system corresponding to a simple Lie algebra. We denote
by $x=(x_{1},\ldots ,x_{N})\in \mathbb{R}_{>0}^{N}$ any solution of the
algebraic system 
\begin{equation}
x^{\alpha _{i}}=t_{i}^{-1},i=1,\ldots ,n.  \label{system}
\end{equation}%
Let us briefly explain why such a solution necessary exists.\ Let $%
U=(u_{i,j})$ be the $N\times n$ matrix whose entries are determined by the
decompositions $\alpha _{i}=\sum_{j=1}^{N}u_{i,j}\varepsilon _{j},i=1,\ldots
,n.$ By taking the logarithm of each equation in (\ref{system}), we are lead
to solve the equation $^{t}UX=Y$ where $X=(\ln x_{1},\ldots ,\ln x_{N})$ and 
$Y=(-\ln t_{1},\ldots ,-\ln t_{n}).$ This linear system necessary admits
solutions because $U$ has rank $n$. This follows from the fact that the set
of simple roots $\{\alpha _{1},\ldots ,\alpha _{n}\}$ generates a $n$%
-dimensional subspace in $\mathbb{R}^{N}$. Each solution $X$ yields positive
solutions $x_{1},\ldots ,x_{N}$ of (\ref{system}). Observe that (\ref{system}%
) admits a unique solution when $N=n$. With the convention of \cite{Bour},
this happens for the root systems $B_{n},C_{n},D_{n},E_{8},F_{4}$ and $G_{2}$%
. Now consider $a\in B(\delta ).$ By the previous definition of $t_{a},$ we
then derive the relation $t_{a}=x^{-\delta }x^{\mathrm{wt}(a)}.$ This allows
to write 
\begin{equation*}
\Sigma _{\delta }(t)=x^{-\delta }\sum_{a\in B(\delta )}x^{\mathrm{wt}%
(a)}=x^{-\delta }s_{\delta }(x)\text{ and }p_{a}=\frac{x^{\mathrm{wt}(a)}}{%
s_{\delta }(x)}.
\end{equation*}

\begin{example}
\ \label{exa_vect}

\begin{enumerate}
\item Assume $\delta =\omega _{1}$ for type $A_{n-1}$.\ Then we have 
\begin{equation*}
B(\delta ):a_{1}\overset{1}{\rightarrow }a_{2}\overset{2}{\rightarrow }%
\cdots \overset{n-1}{\rightarrow }a_{n}.
\end{equation*}
The simple roots are the $\alpha _{i}=\varepsilon _{i}-\varepsilon
_{i+1},i=1,\ldots ,n-1.$ We thus have $t_{i}=\frac{x_{i+1}}{x_{i}}$ for any $%
i=1,\ldots ,n-1.\;$We obtain $x_{i}=x_{1}t_{1}\cdots t_{i-1}$ for any $%
i=2,\ldots ,n$. In that case, we have $N=(n-1)+1$ where $n-1$ is the rank of
the root system considered. We can normalize our probability distribution so
that $x_{1}+\cdots +x_{n}=1$. This gives $p_{i}=x_{i}$ for $i=1,\ldots ,n.$
The corresponding random walk in $P$ corresponds to the ballot problem where
each transition in the direction $\varepsilon _{i}$ has probability $p_{i}$.

\item Assume $\delta =\omega _{1}$ for type $C_{n}$.\ Then we have 
\begin{equation*}
B(\delta ):a_{1}\overset{1}{\rightarrow }a_{2}\overset{2}{\rightarrow }%
\cdots \overset{n-1}{\rightarrow }a_{n}\overset{1}{\rightarrow }a_{\overline{%
n}}\overset{n-1}{\rightarrow }\cdots \overset{2}{\rightarrow }a_{\overline{2}%
}\overset{1}{\rightarrow }a_{\overline{1}}.
\end{equation*}
The simple roots are the $\alpha _{i}=\varepsilon _{i}-\varepsilon
_{i+1},i=1,\ldots ,n-1$ and $\alpha _{n}=2\varepsilon _{n}.\;$We thus have $%
t_{i}=\frac{x_{i+1}}{x_{i}}$ for any $i=1,\ldots ,n-1$ and $%
t_{n}=x_{n}^{-2}.\;$This gives ${x_{i}=\frac{1}{t_{i}\cdots t_{n-1}\sqrt{%
t_{n}}}}$ for any $i=1,\ldots n-1$ and $x_{n}=\frac{1}{\sqrt{t_{n}}}.$  So 
$p_{a_{i}}=\frac{x_{i}}{s_{\omega _{1}}(x)}$ and $p_{a_{\overline{i}}}=\frac{%
x_{i}^{-1}}{s_{\omega _{1}}(x)}$ where $s_{\omega
_{1}}(x)=\sum_{i=1}^{n}(x_{i}+x_{i}^{-1}).$ We thus obtain a random walk in $%
\mathbb{Z}^{n},$ with steps from one point to one of the nearest neighbors,
and $p_{a_{i}}p_{a_{\overline{i}}}$ independent of $i$.
\end{enumerate}
\end{example}

\subsection{The random walk $\mathcal W$ in the weight lattice}

\subsubsection{From a probability measure on the crystal $B(\protect\delta)$
to a random walk on the weight lattice.}

\label{general}

A random walk in the weight lattice $P$ isomorphic to $\mathbb{Z}^{n}$ is characterized by the law of its
increments, that is a probability measure on $P$. As the map $b\mapsto%
\mathrm{wt}(b)$ sends the crystal $B(\delta)$ into $P$, any probability
measure $p$ on $B(\delta)$ can be pushed forward and defined a probability $%
\mathrm{wt}_{\ast}p$ on $P$.

The random walk with law $\mathrm{wt}_{\ast}p$ can be naturally realized as
a random process defined on the infinite product space ${\mathcal{B}}%
(\delta)^{\otimes\mathbb{N}}$ equipped with the product measure $p^{\otimes 
\mathbb{N}}$. Here is a ``concrete'' realization of this construction. The
infinite product space ${\mathcal{B}}(\delta)^{\otimes\mathbb{N}}$ is the
projective limit of the sequence
of tensor products $\left( B(\delta)^{\otimes\ell}\right) _{\ell\in\mathbb{N}%
}$ associated to the projections $\pi_{\ell^{\prime},\ell}(a_{1}\otimes
a_{2}\otimes\ldots\otimes a_{\ell^{\prime}})=a_{1}\otimes a_{2}\otimes
\ldots\otimes a_{\ell}$ if $\ell\leq\ell^{\prime}$. We denote by $\pi
_{\infty,\ell}$ the canonical projection from ${\mathcal{B}}(\delta
)^{\otimes\mathbb{N}}$ onto $B(\delta)^{\otimes\ell}$. By Kolmogorov's
theorem, we know that there exists a unique probability measure $\mathbb{P}$
on the space ${\mathcal{B}}(\delta)^{\otimes\mathbb{N}}$ whose image by each
projection $\pi_{\infty,\ell}$ is the probability $\mathbb{P}^{(\ell)}$
defined by 
\begin{equation}
\mathbb{P}^{(\ell)}(a_{1}\otimes a_{2}\otimes\ldots\otimes
a_{\ell})=\prod_{i=1}^{\ell}p(a_{i}).  \label{prod-meas}
\end{equation}
On the probability space $({\mathcal{B}}(\delta)^{\otimes\mathbb{N}},\mathbb{%
P})$, the random variables 
\begin{equation*}
a_{1}\otimes a_{2}\otimes\ldots\mapsto a_{\ell},\quad\ell\in\mathbb{N}
\end{equation*}
are independent and identically distributed with law $p$.

Now, on this probability space we define the random variables $\mathcal{W}%
_{\ell}$ by 
\begin{equation}
\mathcal{W}_{\ell}=\mathrm{wt}\circ\pi_{\infty,\ell}\ ,  \label{def_Wl}
\end{equation}
that is 
\begin{equation*}
\mathcal{W}_{\ell}(b)=\mathrm{wt}\left( b^{(\ell)}\right)
\end{equation*}
if $b=a_{1}\otimes a_{2}\otimes\ldots$ and $b^{(\ell)}=a_{1}\otimes
a_{2}\otimes\ldots\otimes a_{\ell}$.

The random process $\left( \mathcal{W}_{\ell}\right) _{\ell\in\mathbb{N}}$
is a realization of the random walk on $P$ with law $\mathrm{wt}_{\ast}p$.
Indeed, for any $\beta,\beta^{\prime}\in P$, we have 
\begin{multline*}
\mathbb{P}\left( \mathcal{W}_{\ell+1}=\beta^{\prime}\mid\mathcal{W}_{\ell
}=\beta\right) =\mathbb{P}^{(\ell+1)}\left( \mathrm{wt}\left( b^{(\ell
)}\otimes a_{\ell+1}\right) =\beta^{\prime}\mid\mathrm{wt}\left( b^{(\ell
)}\right) =\beta\right) = \\
\mathbb{P}^{(\ell+1)}\left( \mathrm{wt}\left( b^{(\ell)}\right) +\mathrm{wt}%
\left( a_{\ell+1}\right) =\beta^{\prime}\mid\mathrm{wt}\left(
b^{(\ell)}\right) =\beta\right) =\mathbb{P}^{(\ell+1)}\left( \mathrm{wt}%
\left( a_{\ell+1}\right) =\beta^{\prime}-\beta\right) .
\end{multline*}

\subsubsection{Application to our particular situation}

Now we come back to the particular choice of a probability distribution $p$ on
the crystal $B(\delta)$ which has been described in the previous section ;   the probability of a vertex $b$ depends only on its weight and,  if there is an arrow $b\overset{i}{\rightarrow}b^{\prime}$, then $%
p_{b^{\prime}}=p_{b}\times t_{i}$.

We extend the notation $p_{b}$ to the vertices $b\in B(\delta)^{\otimes\ell}$
using the rule (\ref{prod-meas}) : 
\begin{equation*}
p_{a_{1}\otimes a_{2}\otimes\ldots\otimes a_{\ell}} =
p_{a_{1}}p_{a_{2}}\ldots p_{a_{\ell}}.
\end{equation*}

We recall that if $b\in B(\delta)$, then $p_{b}=\frac{x^{\mathrm{wt}(b)}}{%
s_{\delta}(x)}$. We see that if $b\in B(\delta)^{\otimes\ell}$, then 
\begin{equation}
p_{b}=\frac{x^{\mathrm{wt}(b)}}{s_{\delta}(x)^{\ell}}.  \label{expr_pb2}
\end{equation}
Using the rules of construction of the tensor powers of the crystal graph,
it is straightforward to verify that the previous properties extend  to
tensor powers : let $b,b^{\prime}$ in $B(\delta)^{\otimes\ell}$,

\begin{itemize}
\item Assume we have $\mathrm{wt}(b)=\mathrm{wt}(b^{\prime }).\;$Then $%
p_{b}=p_{b^{\prime }}$.

\item Assume we have an arrow $b\overset{i}{\rightarrow }b^{\prime }$ in $%
B(\delta )^{\otimes \ell }$. Then $p_{b^{\prime }}=p_{b}\times t_{i}$.
\end{itemize}

If we follow the construction described in Subsection \ref{general}, we
obtain a random walk $\left( \mathcal{W}_{\ell}\right) _{\ell\in\mathbb{N}}$
on the weight lattice $P$ with the following properties :

\begin{itemize}
\item

the law of the increments of the random walk $\left( \mathcal{W}_{\ell
}\right) $ is given by 
\begin{equation}
\mathbb{P}\left( \mathcal{W}_{\ell+1}=\beta^{\prime}\mid\mathcal{W}_{\ell
}=\beta\right) =\frac{K_{\delta,\beta^{\prime}-\beta}x^{\beta^{\prime}-\beta
}}{s_{\delta}(x)}  \label{transW}
\end{equation}

\item the expectation of the increments, also called the drift of the random walk,
is 
\begin{equation}
m:=\frac{1}{s_{\delta}(x)}\sum_{\alpha}K_{\delta,\alpha}x^{\alpha}\alpha.
\label{drift}
\end{equation}

\item the probability of a finite path $\pi=\left( \beta_{0}=0,\beta_{1},\beta
_{2},\ldots,\beta_{\ell}\right) $ is given by 
\begin{equation*}
\frac{1}{s_{\delta}(x)^{\ell}}K_{\delta,\alpha^{(1)}}\times\cdots\times
K_{\delta,\alpha^{(\ell)}}x^{\beta_{\ell}}
\end{equation*}
where $\alpha^{(i)}=\beta_{i}-\beta_{i-1}$ for any $i=1,\ldots,\ell.$
\end{itemize}

\bigskip

\noindent\textbf{Remarks:}

\begin{enumerate}
\item Consider a dominant weight $\lambda $ and $B$ a connected component in 
$B(\delta )^{\otimes \ell }$ isomorphic to $B(\lambda ).$ We set 
\begin{equation}
S_{\lambda ,\ell }(x):=\sum_{b\in B}p_{b}=\sum_{\beta \in P}\sum_{b\in
B_{\beta }}\frac{x^{\beta }}{s_{\delta }(x)^{\ell }}=\frac{1}{s_{\delta
}(x)^{\ell }}\sum_{\beta \in P}K_{\lambda ,\beta }x^{\beta }=\frac{%
s_{\lambda }(x)}{s_{\delta }(x)^{\ell }}  \label{def-S(l,p)}
\end{equation}
where the second equality is a consequence of (\ref{expr_pb2}) and the last
equality is obtained by definition of the Weyl character. Clearly, $%
S_{\lambda ,\ell }(x)$ gives the probability $\mathbb{P}(B)$ that a random
vertex of $B(\delta )^{\otimes \ell }$ belongs to $B$. Observe that this
probability does not depend on the connected component $B$ itself but only
on $\lambda $ and $\ell $.

\item When $\lambda =\delta$ is minuscule the situation simplifies. Indeed $%
K_{\delta ,\alpha }=1$ for any weight $\alpha $ of $V(\delta )$, therefore
the map $\mathrm{wt}$ is one-to-one on $B(\delta )$. The probability of the
path $\pi =\left( \beta _{0}=0,\beta _{1},\beta _{2},\ldots ,\beta _{\ell
}\right) $ is then equal to $\frac{x^{\beta _{\ell }}}{s_{\delta }(x)^{\ell }%
}$. In particular, \emph{when }$\delta $\emph{\ is minuscule, two paths
starting and ending at the same points have the same probability}.
\end{enumerate}

\section{Markov chains in the Weyl chamber}

\label{Sec_MarkovCains}The purpose of this section is to introduce a Markov
chain $(\mathcal{H}_{\ell})$ in the Weyl chamber obtained from $(\mathcal{W}%
_{\ell})_{\ell\geq1}$ by an operation on crystals. This operation consists
in the composition of $\mathcal{W}_{\ell}$ with a transformation of $%
B(\delta)^{\otimes\mathbb{N}}$ which plays the role of the Pitman
transformation. Here and in the sequel, $\delta$ is a fixed dominant weight,
i.e. an element of $P_{+}$.

\subsection{The Markov chain $\mathcal{H}$}

The map $\mathfrak{P}$ has been introduced in \S\ \ref{subsec_paths}.\ It
associates to any vertex $b(\ell)\in B(\delta)^{\otimes\ell}$, the highest
weight vertex $\mathfrak{P}(b(\ell))$ of $B(b(\ell))\subset B(\delta
)^{\otimes\ell}$. By Lemma \ref{Lem_HWV}, the transformation on $B(\delta
)^{\otimes \N}$ also denoted by $\mathfrak{P}$ such that $\mathfrak{P}(b)=(%
\mathfrak{P}(b(\ell)))_{\ell\geq1}$ for any $b=(b(\ell))_{\ell\geq 1}\in
B(\delta)^{\otimes\N}$ is well-defined. We then consider the random
variable $\mathcal{H}_{\ell}:=\mathcal{W}_{\ell}\circ\mathfrak{P}$ (see (\ref%
{def_Wl})) defined on the probability space $\Omega(\delta)=(B(\delta
)^{\otimes\mathbb{N}},\mathbb{P})$ with values in $P_{+}$. This yields a
stochastic process $\mathcal{H=(H}_{\ell}\mathcal{)}_{\ell\geq0}$.

\begin{proposition}
\label{Prop_laws}Consider $\ell \in \mathbb{Z}$ and $\lambda \in P_{+}$.
Then $\mathbb{P}(\mathcal{H}_{\ell }=\lambda )=f_{\lambda ,\delta }^{\ell }%
\frac{s_{\lambda }(x)}{s_{\delta }(x)^{\ell }}$.
\end{proposition}

\begin{proof}
By definition of the random variable $\mathcal{H}_{\ell },$ we have 
\begin{equation*}
\mathbb{P}(\mathcal{H}_{\ell }=\lambda )=\sum_{b_{\lambda }\in \mathrm{HW}%
(B(\delta )^{\otimes \ell }),\mathrm{wt}(b_{\lambda })=\lambda }\mathbb{P}%
(B(b_{\lambda })).
\end{equation*}
We have seen in (\ref{def-S(l,p)}) that $\mathbb{P}(B(b_{\lambda }))=\frac{%
s_{\lambda }(x)}{s_{\delta }(x)^{\ell }}$ does not depend on $b_{\lambda }$
but only on $\lambda $.\ By definition of $f_{\lambda
,\delta }^{\ell }$  (see (\ref{def_f}))  and Theorem \ref{TH_Ka}, the number of connected
components in $B(\delta )^{\otimes \ell }$ isomorphic to $B(\lambda )$ is
equal to $f_{\lambda ,\delta }^{\ell }$. This gives $\mathbb{P}(\mathcal{H}%
_{\ell }=\lambda )=f_{\lambda ,\delta }^{\ell }\frac{s_{\lambda }(x)}{%
s_{\delta }(x)^{\ell }}$.
\end{proof}

We can now state the main result of this Section

\begin{theorem}
\label{Th_main}The stochastic process $\mathcal{H}$ is a Markov chain with
transition probabilities 
\begin{equation*}
\Pi _{\mathcal{H}}(\mu ,\lambda )=\frac{m_{\mu ,\delta }^{\lambda
}s_{\lambda }(x)}{s_{\delta }(x)s_{\mu }(x)}\quad \lambda ,\mu \in P_{+}.
\end{equation*}
\end{theorem}

\begin{proof}
Consider a sequence of dominant weights $\lambda ^{(1)},\ldots ,\lambda
^{(\ell )},\lambda ^{(\ell +1)}$ such that $\lambda ^{(1)}=\delta ,\lambda
^{(\ell )}=\mu $ and $\lambda ^{(\ell +1)}=\lambda $. Recall that, for any $%
k=1,\ldots ,\ell ,$ the integer  $m_{\lambda ^{(k)},\delta }^{\lambda ^{(k+1)}}$ is the
multiplicity of $V(\lambda ^{(k+1)})$ in $V(\lambda ^{(k)})\otimes V(\delta
).\;$Let $b\in \Omega (\delta )$ be such that $\mathcal{H}_{k}(b)=\lambda
^{(k)}$ for any $k=1,\ldots ,\ell +1$. Write $h^{(\ell +1)}=a_{1}^{h}\otimes
\cdots \otimes a_{\ell +1}^{h}$ for the highest weight vertex of $B(b^{(\ell
+1)})\subset B(\delta )^{\otimes \ell +1}$.\ By Lemma \ref{Lem_HWV}, for $k=1,\ldots ,\ell +1,$ the vertex $%
h^{(k)}=a_{1}^{h}\otimes \cdots \otimes a_{k}^{h}$ is  the highest weight vertex of $b^{(k)}$ and has weight $\lambda ^{(k)}.$
Let us denote by $S$ the set of highest weight vertices of $B(\delta
)^{\otimes \ell +1}$ whose projection on $B(\delta )^{k}$ has weight $%
\lambda ^{(k)}$ for any $k=1,\ldots ,\ell +1$.\ By Theorem \ref{TH_Ka} and a
straightforward induction, we have 
\begin{equation*}
\mathrm{card}(S)=\prod_{k=1}^{\ell }m_{\lambda ^{(k)},\delta }^{\lambda
^{(k+1)}}.
\end{equation*}
We can now write 
\begin{equation*}
\mathbb{P}(\mathcal{H}_{\ell +1}=\lambda ,\mathcal{H}_{k}=\lambda ^{(k)}%
\text{for any }k=1,\ldots ,\ell )=\sum_{h\in S}\ \sum_{b\in
B(h)}p_{b}=\sum_{h\in S}\mathbb{P}(B(h)).
\end{equation*}
Each $h$ in $S$ is a highest weight vertex of weight $\lambda $ and one gets $\mathbb{P}(B(h))=S_{\lambda ,\ell +1}(x)$
 (see (\ref{def-S(l,p)})). This
gives 
\begin{equation*}
\mathbb{P}(\mathcal{H}_{\ell +1}=\lambda ,\mathcal{H}_{k}=\lambda ^{(k)}%
\text{ for any }k=1,\ldots ,\ell )=\prod_{k=1}^{\ell }m_{\lambda
^{(k)},\delta }^{\lambda ^{(k+1)}}\times S_{\lambda ,\ell +1}(x).
\end{equation*}
Similarly, we have 
\begin{equation*}
\mathbb{P}(\mathcal{H}_{k}=\lambda ^{(k)}\text{ for any }k=1,\ldots ,\ell
)=\prod_{k=1}^{\ell -1}m_{\lambda ^{(k)},\delta }^{\lambda ^{(k+1)}}\times
S_{\mu ,\ell }(x).
\end{equation*}
Hence 
\begin{equation*}
\mathbb{P}(\mathcal{H}_{\ell +1}=\lambda \mid \mathcal{H}_{k}=\lambda ^{(k)}%
\text{ for any }k=1,\ldots ,\ell )=\frac{m_{\lambda ^{(\ell )},\delta
}^{\lambda ^{(\ell +1)}}S_{\lambda ^{(\ell +1)},\ell +1}(x)}{S_{\lambda
^{(\ell )},\ell }(x)}=\frac{m_{\mu ,\delta }^{\lambda }s_{\lambda }(x)}{%
s_{\delta }(x)s_{\mu }(x)}.
\end{equation*}
In particular, $\mathbb{P}(\mathcal{H}_{\ell +1}=\lambda \mid \mathcal{H}%
_{k}=\lambda ^{(k)}$ for any $k=1,\ldots ,\ell )$ depends only on $\lambda $
and $\mu = \lambda^{(\ell)} $ which shows the Markov property.
\end{proof}

\subsection{Intertwining operators}

For any $\lambda\in P_{+}$ and $\beta\in P,$ the event $(\mathcal{W}_{\ell
}=\beta,\mathcal{H}_{\ell}=\lambda)$ contains all the elements $b$ in $%
\Omega(\delta)$ such that $b^{(\ell)}$ has weight $\beta$ and belongs to a
connected component of $B(\delta)^{\otimes\ell}$ with highest weight $%
\lambda $.\ This gives $\mathbb{P}(\mathcal{W}_{\ell}=\beta,\mathcal{H}%
_{\ell}=\lambda)=\frac{1}{s_{\delta}(x)^{\ell}}f_{\ell,\lambda}K_{\lambda,%
\beta }x^{\beta}$ since there is $f_{\ell,\lambda}$ connected components in $%
B(\delta)^{\otimes\ell}$ of highest weight $\lambda$, each of them contains $%
K_{\lambda,\beta}$ vertices of weight $\beta$ whose common probability is $%
\frac{x^{\beta}}{s_{\delta}(x)^{\ell}}$.\ Using Proposition \ref{Prop_laws},
we obtain 
\begin{equation}
\mathbb{P}(\mathcal{W}_{\ell}=\beta\mid\mathcal{H}_{\ell}=\lambda )=\frac{%
K_{\lambda,\beta}x^{\beta}}{s_{\lambda}(x)}  \label{def_Inter}
\end{equation}
which is independent of $\ell$.\ We set $\mathcal{K}(\lambda,\beta):=\mathbb{P}(%
\mathcal{W}_{\ell}=\beta\mid\mathcal{H}_{\ell}=\lambda)=\frac{K_{\lambda
,\beta}x^{\beta}}{s_{\lambda}(x)}$.

\begin{theorem}
We have the intertwining relation $\Pi _{\mathcal{H}}\mathcal{K}=\mathcal{K}%
\Pi _{\mathcal{W}}.$
\end{theorem}

\begin{proof}
Consider $\mu \in P_{+}$ and $\beta \in P$. Write $(\Pi _{\mathcal{H}}%
\mathcal{K})(\mu ,\beta )$ for the coefficient of the matrix product $\Pi _{%
\mathcal{H}}\mathcal{K}$ associated to the pair $(\mu ,\beta )$. Define $(%
\mathcal{K}\Pi _{\mathcal{W}})(\mu ,\beta )$ similarly.\ By using Theorem %
\ref{Th_main} and (\ref{def_Inter}), we have 
\begin{equation*}
(\Pi _{\mathcal{H}}\mathcal{K})(\mu ,\beta )=\sum_{\lambda \in P_{+}}\Pi _{%
\mathcal{H}}(\mu ,\lambda )\mathcal{K}(\lambda ,\beta )=\sum_{\lambda \in
P_{+}}\frac{m_{\mu ,\delta }^{\lambda }s_{\lambda }(x)}{s_{\delta }(x)s_{\mu
}(x)}\frac{K_{\lambda ,\beta }x^{\beta }}{s_{\lambda }(x)}=\sum_{\lambda \in
P_{+}}\frac{m_{\mu ,\delta }^{\lambda }K_{\lambda ,\beta }x^{\beta }}{%
s_{\delta }(x)s_{\mu }(x)}.
\end{equation*}
This gives 
\begin{equation*}
(\Pi _{\mathcal{H}}\mathcal{K})(\mu ,\beta )=\frac{x^{\beta }}{s_{\delta
}(x)s_{\mu }(x)}\sum_{\lambda \in P_{+}}m_{\mu ,\delta }^{\lambda
}K_{\lambda ,\beta }\text{.}
\end{equation*}
On the other hand, using (\ref{transW}) 
\begin{equation*}
(\mathcal{K}\Pi _{\mathcal{W}})(\mu ,\beta )=\sum_{\gamma \in P}\mathcal{K}%
(\mu ,\gamma )\Pi _{\mathcal{W}}(\gamma ,\beta )=\sum_{\gamma \in P}\frac{%
K_{\mu ,\gamma }x^{\gamma }}{s_{\mu }(x)}K_{\delta ,\beta -\gamma }\frac{%
x^{\beta -\gamma }}{s_{\delta }(x)}.
\end{equation*}
This yields to
$\displaystyle 
(\mathcal{K}\Pi _{\mathcal{W}})(\mu ,\beta )=\frac{x^{\beta }}{s_{\delta
}(x)s_{\mu }(x)}\sum_{\gamma \in P}K_{\mu ,\gamma }K_{\delta ,\beta -\gamma
}. $
Therefore the equality $\Pi _{\mathcal{H}}\mathcal{K=K}\Pi _{\mathcal{W}}$
reduces to 
$\displaystyle 
\sum_{\lambda \in P_{+}}m_{\mu ,\delta }^{\lambda }K_{\lambda ,\beta
}=\sum_{\gamma \in P}K_{\mu ,\gamma }K_{\delta ,\beta -\gamma }
$
which was established in Lemma \ref{Lem_identity}.
\end{proof}

\section{Restriction to the Weyl chamber}

\label{Sec_restric}

We have explicit formulae for the transition matrices $\Pi_{\mathcal{W}}$
and $\Pi_{\mathcal{H}}$ of the Markov chains $\mathcal{W}$ and $\mathcal{H}$. The  matrix $\Pi_{\mathcal{W}}$ has entries in $P$  and the matrix $\Pi_{\mathcal{H}}$ has entries in $P_{+}$%
. We will see that if the representation $\delta$ is minuscule then $\Pi _{%
\mathcal{H}}$ is a Doob transform of the restriction of $\Pi_{\mathcal{W}}$
to $P_{+}$.

\subsection{Doob transform of the random walk $\mathcal{W}$ restricted to
the Weyl chamber}

\label{subsec_Psi} Recall that $(\mathcal{W}_{\ell})_{\ell\geq1}$ is a
Markov chain with transition matrix $\Pi_{\mathcal{W}}$. Since the closed
Weyl chamber $\overline{C}$ is a subset of $\mathbb{R}^{N}$, it makes sense
to consider the substochastic matrix $\Pi_{\mathcal{W}}^{\overline{C}%
}(\mu,\lambda)$, that is the restriction of the transition matrix of $%
\mathcal{W}$ to $\overline{C}$.\ The following Proposition answers the
question whether the transition matrix $\Pi_{\mathcal{H}}$ of the Markov
chain $\mathcal{H}$ (see Theorem \ref{Th_main}) can be regarded as a Doob
transform of $\Pi_{\mathcal{W}}^{\overline{C}}$.\ We denote by $\psi$ the
function defined on $P_{+}$ by $\psi(\lambda)=x^{-\lambda}s_{\lambda}(x)$.

\begin{proposition}
\label{Th_psi}If $\delta $ is a minuscule representation, then the
transition matrix $\Pi _{\mathcal{H}}$ is the Doob $\psi $-transform of the
substochastic matrix $\Pi _{\mathcal{W}}^{\overline{C}}$, in particular $%
\psi $ is  harmonic with respect to  this substochastic matrix. If $\delta $ is
not minuscule, then the transition matrix $\Pi _{\mathcal{H}}$ cannot be
realized as a Doob transform of the substochastic matrix $\Pi _{\mathcal{W}%
}^{\overline{C}}$.
\end{proposition}

\begin{proof}
Given $\lambda ,\mu $ in $P_{+}$ we have 
\begin{equation*}
\Pi _{\mathcal{W}}^{\overline{C}}(\mu ,\lambda )=K_{\delta ,\lambda -\mu
}p_{\lambda -\mu }\text{ and }\Pi _{\mathcal{H}}(\mu ,\lambda )=\frac{m_{\mu
,\delta }^{\lambda }s_{\lambda }(x)}{s_{\delta }(x)s_{\mu }(x)}.
\end{equation*}
If we assume $\delta $ is minuscule, we have $K_{\delta ,\lambda -\mu
}=m_{\mu ,\delta }^{\lambda }\in \{0,1\}$ for any $\lambda ,\mu \in P_{+}$
by the Remark in \S~\ref{subsec_minus}.\ Therefore 
\begin{equation*}
\Pi _{\mathcal{H}}(\mu ,\lambda )=\frac{\psi (\lambda )}{\psi (\mu )}\Pi _{%
\mathcal{W}}^{\overline{C}}(\mu ,\lambda ).
\end{equation*}
Conversely, if $\Pi _{\mathcal{H}}$ can be realized as a $h$-transform of
the substochastic matrix $\Pi _{\mathcal{W}}^{\overline{C}}$ we must have 
\begin{equation*}
\frac{h(\lambda )}{h(\mu )}K_{\delta ,\lambda -\mu }\frac{x^{\lambda }}{%
x^{\mu }s_{\delta }(x)}=\frac{m_{\mu ,\delta }^{\lambda }s_{\lambda }(x)}{%
s_{\delta }(x)s_{\mu }(x)}\text{ for any }\lambda ,\mu \in P_{+}.
\end{equation*}
This is equivalent to the equality 
\begin{equation}
K_{\delta ,\lambda -\mu }\frac{h(\lambda )}{h(\mu )}=m_{\mu ,\delta
}^{\lambda }\frac{x^{-\lambda }s_{\lambda }(x)}{x^{-\mu }s_{\mu }(x)}\text{
for any }\lambda ,\mu \in P_{+}.  \label{equality}
\end{equation}
Since $h$ and $\psi $ are positive functions, there exists a positive
function $g$ such that $h(\lambda )=g(\lambda )\psi (\lambda )$, for any $%
\lambda \in P_{+}$.\ We thus obtain 
\begin{equation}
K_{\delta ,\lambda -\mu }\frac{g(\lambda )}{g(\mu )}=m_{\mu ,\delta
}^{\lambda }\text{ for any }\lambda ,\mu \in P_{+}.  \label{rela}
\end{equation}
Assume (\ref{rela}) holds and $\delta $ is not minuscule. By the Remark in 
\S\ \ref{subsec_minus}, there exists a dominant weight $\kappa \in P_{+}$
distinct of $\delta $ such that $K_{\delta ,\kappa }\neq 0$. For $\mu =0$
and $\lambda =\kappa ,$ we then obtain 
\begin{equation*}
K_{\delta ,\kappa }\frac{g(\kappa )}{g(0)}=m_{0,\delta }^{\kappa }.
\end{equation*}
Now, since $\kappa \neq \delta ,$ we have $m_{0,\delta }^{\kappa }=0.$
Recall that $K_{\delta ,\kappa }\neq 0$. This gives $g(\kappa )=0$. Contradiction. 
\end{proof}

\subsection{Limit of $\protect\psi$ along a drift}

The purpose of the remaining paragraphs of this section is to connect the
Markov chain $\mathcal{H}$ to the random walk $\mathcal{W}$ conditioned to
never exit $\overline{C}$. We have seen in \S\ \ref%
{subsec_paths_crys} that, for the minuscule representation $V(\delta)$, the
vertices of $B(\delta)^{\otimes\ell}$ can be identified with the paths $%
Z(\delta,\ell).$ Moreover, by Proposition \ref{Prop_minuscule}, the highest
weight vertices of $B(\delta)^{\otimes\ell}$ are   identified with the
paths $Z^{+}(\delta,\ell)$ which remains in $\overline{C}$. The drift $m$ of 
$\mathcal{W}$ given by Formula (\ref{drift}) belongs to $C$ (the open Weyl
chamber) when 
\begin{equation}
m=\sum_{i\in I}m_{i}\omega_{i}\text{ with }m_{i}>0\text{ for any }i\in I.
\label{full}
\end{equation}

\begin{lemma}
\label{lemma_ti}

\begin{enumerate}
\item We have $m\in C$ if and only if $0<t_{i}<1$ for any $i\in I$.

\item For any direction $\vec{d}$ in $C$, there exists an $n$-tuple $%
t=(t_{1},\ldots ,t_{n})$ with $0<t_{i}<1$ such that $\vec{d}$ is the direction of
the drift $m$ associated to the random walk $\mathcal{W}$ defined from $t$
as in \S\ \ref{subsec_paths_crys}.
\end{enumerate}
\end{lemma}

\begin{proof}
1: Recall that $m=\sum_{b\in B(\delta )}p_{b}\mathrm{wt}(b)$. By (\ref%
{defWeightCrys}), for any $b\in B(\delta)$, the coordinates of the weight $\mathrm{wt}(b)$ are determined by the
$i$-chains containing $b$. Here by such an $i$-chain, we mean the sub-crystal
containing all the vertices connected to $b$ by arrows $i$. By (\ref{defWeightCrys}) and (%
\ref{expr_pb2}), the contribution of any $i$-chain of length $k$%
\begin{equation*}
a_{1}\overset{i}{\rightarrow }a_{2}\overset{i}{\rightarrow }\cdots \overset{i%
}{\rightarrow }a_{k+1}
\end{equation*}
to $m$ is equal to 
\begin{equation*}
p_{a_{1}}\sum_{j=0}^{[k/2]}(k-2j)(t_{i}^{j}-t_{i}^{k-j})\omega _{i}.
\end{equation*}
For each fixed $i$, all these contributions are positive if $0<t_{i}<1$ and
they are all nonpositive if $i>1$. This proves Assertion 1.

2: For any $i=1,\ldots ,n$, let $S_{i}$ be the set of vertices $a$ in $%
B(\delta )$ such that $\varepsilon _{i}(a)=0$. For any $a\in S_{i},$ write $%
k_{a}$ the length of the $i$-chain in $B(\delta )$ starting at $a$. We
denote by $m(t)=\sum_{i=1}^{n}m_{i}(t)\omega _{i}\in C$ the drift
corresponding to the $n$-tuple $t=(t_{1},\ldots ,t_{n})$ with $0<t_{i}<1.$
By the previous arguments 
\begin{equation*}
m_{i}(t)=m_{i}(t_{i})=\sum_{a\in
S_{i}}p_{a}\sum_{j=0}^{[k_{a}/2]}(k_{a}-2j)(t_{i}^{j}-t_{i}^{k_{a}-j})
\end{equation*}%
depends only on $t_{i}$. Write $M_{i}=\max \{m_{i}(t_{i})\mid 0<t_{i}<1\}$.
Then $M_{i}>0$ and for any $u_{i}$ in $]0,M_{i}]$, there exists $t_{i}\in
]0,1[$ such that $m_{i}(t_{i})=u_{i}$. Now consider a direction $\vec{d}$ in $C$.
Assume $v=(v_{1},\ldots ,v_{n})$ belongs to $\R\vec{d}$.$\;$There exists $c\in 
\mathbb{R}_{>0}$ such that $u_{i}=cv_{i}\in ]0,M_{i}]$ for any $i=1,\ldots ,n
$. It then suffices to choose each $t_{i}$ so that $m_{i}(t_{i})=u_{i}$.
\end{proof}

\bigskip

In the sequel, we \emph{assume that }$m\in C$. Consider  a positive
root $\alpha$, decomposed as  $\alpha=\alpha_{i_{1}}+\cdots+\alpha_{i_{r}}$ on the basis of simple roots. Then $t^{[\alpha]}=t_{i_{1}}%
\cdots t_{i_{r}}.$ According to the fact that $m\in C$, we immediately
derive from the previous lemma that $0<t^{[\alpha]}<1$. In particular the
product 
\begin{equation}
\nabla=\prod_{\alpha\in R_{+}}\frac{1}{1-x^{-\alpha}}=\prod_{\alpha\in R_{+}}%
\frac{1}{1-t^{[\alpha]}}  \label{def_nabla}
\end{equation}
running on the (finite) set of positive roots is well-defined and finite.

\begin{proposition}
\label{prop-_imit-psi}Assume $m\in C$ and consider a sequence 
$(\lambda ^{(a)})_{a\in \mathbb{N}}$
of dominant
weights  such that $\lambda ^{(a)}=am+o(a).$
Then $\lim_{a\rightarrow +\infty }x^{-\lambda ^{(a)}}s_{\lambda
^{(a)}}(x)=\nabla .$
\end{proposition}

\begin{proof}
By the Weyl character formula, we have 
\begin{equation*}
s_{\lambda ^{(a)}}(x)=\nabla \sum_{w\in W}\varepsilon (w)x^{w(\lambda
^{(a)}+\rho )-\rho }.
\end{equation*}
This gives 
\begin{equation*}
x^{-\lambda ^{(a)}}s_{\lambda ^{(a)}}(x)=\nabla \sum_{w\in W}\varepsilon
(w)x^{w(\lambda ^{(a)}+\rho )-\lambda ^{(a)}-\rho }=\nabla \sum_{w\in
W}\varepsilon (w)t^{[\lambda ^{(a)}+\rho -w(\lambda ^{(a)}+\rho )]}.
\end{equation*}
For $w=1$, one gets   $\varepsilon (w)t^{[\lambda ^{(a)}+\rho -w(\lambda ^{(a)}+\rho
)]}=1$. So it suffices to prove that $\lim_{a\rightarrow +\infty
}t^{[\lambda ^{(a)}+\rho -w(\lambda ^{(a)}+\rho )]}=0$ for any $w\neq 1$.
Consider $w\neq 1$ and set 
\begin{equation*}
u(a)=\lambda ^{(a)}+\rho -w(\lambda ^{(a)}+\rho )=\lambda ^{(a)}-w(\lambda
^{(a)})+\rho -w(\rho ).
\end{equation*}
Since $m\in C$, the weight $\lambda ^{(a)}$ belongs to $C$ for $a$ large enough. In the
sequel, we can thus assume that $\lambda ^{(a)}+\rho \in C$.\ Its stabilizer
under the action of the Weyl group is then trivial. Now the weights of the
finite-dimensional representation $V(\lambda ^{(a)}+\rho )$ are stable under
the action of $W$. Thus $w(\lambda ^{(a)}+\rho )$ is a weight of $V(\lambda
^{(a)}+\rho )$. This implies that $w(\lambda ^{(a)}+\rho )\leq \lambda
^{(a)}+\rho ,$ that is $u(a)\in Q_{+}$ is a linear combination of simple
root with nonnegative coefficients. Since $\lambda ^{(a)}=am+o(a)$ and $\rho 
$ is fixed, we have 
\begin{equation*}
u(a)=a(m-w(m))+o(a)\in Q_{+}.
\end{equation*}
For any $w\neq 1,$ one gets  $m\neq w(m)$.\ We can set $m-w(m)=\sum_{i=1}^{n}m_{i}%
\alpha _{i}$ where the $m_{i}\ $belong to $\mathbb{R}_{\geq 0}$ for any $%
i=1,\ldots ,n$ and $m_{i_{0}}=\max_{i=1,\ldots ,n}\{m_{i}\}>0$. This gives $%
u(a)=a\sum_{i=1}^{n}m_{i}\alpha _{i}+o(a)$ and 
\begin{equation*}
t^{[\lambda ^{(a)}+\rho -w(\lambda ^{(a)}+\rho )]}=(t^{[m]})^{a}\times
t^{[o(a)]}\leq t_{i_{0}}^{am_{i_{0}}}\times t^{[o(a)]}
\end{equation*}
tends to $0$ when $a$ tends to infinity for $0<t_{i_{a}}<1$.
\end{proof}

\subsection{Random walks $\mathcal{W}$ with fixed drift}

Let $\delta $ be a minuscule representation.\ In that case, there is a
bijection between the paths from $0$ to $\lambda \in P_{+}$ and the highest
weight vertices of $B(\delta )^{\otimes \ell }$ of weight $\lambda $.\
Assume the probability distribution on $B(\delta )$ is such that $m\in C.\;$%
The following proposition shows that the probability distribution on $%
B(\delta )$ is completely determined by $m$ under the previous hypotheses.
This remark will not be needed in the sequel of the article.

\begin{proposition}
\label{Th_m_law}Assume $\delta $ is minuscule. If two probability
distributions on $B(\delta )$ satisfying the conditions imposed in \S\ \ref%
{subsec_def_pa} have the same drift $m\in C$, then they coincide.
\end{proposition}

\begin{proof}
Assume we have a probability distribution on $B(\delta )$ as in \S\ \ref%
{subsec_def_pa}.\ Consider $\lambda ^{(\ell )}=\ell m+O(1 )$ a sequence
of dominants weights which tends to infinity in the direction of the drift $m$.\ Let $%
\gamma \in P$. Since $\delta $ is minuscule, all the paths from $0$ to $%
\lambda ^{(\ell )}-\gamma $ (resp. to $\lambda ^{(\ell )}$) have the same
probability ; furthermore,  there exist $f_{\lambda ^{(\ell )}-\gamma ,\delta }^{\ell }$
(resp. $f_{\lambda ^{(\ell )},\delta }^{\ell }$) such paths by Proposition %
\ref{Prop_minuscule}. We thus have 
\begin{equation*}
\frac{f_{\lambda ^{(\ell )}-\gamma ,\delta }^{\ell }x^{\lambda ^{(\ell
)}-\gamma }}{f_{\lambda ^{(\ell )},\delta }^{\ell }x^{\lambda ^{(\ell )}}}=%
\frac{\mathbb{P}(\mathcal{W}_{1}\in \overline{C},\ldots ,\mathcal{W}_{\ell
}\in \overline{C},\mathcal{W}_{\ell }=\lambda ^{(\ell )}-\gamma )}{\mathbb{P}%
(\mathcal{W}_{1}\in \overline{C},\ldots ,\mathcal{W}_{\ell }\in \overline{C},%
\mathcal{W}_{\ell }=\lambda ^{(\ell )})}.
\end{equation*}
By Theorem \ref{tll-q}, we know that this quotient tends to $1$ when $\ell $
tends to infinity. This gives 
\begin{equation}
\lim_{\ell \rightarrow +\infty }\frac{f_{\lambda ^{(\ell )}-\gamma ,\delta
}^{\ell }}{f_{\lambda ^{(\ell )},\delta }^{\ell }}=x^{\gamma },
\label{lim_gamma}
\end{equation}
hence $x^{\gamma }$ is determined by $m$. Thus the unique probability
distribution with drift $m$ defined on $B(\delta )$ verifies $p_{a}=\frac{x^{%
\mathrm{wt}(a)}}{s_{\delta }(x)}$ where $x^{\mathrm{wt}(a)}$ and $s_{\delta
}(x)$ are given by (\ref{lim_gamma}).
\end{proof}

\subsection{Transition matrix of $\mathcal W$ conditioned to never
exit the Weyl chamber}

In this paragraph, we \emph{assume }$\delta $\emph{\ is minuscule} and $m\in
C$. Set $\Pi =\Pi _{\mathcal{W}}^{\overline{C}}$. Let us denote by $\Gamma $
the Green function associated to the substochastic matrix $\Pi $. Consider $%
\mu \in P_{+}$. For any $\lambda \in P$, we have $\Gamma (\mu ,\lambda
)=\sum_{\ell \geq 0}\Pi ^{\ell }(\mu ,\lambda ).$ Clearly, $\Gamma (\mu
,\lambda )=0$ if $\lambda \notin P_{+}$. We consider the Martin Kernel 
\begin{equation}
K(\mu ,\mathcal{W}_{\ell })=\frac{\Gamma (\mu ,\mathcal{W}_{\ell })}{\Gamma
(0,\mathcal{W}_{\ell })}  \label{limit}
\end{equation}%
defined almost surely for $\ell $ large enough.

In order to apply Theorem \ref{Th_Doob}, we want to prove that $K$ converges
almost surely to the harmonic function $\psi $ of \S\ \ref{subsec_Psi}.

Write $b_{\mu }$ for the highest weight vertex of $B(\mu ).$ For $\lambda
\in P_{+}$, let $B(\mu ,\ell ,\lambda )$ be the subset of vertices $%
b=a_{1}\otimes \cdots \otimes a_{\ell }$ in $B(\delta )^{\otimes \ell }$
such that 
\begin{equation*}
\mathrm{wt}(b_{\mu }\otimes a_{1}\otimes \cdots \otimes a_{k})\in P_{+}\text{
for any }k=1,\ldots ,\ell \text{ and }\mathrm{wt}(b_{\mu }\otimes b)=\lambda
.
\end{equation*}
By definition of $\Pi =\Pi _{\mathcal{W}}^{\overline{C}}$, and since $\delta 
$ is minuscule, we have 
\begin{equation*}
\Pi ^{\ell }(\mu ,\lambda )=\mathrm{card}(B(\mu ,\ell ,\lambda ))\frac{%
x^{\lambda -\mu }}{s_{\delta }(x)^{\ell }}.
\end{equation*}
Indeed, all the paths from $\mu $ to $\lambda $ of length $\ell $ have the
same probability $\frac{x^{\lambda -\mu }}{s_{\delta }(x)^{\ell }}$.\ By (%
\ref{def_f}), Theorem~\ref{TH_Ka} and Proposition \ref{Prop_minuscule}, we
know that 
\begin{equation*}
\mathrm{card}(B(\mu ,\ell ,\lambda ))=f_{\lambda /\mu }^{\ell }
\end{equation*}
and $f_{\lambda /\mu }^{\ell }$ is the number of highest weight vertices of
weight $\lambda $ in $B(\mu )\otimes B(\delta )^{\otimes \ell }$. According
to Lemma \ref{Lem_HWV}, they can be written as  $b_{\mu }\otimes b$
with $\varepsilon _{i}(b)\leq \varphi _{i}(b_{\mu })$ for any $i\in I$.

\bigskip

\noindent\textbf{Remark:} When $\delta$ is not minuscule, the vertices $%
b_{\mu}\otimes b$ with $b\in B(\lambda,\mu,\ell)$ are not necessarily of
highest weight and we can have $\mathrm{card}(B(\lambda,\mu,\ell
))>f_{\lambda/\mu}^{\ell}.$ A vertex of $B(\delta)^{\ell}$ which is not of
highest weight can yield a path in $\overline{C}$.

\bigskip

According to the Proposition \ref{Prop_dec_skew}, given any sequence $%
\lambda ^{(a)}$ of weights of the form $\lambda ^{(a)}=am+o(a)$, we can
write for $a$ large enough 
\begin{equation*}
\Gamma (\mu ,\lambda ^{(a)})=\sum_{\ell \geq 0}f_{\lambda ^{(a)}/\mu }^{\ell
}\frac{x^{\lambda ^{(a)}-\mu }}{s_{\delta }(x)^{\ell }}=x^{\lambda
^{(a)}-\mu }\sum_{\ell \geq 0}\frac{1}{s_{\delta }(x)^{\ell }}\sum_{\gamma
\in P}f_{\lambda ^{(a)}-\gamma }^{\ell }K_{\mu ,\gamma }=x^{\lambda
^{(a)}-\mu }\sum_{\gamma \in P}K_{\mu ,\gamma }\sum_{\ell \geq 0}\frac{%
f_{\lambda ^{(a)}-\gamma }^{\ell }}{s_{\delta }(x)^{\ell }}.
\end{equation*}
Since 
\begin{equation*}
\Gamma (0,\lambda ^{(a)})=x^{\lambda ^{(a)}}\sum_{\gamma \in P}K_{0,\gamma
}\sum_{\ell \geq 0}\frac{f_{\lambda ^{(a)}-\gamma }^{\ell }}{s_{\delta
}(x)^{\ell }}=\sum_{\ell \geq 0}f_{\lambda }^{\ell }\frac{x^{\lambda ^{(a)}}%
}{s_{\delta }(x)^{\ell }},
\end{equation*}
this yields for $a$ large enough 
\begin{equation*}
K(\mu ,\lambda ^{(a)})=x^{-\mu }\sum_{\gamma \in P}K_{\mu ,\gamma }\frac{%
\sum_{\ell \geq 0}\frac{f_{\lambda ^{(a)}-\gamma }^{\ell }}{s_{\delta
}(x)^{\ell }}}{\sum_{\ell \geq 0}\frac{f_{\lambda ^{(a)}}^{\ell }}{s_{\delta
}(x)^{\ell }}}=x^{-\mu }\sum_{\gamma \in P}K_{\mu ,\gamma }x^{\gamma }\frac{%
\sum_{\ell \geq 0}f_{\lambda ^{(a)}-\gamma }^{\ell }\frac{x^{\lambda
^{(a)}-\gamma }}{s_{\delta }(x)^{\ell }}}{\sum_{\ell \geq 0}f_{\lambda
^{(a)}}^{\ell }\frac{x^{\lambda ^{(a)}}}{s_{\delta }(x)^{\ell }}}.
\end{equation*}
Thus 
\begin{equation}
K(\mu ,\lambda ^{(a)})=x^{-\mu }\sum_{\gamma \text{ weight of }V(\mu
)}K_{\mu ,\gamma }x^{\gamma }\frac{\Gamma (0,\lambda ^{(a)}-\gamma )}{\Gamma
(0,\lambda ^{(a)})}.  \label{expand_K}
\end{equation}
Now we have the following proposition

\begin{proposition}
\label{Prop_quotient}Consider $\gamma \in P$ a fixed weight. Then, under the
previous assumptions on $m$%
\begin{equation*}
\lim_{a\rightarrow +\infty }\frac{\Gamma (0,\mathcal{W}_{a}-\gamma )}{\Gamma
(0,\mathcal{W}_{a})}=1\text{ (a.s).}
\end{equation*}
\end{proposition}

\begin{proof}
The weights $\gamma $ run over the set of weights of $V(\mu )$. This set is
finite therefore the statement follows immediately by applying Theorem \ref%
{ren-q}.
\end{proof}

\bigskip

The strong law of large numbers states that $\mathcal{W}_{a}=am+o(a)$ almost
surely. With (\ref{expand_K}) this implies that almost surely, for $a$ large
enough 
\begin{equation*}
K(\mu ,\mathcal{W}_{a})=x^{-\mu }\sum_{\gamma \text{ weight of }V(\mu
)}K_{\mu ,\gamma }x^{\gamma }\frac{\Gamma (0,\mathcal{W}_{a}-\gamma )}{%
\Gamma (0,\mathcal{W}_{a})}\text{.}
\end{equation*}
Proposition \ref{Prop_quotient} then gives 
\begin{equation}
L=\lim_{a\rightarrow +\infty }K(\mu ,\mathcal{W}_{a})=x^{-\mu }\sum_{\gamma 
\text{ weight of }V(\mu )}K_{\mu ,\gamma }x^{\gamma }=x^{-\mu }s_{\mu
}(x)=\psi (\mu )\text{ (a.s).}  \label{K_tends_psi}
\end{equation}
that is, $L$ coincides with the harmonic function of Proposition \ref{Th_psi}%
. By Theorem \ref{Th_Doob}, there exists a constant $c$ such that $\psi =ch_{%
\overline{C}}$ where $h_{\overline{C}}$ is the harmonic function defined in 
\S\ \ref{sub_sec_Doobh} associated to the restriction of $(\mathcal{W}_{\ell
})_{\ell \geq 0}$ to the close Weyl chamber $\overline{C}$.\ By Theorems \ref%
{Th_psi} and \ref{Th_main}, we thus derive 
\begin{equation*}
\Pi _{h_{\overline{C}}}(\mu ,\lambda )=\Pi _{\psi }(\mu ,\lambda )=\Pi _{%
\mathcal{H}}(\mu ,\lambda )=\frac{m_{\mu ,\delta }^{\lambda }s_{\lambda }(x)%
}{s_{\delta }(x)s_{\mu }(x)}=\frac{s_{\lambda }(x)}{s_{\delta }(x)s_{\mu }(x)%
}
\end{equation*}
since $\delta $ is a minuscule representation (so that $m_{\mu ,\delta
}^{\lambda }=1$ for any $\lambda ,\mu \in P_{+}$ such that $\lambda -\mu $
is a weight of $B(\delta )$ (see \S\ \ref{subsec_minus})).

\begin{theorem}
\label{Th_coincide}Assume $\delta $ is a minuscule representation and $m\in
C $. Then the transition matrix of the Markov chain $\mathcal{H}$ is the
same as the transition matrix of $(\mathcal{W}_{\ell }^{\overline{C}%
})_{\ell \geq 0}$, which is the random walk $\mathcal W$ conditioned to never exit the cone $\overline{C}$. That is,
the corresponding transition probabilities are given by 
\begin{equation*}
\Pi _{\mathcal{H}}(\mu ,\lambda )=\frac{s_{\lambda }(x)}{s_{\delta
}(x)s_{\mu }(x)}\text{ for any }\lambda ,\mu \in P_{+}\text{ such that }%
\lambda -\mu \text{ is a weight of }B(\delta ).
\end{equation*}
\end{theorem}

\begin{corollary}
\label{Cor_StayinC} With the above notation and assumptions we have for any $%
\lambda \in P_{+}$%
\begin{equation*}
\mathbb{P}_{\lambda }(\mathcal{W}_{\ell }\in \overline{C}\text{ for any }%
\ell \geq 1)=x^{-\lambda }s_{\lambda }(x)\prod_{\alpha \in
R_{+}}(1-x^{-\alpha }).
\end{equation*}
\end{corollary}

\begin{proof}
Recall that the function $\lambda \longmapsto \mathbb{P}_{\lambda }(\mathcal{%
W}_{\ell }\in \overline{C}$ for any $\ell \geq 1)$ is harmonic. By Theorem~
\ref{Th_coincide}, there is a positive constant $c$ such that $\mathbb{P}%
_{\lambda }(\mathcal{W}_{\ell }\in \overline{C}$ for any $\ell \geq
1)=cx^{-\lambda }s_{\lambda }(x)$. Now, for any sequence $\lambda
^{(a)}=am+o(a)$ of dominant weights 
\begin{equation*}
\lim_{a\rightarrow +\infty }\mathbb{P}_{\lambda ^{(a)}}(\mathcal{W}_{\ell
}\in \overline{C}\text{ for any }\ell \geq 1)=\lim_{a\rightarrow +\infty }%
\mathbb{P}_{0}(\mathcal{W}_{\ell }+\lambda ^{(a)}\in \overline{C}\text{ for
any }\ell \geq 1)=1.
\end{equation*}
On the other hand, we know by Proposition \ref{prop-_imit-psi} that $%
\lim_{a\rightarrow +\infty }x^{-\lambda }s_{\lambda }(x)=\nabla $ (see (\ref%
{def_nabla})). Therefore $c=\frac{1}{\nabla }$ and we are done.
\end{proof}

\begin{examples}
Consider the random walk $\mathcal{W}_{\ell }$ in $\mathbb{Z}^{n},$ with
steps from one point to one of the four nearest neighbors, and the condition 
$p_{a_{i}}p_{a_{\overline{i}}}$ independent of $i$ (see Example \ref%
{exa_vect}). Let $\nu \in P_{+}$. Let us compute $\mathbb{P}_{\nu }(\mathcal{%
W}_{\ell }\in \overline{C}$ for any $\ell \geq 1)$ for 
\begin{eqnarray*}
\overline{C} &=&\{\lambda \mid \lambda _{1}\geq \cdots \geq \lambda
_{n-1}\geq \lambda _{n}\geq 0,\lambda _{i}\in \mathbb{Z\}}\text{ and} \\
\overline{C} &=&\{\lambda \mid \lambda _{1}\geq \cdots \geq \lambda
_{n-1}\geq \left| \lambda _{n}\right| \geq 0,\lambda _{i}\in \mathbb{Z\cup }%
\frac{1}{2}\mathbb{Z\}}
\end{eqnarray*}
that is for the Weyl chambers of types $C_{n}$ and $D_{n}$.

\begin{enumerate}
\item In type $C_{n}$, the process $\mathcal{W}_{\ell }$ is obtained from $%
B(\omega _{1})$. The simple roots are the $\alpha _{i}=\varepsilon
_{i}-\varepsilon _{i+1},i=1,\ldots ,n-1$ and $\alpha _{n}=2\varepsilon _{n},$
we obtain ${x_{i}=\frac{1}{t_{i}\cdots t_{n-1}\sqrt{t_{n}}}}$ for any $%
i=1,\ldots n-1$ and $x_{n}=\frac{1}{\sqrt{t_{n}}}.$ The positive roots are $%
\varepsilon _{i}\pm \varepsilon _{i}$ with $1\leq i<j\leq n$ and $%
2\varepsilon _{i}$ with $i=1,\ldots ,n$. The desired probability is
therefore 
\begin{equation*}
\mathbb{P}_{\nu }(\mathcal{W}_{\ell }\in \overline{C}\text{ for any }\ell
\geq 1)=x^{-\nu }s_{\nu }^{C_{n}}(x)\prod_{1\leq i<j\leq n}(1-\frac{x_{j}}{%
x_{i}})(1-\frac{1}{x_{i}x_{j}})\prod_{1\leq i\leq n}(1-\frac{1}{x_{i}^{2}})
\end{equation*}
where $s_{\nu }^{C_{n}}(x)$ is the Weyl character of type $C_{n}$ associated
to $\nu $ specialized in $x_{1},\ldots ,x_{n}$.

\item In type $D_{n}$, the process $\mathcal{W}_{\ell }$ is also obtained
from $B(\omega _{1})$. The simple roots are the $\alpha _{i}=\varepsilon
_{i}-\varepsilon _{i+1},i=1,\ldots ,n-1$ and $\alpha _{n}=\varepsilon
_{n-1}+\varepsilon _{n}$. We obtain ${x_{i}=\frac{1}{t_{i}\cdots t_{n-2}%
\sqrt{t_{n-1}t_{n}}}}$ for any $i=1,\ldots ,n-2,$ $x_{n-1}=\frac{1}{\sqrt{%
t_{n-1}t_{n}}}$ and $x_{n}=\sqrt{\frac{t_{n-1}}{t_{n}}}.$ The positive roots
are $\varepsilon _{i}\pm \varepsilon _{i}$ with $1\leq i<j\leq n$. 
\begin{equation*}
\mathbb{P}_{\nu }(\mathcal{W}_{\ell }\in \overline{C}\text{ for any }\ell
\geq 1)=x^{-\nu }s_{\nu }^{D_{n}}(x)\prod_{1\leq i<j\leq n}(1-\frac{x_{j}}{%
x_{i}})(1-\frac{1}{x_{i}x_{j}})
\end{equation*}
where $s_{\nu }^{D_{n}}(x)$ is the Weyl character of type $D_{n}$ associated
to $\nu $ specialized in $x_{1},\ldots ,x_{n}$.
\end{enumerate}
\end{examples}

\section{Complementary results}

\label{Sec_miscell}

\subsection{Asymptotic behavior of the coefficients $f_{\protect\lambda/%
\protect\mu,\protect\delta }^{\ell}$}

In \cite{St}, Stanley studied the asymptotic behavior of $f_{\lambda /\mu
,\delta }^{\ell }$ when $V(\delta )$ is the defining representation of $%
\mathfrak{gl}_{n}$ (i.e. $\delta =\omega _{1}$ is associated to the
classical ballot problem).\ More precisely, he established that for any
fixed $\mu \in P_{+}$ and any direction $\vec{d}$ in $C$ 
\begin{equation}
\lim_{\ell \rightarrow \infty }\frac{f_{\lambda ^{(\ell )}/\mu ,\omega
_{1}}^{\ell }}{f_{\lambda ^{(\ell )},\omega _{1}}^{\ell }}=s_{\mu
}(m)  \label{Stan}
\end{equation}%
where $m=(m_{1},\ldots ,m_{n})\in d$ is such that $m_{i}\geq 0$ for any $%
i=1,\ldots n,$ the sum $m_{1}+\cdots +m_{n}$ equals $1$ and $\lambda ^{(\ell )}=\ell m+o(\ell
)$ tends to $\infty $ in the direction $\vec{d}$.

By the previous theorem, one may extend this result as follows.\ Assume $\vec d$
is a direction in $C$. By $2$ of Lemma \ref{lemma_ti}, there exists a $n$%
-tuple $t=(t_{1},\ldots ,t_{n})$ with $0<t_{i}<1$ such that $\vec d$ is the
direction of the drift $m\in C$ associated to the random walk $\mathcal{W}$
defined from $t$ as in \S\ \ref{subsec_paths_crys}.

\begin{theorem}
\label{Th_Asympt}
Assume $\delta $ is minuscule. If  $\lambda ^{(\ell )}=\ell m+o(\ell^\alpha)$ with $\alpha<2/3$, then \begin{equation}
\lim_{\ell \rightarrow \infty }\frac{f_{\lambda ^{(\ell )}/\mu ,\delta
}^{\ell }}{f_{\lambda ^{(\ell )},\delta }^{\ell }}=s_{\mu }(x).
\label{asympt}
\end{equation}%

\end{theorem}

\begin{proof}
Consider $\lambda ^{(\ell )}=\ell m+o(\ell )$ a sequence of dominants
weights which tends to infinity in the direction $\vec{d}$.\ By Proposition \ref%
{Prop_dec_skew} 
\begin{equation}
\frac{f_{\lambda ^{(\ell )}/\mu ,\delta }^{\ell }}{f_{\lambda ^{(\ell
)},\delta }^{\ell }}=\sum_{\gamma \in P}K_{\mu ,\gamma }\frac{f_{\lambda
^{(\ell )}-\gamma ,\delta }^{\ell }}{f_{\lambda ^{(\ell )},\delta }^{\ell }}%
=\sum_{\gamma \in P}K_{\mu ,\gamma }\frac{f_{\lambda ^{(\ell )}-\gamma
,\delta }^{\ell }x^{\lambda ^{(\ell )}-\gamma }}{f_{\lambda ^{(\ell
)},\delta }^{\ell }x^{\lambda ^{(\ell )}}}x^{\gamma }  \label{quo-f}
\end{equation}
where the sums are finite since the set of weight in $V(\mu )$ is finite.
Note that, for any $\gamma \in P$%
\begin{equation*}
\frac{f_{\lambda ^{(\ell )}-\gamma ,\delta }^{\ell }x^{\lambda ^{(\ell
)}-\gamma }}{f_{\lambda ^{(\ell )},\delta }^{\ell }x^{\lambda ^{(\ell )}}}=%
\frac{\mathbb{P}(\mathcal{W}_{1}\in \overline{C},\ldots ,\mathcal{W}_{\ell
}\in \overline{C},\mathcal{W}_{\ell }=\lambda ^{(\ell )}-\gamma )}{\mathbb{P}%
(\mathcal{W}_{1}\in \overline{C},\ldots ,\mathcal{W}_{\ell }\in \overline{C},%
\mathcal{W}_{\ell }=\lambda ^{(\ell )})}.
\end{equation*}
By Theorem \ref{tll-q}, we know that this quotient tends to $1$ when $\ell $
tends to infinity. This implies 
\begin{equation*}
\lim_{\ell \rightarrow +\infty }\frac{f_{\lambda ^{(\ell )}/\mu ,\delta
}^{\ell }}{f_{\lambda ^{(\ell )},\delta }^{\ell }}=\sum_{\gamma \in P}K_{\mu
,\gamma }x^{\gamma }=s_{\mu }(x)
\end{equation*}
as announced.
\end{proof}

\bigskip

\noindent\textbf{Remark:} When $\delta=\omega_{1}$ in type $A_{n-1}$ (ballot
problem), we have $p_{i}=\frac{x_{i}}{s_{\delta}(x)}$. Then $m=\frac {1}{%
s_{\delta}(x)}\sum_{i=1}^{n}x_{i}\varepsilon_{i}$. We can normalize the law
so that $s_{\delta}(x)=1$ (it suffices to replace each $x_{i}$ by $\frac{%
x_{i}}{s_{\delta}(x)}$). Then $p_{i}=m_{i}=x_{i}$. This shows that our
theorem can indeed be regarded as a generalization of (\ref{Stan}). Finally
note that the proof of (\ref{Stan}) mainly uses the representation theory of
the symmetric group. It seems nevertheless difficult to obtain a purely
algebraic proof of the limits (\ref{asympt}).

\subsection{Random walks defined from non irreducible representations}

In Section \ref{Sec_MarkovCains}, we have defined the random walk $(\mathcal{%
W}_{\ell})_{\ell\geq0}$ starting from the crystal $B(\delta)$ of the
irreducible module $V_{q}(\delta).\;$In fact, most of our results can be
easily adapted to the case where $(\mathcal{W}_{\ell})_{\ell\geq0}$ is
defined from the crystal $B(M)$ of a $U_{q}(\mathfrak{g})$-module $M$
possibly non irreducible. These random walks are therefore based on tensor
products of non irreducible representations.\ Such products will also
provide us with random walks for which the conditioned law to never exit $%
\overline{C}$ can be made explicit in terms of characters.

To do this, we consider similarly $t_{1},\ldots ,t_{n}$ some positive real
numbers and $x_{1},\ldots ,x_{N}$ such that $x^{\alpha _{i}}=t_{i}^{-1}.$
This permits to define a probability distribution on $B(M)$, setting $p_{a}=%
\frac{x^{\mathrm{wt}(a)}}{s_{M}(x)}$ where $s_{M}$ is the character of $M$,
that is the sum of the characters of its irreducible components. The random
variable $X$ on $B(M)$ is such that $X(a)=\mathrm{wt}(a)$ for any $a\in
B(M). $ The product probability distribution on $B(M)^{\otimes \ell }$
verifies $p_{b}=\frac{x^{\mathrm{wt}(b)}}{s_{M}(x)^{\ell }}$ for any $b\in
B(\delta )^{\otimes \ell }$. We introduce similarly $\Omega (M)$ as the
projective limit of the tensor products $B(M)^{\otimes \ell }$ and the
Markov chain $\mathcal{W=}(\mathcal{W}_{\ell })_{\ell \geq 0}$ such that $%
\mathcal{W}_{\ell }(b)=\mathrm{wt}(b^{(\ell )})$ for any $\ell \geq 0$. Its
transition matrix verifies 
\begin{equation*}
\Pi _{\mathcal{W}}(\beta ,\beta ^{\prime })=K_{M,\beta ^{\prime }-\beta }%
\frac{x^{\beta ^{\prime }-\beta }}{s_{M}(x)}
\end{equation*}
where $K_{M,\beta ^{\prime }-\beta }$ is the dimension of the weight space $%
\beta ^{\prime }-\beta $ in $M$.\ Write $M=\oplus _{\nu \in
P_{+}}V_{q}^{\oplus m_{\nu }}(\nu )$ for the decomposition of $M$ into its
irreducible components. Then the set of increments for the Markov chain
corresponding to $M$ is the union of the sets of transition for the Markov
chains corresponding to each $V(\nu )$ with $m_{\nu }>0$. Observe that the
Markov chain obtained for an isotypical representation $M=V_{q}^{\oplus
m_{\nu }}(\nu )$ is the same as the Markov chain for $V(\nu )$. Nevertheless
when $M$ admits at least two non isomorphic irreducible components, the
matrix $\Pi _{\mathcal{W}}$ depends on the multiplicities $m_{\nu }$.

One can consider the random variable $\mathcal{H}_{\ell }$ on $\Omega (M)$
such that $\mathcal{H}_{\ell }(b)=\mathrm{wt}(\mathfrak{P}(b(\ell)))$ for any $b\in
\Omega (M),$ that is $\mathcal{H}_{\ell }(b)$ is the highest weight of $B(b)$%
, the connected component of $B(M)^{\otimes \ell }$ containing $b^{(\ell )}$%
.\ In the proof of Theorem \ref{Th_main}, we do not use the fact that $%
B(\delta )$ is connected (or equivalently that $V_{q}(\delta )$ is
irreducible). Then the same proof shows that $\mathcal{H}=(\mathcal{H}_{\ell
})_{\ell \geq 0}$ is yet a Markov chain with transition probabilities 
\begin{equation*}
\Pi _{\mathcal{H}}(\mu ,\lambda )=\frac{m_{\mu ,M}^{\lambda }s_{\lambda }(x)%
}{s_{M}(x)s_{\mu }(x)}
\end{equation*}
where $m_{\mu ,M}^{\lambda }$ is the multiplicity of $V_{q}(\lambda )$ in $%
V_{q}(\mu )\otimes M$.

In order to obtain an analogue of Proposition \ref{Th_psi}, we will say that 
$M$ is of \emph{minuscule type} if all its irreducible components are
minuscule representations and $\dim M_{\mu }\in \{0,1\}$ for any weight $\mu 
$. In particular, the nonzero multiplicities in the decomposition $M=\oplus
_{\nu \in P_{+}}V_{q}(\nu )$ in irreducible are equal to $1$. When $M$ is of
minuscule type, we thus have $K_{M,\lambda -\mu }=m_{\mu ,M}^{\lambda }\in
\{0,1\}$ for any $\lambda ,\mu \in P_{+}$ since $K_{\nu ,\lambda -\mu
}=m_{\mu ,\nu }^{\lambda }$ for any $\nu $ and the irreducible components $%
V_{q}(\nu )$ have no common weight. We give below the table of the possible
non irreducible minuscule type representations for each root system. For a
table of the minuscule representations see \S\ \ref{subsec_minus}. 
\begin{equation*}
\begin{tabular}{|l|l|}
\hline
type & non irreducible minuscule type representations \\ \hline
$A_{n}$ & $\oplus _{j\in J}V(\omega _{j})$ for $J\subset \{1,\ldots ,n-1\}$ and $\left|
J\right| >1$ \\ \hline
$D_{n}$ & $\oplus _{j\in J}V(\omega _{j})$ for $J\subset \{1,n-1,n\}$ and $\left|
J\right| >1$ \\ \hline
$E_{6}$ & $V(\omega _{1})\oplus V(\omega _{6}).$ \\ \hline
\end{tabular}%
\end{equation*}

\begin{theorem}
Assume $M$ is of \emph{minuscule type}. Then, the transition matrix $\Pi _{%
\mathcal{H}}$ can be realized as a $\psi $-transform of the substochastic
matrix $\Pi _{\mathcal{W}}^{\overline{C}}$ where $\psi $ is the harmonic
function defined by $\psi (\lambda )=x^{-\lambda }s_{\lambda }(x)$ for any $%
\lambda \in P_{+}.$
\end{theorem}

\begin{proof}
Write $M=\oplus _{\nu \in P_{+}}V_{q}(\nu )$ the decomposition of $M$ in its
irreducible components.\ We have 
\begin{equation}
K_{M,\lambda -\mu }=\sum_{\nu }K_{\nu ,\lambda -\mu }\in \{0,1\}\text{ and }%
m_{\mu ,M}^{\lambda }=\sum_{\nu }m_{\mu ,\nu }^{\lambda }\in \{0,1\}.
\label{dec_Km}
\end{equation}
Similarly to the proof of Proposition \ref{Th_psi}, the matrix $\Pi _{\mathcal{H}}$ can
be realized as the $\psi $-transform of the substochastic matrix $\Pi _{%
\mathcal{W}}^{\overline{C}}$ since we have $K_{M,\lambda -\mu }=m_{\mu
,M}^{\lambda }$ for any $\lambda ,\mu \in P_{+}$.
\end{proof}

\bigskip

In the rest of this paragraph, we assume that $m=\mathbb{E}(X)$ belongs to $%
C $. As in Lemma \ref{lemma_ti}, this is equivalent to the assumption $%
0<t_{i}<1$ for any $i=1,\ldots ,n.\;$If we denote by $f_{\lambda /\mu
,M}^{\ell }$ the multiplicity of $V_{q}(\lambda )$ in $V_{q}(\mu )\otimes
M^{\otimes \ell },$ the decomposition of Proposition \ref{Prop_dec_skew} yet
holds. Moreover, Proposition \ref{Prop_minuscule} admits a straightforward
analogue which guarantees that each vertex $b_{\mu }\otimes b$ with $b\in
B(M)^{\otimes \ell }$ yielding a path in $\overline{C}$ is of highest
weight.\ If we consider $\lambda ^{(a)}$ a sequence of weights of the form $%
\lambda ^{(a)}=am+o(a)$, one has for $a$ large enough, $\lambda ^{(a)}\in
P_{+}$ and 
\begin{equation*}
f_{\lambda ^{(a)}/\mu ,M}^{\ell }=\sum_{\kappa \in P_{+}}f_{\kappa ,M}^{\ell
}K_{\mu ,\lambda ^{(a)}-\kappa }=\sum_{\gamma \in P}f_{\lambda ^{(a)}-\gamma
,M}^{\ell }K_{\mu ,\gamma }.
\end{equation*}
The proof of Theorem \ref{Th_coincide} leads to the

\begin{theorem}
\label{Th_CoincideM}Assume $M$ is a minuscule type representation and $m\in
C $. Then the transition matrix of $(\mathcal{W}_{\ell }^{\overline{C}%
})_{\ell \geq 0}$ is the same as
the transition matrix of the Markov chain $\mathcal{H}$. That is, the
corresponding transition probabilities are given by 
\begin{equation*}
\Pi (\mu ,\lambda )=\frac{s_{\lambda }(x)}{s_{M}(x)s_{\mu }(x)}\text{ for
any }\lambda ,\mu \in P_{+}\text{ such that }\lambda -\mu \text{ is a weight
of }B(M).
\end{equation*}
\end{theorem}

\begin{example}
Consider the minuscule type representation $M=V(\omega _{1})\oplus V(\omega
_{n-1})\oplus V(\omega _{n})$ in type $D_{n}$.\ The simple roots are the $%
\alpha _{i}=\varepsilon _{i}-\varepsilon _{i+1},i=1,\ldots ,n-1$ and $\alpha
_{n}=\varepsilon _{n-1}+\varepsilon _{n}.\;$We thus have $t_{i}=\frac{x_{i+1}%
}{x_{i}}$ for any $i=1,\ldots ,n-1$ and $t_{n}=\frac{1}{x_{n-1}x_{n}}.\;$%
This gives ${x_{i}=\frac{1}{t_{i}\cdots t_{n-2}\sqrt{t_{n-1}t_{n}}}}$ for
any $i=1,\ldots ,n-2,$ $x_{n-1}=\frac{1}{\sqrt{t_{n-1}t_{n}}}$ and $x_{n}=%
\sqrt{\frac{t_{n-1}}{t_{n}}}.$ We have $s_{M}(x)=s_{\omega
_{1}}(x)+s_{\omega _{n-1}}(x)+s_{\omega _{n}}(x)$.\ The weights $\mu $ of $M$
(an thus the possible transitions for $\mathcal{W}$ in $P$) are such that 
\begin{equation*}
\mu \in \{\pm \varepsilon _{i}\mid i=1,\ldots ,n\}\sqcup \{\pm \frac{1}{2}%
\varepsilon _{1}\pm \cdots \pm \frac{1}{2}\varepsilon _{n}\}.
\end{equation*}%
We thus have $2^{n}+2n$ possible transitions and the probability
corresponding to the transition $\mu $ is $p_{\mu }=\frac{x^{\mu }}{s_{M}(x)}%
.$
\end{example}

\section{Appendix: Miscellaneous proofs}

\subsection{Proof of Theorem \protect\ref{Th_Doob}}

The probability space on which the chain is defined is equipped with the
probability $\mathbb{P}_{x^\ast}$ defining the chain issued from the
particular point $x^\ast$.

Denote by $K_h$ the Martin kernel associated to $\Pi_h$ ; for any $x, y \in M
$ one gets 
\begin{equation*}
K_h(x, y)= {\frac{h(x^*)}{h(x)}}K(x, y). 
\end{equation*}
The Markov chain $(Y_\ell^h)$ is transient in $M$ and it converges almost
surely in the Martin topology to a random variable $Y_\infty^h$ taking
values in the Martin boundary $\mathcal{M}$ (see for instance Theorem 7.19
in \cite{W}). That means that for $\mathbb{P}_{x^\ast}$-almost all $\omega$,  the sequence of functions $%
\displaystyle 
\Bigl(K_h(\cdot , Y_\ell^h(\omega))\Bigr)_{\ell \geq 1} $ converges
pointwise to some (random) function $K_h(\cdot , Y_\infty^h(\omega))$ ;  in
this context, the hypothesis of the theorem may be written 
\begin{equation*}
\forall x \in M, \quad K_h(x, Y_\infty^h(\omega))={\frac{h(x^*)}{h(x)}}%
f(x)\quad \mbox{\rm a.s.} \ .
\end{equation*}
Since the family of functions $\Bigl\{ K_h(\cdot, \xi), \xi \in \mathcal{M}%
\Bigr\}$ separates the boundary points, the function  $K_h(\cdot ,
Y_\infty^h(\omega))$ is almost surely constant, i-e $Y_\infty^h(\omega)=\xi_%
\infty \ \mbox{\rm a.s.}$ for some fixed boundary point $\xi_\infty$.

It remains to prove that $K_h(x, \xi_\infty)$ does not depend on $x$. Note
that $\mathbb{P}_x\ll\mathbb{P}_{x^\star}$ for all $x\in M$, since $%
\Gamma(x^\ast,x)>0$. Setting $\nu_{x}(B):= \mathbb{P}_{x}[Y_\infty^h\in B]$
for any Borel set $B$ in $\mathcal{M}$ and any $x \in M$, one gets (see for
instance Theorem 7.42 in \cite{W}) 
\begin{equation}  \label{martin}
\nu_x(B)= \int_BK_h(x, \xi)\  \nu_{x^*}(d\xi).
\end{equation}
By the above, $\nu_{x^*}$ is the Dirac mass at the point $\xi_\infty$ ; the
equality (\ref{martin}) with $B= M$ gives  $1= K_h(x, \xi_\infty)$ for any $%
x \in M$ , i-e $\displaystyle f(x)={\frac{h(x)}{h(x^*)}}$. The proof is
complete.

\subsection{Proof of Lemma \protect\ref{lem-marc}}

We begin by a lemma coming from Garbit's thesis (\cite{Ga1}, \cite{Ga2}).
This lemma gives a uniform lower bound for the probability that a centered
random walk goes from the ball of radius $r+1$ to the ball of radius $r$
without leaving a cone.

Let $(R_{\ell})_{{\ell}\geq1}$ be a centered random walk in $\mathbb{R}^n$,
with finite second moment. Let $\mathcal{C}_0$ be an open convex cone in $%
\mathbb{R}^n$. We fix a unitary vector $\vec{u}$ in $\mathcal{C}_0$. For any $a>0$,
we denote by $\mathcal{C}_a$ the translated cone $\mathcal{C}_a:=a \vec{u}+%
\mathcal{C}_0, $ and for any $r>0$ we denote by $\mathcal{C}_a(r)$ the
truncated translated cone $\mathcal{C}_a(r):=\left\{x\in \mathcal{C}%
_a\mid\|x\|\leq r\right\}, $ which is non-empty if $a<r$.

\begin{lemma}[\protect\cite{Ga1}, \protect\cite{Ga2}]
\label{RG} There exist a positive integer $k$ and real positive numbers $r_0$, $a$ and $\rho$, with $a\leq r_0$ such that 
\begin{equation*}
\forall r\geq r_0, \ \forall x\in \mathcal{C}_a(r+1),\ \mathbb{P}_x\left[%
R_1\in \mathcal{C}_0,\ldots,R_k\in \mathcal{C}_0, R_k\in \mathcal{C}_a(r)%
\right]\geq\rho. 
\end{equation*}
\end{lemma}

\begin{proof}[Proof of Lemma \protect\ref{RG}]
We define, for $y\in \mathbb{R}^n$ and ${\ell}\in\mathbb{N}$,  
\begin{equation*}
p_{\ell}(y):=\mathbb{P}_0\left[y+\frac{R_1}{\sqrt {\ell}}\in \mathcal{C}%
_0,\ldots, y+\frac{R_{\ell}}{\sqrt {\ell}}\in \mathcal{C}_0, y+\frac{R_{\ell}%
}{\sqrt {\ell}}\in \mathcal{C}_{1}(2)\right].  
\end{equation*}
We have also  
\begin{equation*}
p_{\ell}(y)=\mathbb{P}_{\sqrt {\ell}y}\left[R_1\in \mathcal{C}_0,\ldots,
R_{\ell}\in \mathcal{C}_0, R_{\ell}\in \mathcal{C}_{\sqrt {\ell}}(2\sqrt {\ell}%
)\right].  
\end{equation*}
If a sequence $(y_{\ell})$ in $\mathbb{R}^n$ converges to a point $y$ in $%
\mathcal{C}_0$, then the ``Donsker line'' going through the points $%
\left(y_{\ell}+\frac{R_i}{\sqrt {\ell}}\right)_{i=1,\ldots,{\ell}}$
converges in law to a Brownian path (between times 0 and 1) starting at $y$;
the probability that this Brownian path stays in $\mathcal{C}_0$ and ends in 
$\mathcal{C}_1(2)$ is positive ; from Portemanteau theorem (c.f.\cite{Bil}),
we deduce that $\displaystyle \liminf p_{\ell}(y_{\ell})>0$. It follows
that, for any compact $K\subset \mathcal{C}_0$, there exists a positive
integer ${\ell}$ such that $\displaystyle \inf_{y\in K} p_{\ell}(y)>0$.

We choose a positive integer $k$ and a positive real $\rho$ such that, for
all $y\in\overline{\mathcal{C}_{1}(3)}$, $p_k(y)\geq\rho$. We note that if $%
\sqrt k y\in \mathcal{C}_{\sqrt k}(2\sqrt k +1) $ then $y\in\overline{%
\mathcal{C}_{1}(3)}$, and we conclude that  
\begin{equation*}
\forall z\in \mathcal{C}_{\sqrt k}(2\sqrt k +1),\ \mathbb{P}_z\left[R_1\in 
\mathcal{C}_0,\ldots, R_k\in \mathcal{C}_0, R_k\in \mathcal{C}_{\sqrt
k}(2\sqrt k)\right]\geq\rho.  
\end{equation*}
We define $r_0:=2\sqrt k$ and $a:=\sqrt k$. The four parameters of the lemma
are now fixed, and the result is proved for $r=r_0$.

We consider now $r>r_0$ and $z\in \mathcal{C}_a(r+1)$ and we distinguish two
cases.
\begin{itemize}
\item 
If $z\in \mathcal{C}_a(r_0+1)$ then $\mathbb{P}_z\left[R_1\in \mathcal{C}
_0,\ldots, R_k\in \mathcal{C}_0, R_k\in \mathcal{C}_a(r)\right]\geq\rho$,%
${}$ because $\mathcal{C}%
_a(r_0)\subset \mathcal{C}_a(r)$ and $\mathbb{P}_z\left[R_1\in \mathcal{C}_0,\ldots, R_k\in 
\mathcal{C}_0, R_k\in \mathcal{C}_a(r_0)\right]\geq\rho$.
\item 
If $z\notin \mathcal{C}_a(r_0+1)$ we remark that $\|a\vec{u}\|<r_0+1<\|z\|$ and we
denote by $z^{\prime }$ the point on the segment $[a\vec{u},z]$ such that $%
\|z^{\prime }\|=r_0+1$. We have $z^{\prime }\in \mathcal{C}_a(r_0+1)$ and we
verify, by inclusion of events, that 
\begin{equation*}
\mathbb{P}_{z^{\prime }}\left[R_1\in \mathcal{C}_0,\ldots, R_k\in \mathcal{C}%
_0, R_k\in \mathcal{C}_a(r_0)\right]\leq  \mathbb{P}_z\left[R_1\in \mathcal{C%
}_0,\ldots, R_k\in \mathcal{C}_0, R_k\in \mathcal{C}_a(r)\right].  
\end{equation*}
We conclude that this last term is $\geq \rho$.
\end{itemize}
The proof of the lemma
is complete.
\end{proof}

\bigskip

\begin{proof}[Proof of Lemma \protect\ref{lem-marc}]
We come back now to the study of a sequence of bounded independent
identically distributed random variables $(X_\ell)$ taken values in a
discrete subgroup $G$ of $\mathbb{R}^n$. We denote by $m$ there common mean
and by $(S_\ell)$ the associated random walk. Moreover we consider a cone $%
\mathcal{C}$ and an open convex subcone $\mathcal{C}_0$ satisfying the
assumptions stated in Section \ref{stayC}.

Let $\alpha\in]1/2,2/3[$. We suppose that the sequence $({\ell}%
^{-\alpha}\|g_{\ell}-{\ell}m\|)$ is bounded, and we want to prove that 
\begin{equation*}
\liminf\left(\mathbb{P}\left[S_1\in\mathcal{C},\ldots,S_{\ell}\in\mathcal{C}%
,S_{\ell}=g_{\ell}\right]\right)^{{\ell}^{-\alpha}}>0. 
\end{equation*}
It is sufficient to prove this result with $\mathcal{C}$ replaced by $%
\mathcal{C}_0$, so we suppose in the sequel that $\mathcal{C}$ satisfies the
assumptions imposed to $\mathcal{C}_0$. By Lemma \ref{psc}, we know that 
\begin{equation*}
p:=\mathbb{P}\left[ \forall {\ell} \geq 1, \ S_{\ell} \in \mathcal{C}\right]
>0. 
\end{equation*}

Since $\alpha>1/2$, the sequence ${\ell}^{-\alpha}(S_{\ell}-{\ell}m)$ goes
to zero in quadratic mean, hence in probability. For $r_0$ large enough, for
all ${\ell}$, we have 
\begin{equation}  \label{fuite-control}
\mathbb{P}\left[S_1\in \mathcal{C},\ldots,S_{\ell}\in \mathcal{C},\|S_{\ell}-%
{\ell}m\|\leq r_0 + [{\ell}^\alpha]\right]\geq p/2.
\end{equation}
From Lemma \ref{RG} applied to the random walk $R_{\ell}=S_{\ell}-{\ell}m$,
we know that there exist an integer $k>0$, a positive real number $r_0$ and
a positive real number $\rho$ such that, for all $r\geq r_0$ and all $x\in 
\mathcal{C}$ with $\|x\|\leq r+1$, 
\begin{equation*}
\mathbb{P}_x\left[S_1-m\in \mathcal{C},\ldots,S_k-km\in \mathcal{C},
\|S_k-km\|\leq r\right]\geq\rho. 
\end{equation*}
All the more, for all $r>r_0$ and all $x\in \mathcal{C}$ with $\|x\|\leq r+1$%
, 
\begin{equation}  \label{retour-arr}
\mathbb{P}_x\left[S_1\in \mathcal{C},\ldots,S_k\in \mathcal{C},
\|S_k-km\|\leq r\right]\geq\rho.
\end{equation}
We fix $k$, $r_0$ and $\rho$.

Looking to a path $(S_i)_{1\leq i\leq {\ell}+k[{\ell}^\alpha]}$ staying in
the cone and such that $\|S_{\ell}-{\ell}m\|\leq r_0 + [{\ell}^\alpha]$ and $%
\|S_{{\ell}+kj}-({\ell}+kj)m\|\leq r_0 + [{\ell}^\alpha]-j$, $1\leq j\leq[{%
\ell}^\alpha]$, we deduce from (\ref{fuite-control}) and (\ref{retour-arr})
that, for all ${\ell}$, 
\begin{equation*}
\mathbb{P}\left[S_1\in \mathcal{C},\ldots,S_{{\ell}+k[{\ell}^\alpha]}\in 
\mathcal{C}, \left\|S_{{\ell}+k[{\ell}^\alpha]}-({\ell}+k[{\ell}%
^\alpha])m\right\|\leq r_0\right]\geq\frac{p}2\rho^{{\ell}^\alpha}. 
\end{equation*}

The distance between the ball $B\left({\ell}m, r_0\right)$ and the
complementary of the cone $\mathcal{C}$ increases linearly with ${\ell}$. We
can fix an integer $d$ such that, for all ${\ell}$, the distance between the
ball $B\left(d{\ell}m, r_0\right)$ and the complementary of the cone is
greater than $2{\ell}$ times the $L^\infty$ norm of the random variables $%
X_{\ell}$.

Let us summarize what we know :

\begin{itemize}
\item For all ${\ell}$, 
\begin{equation}  \label{un}
\mathbb{P}\left[S_1\in \mathcal{C},\ldots,S_{d{\ell}+k[(d{\ell})^\alpha]}\in 
\mathcal{C}, \left\|S_{d{\ell}+k[(d{\ell})^\alpha]}-(d{\ell}+k[(d{\ell}%
)^\alpha])m\right\|\leq r_0\right]\geq\frac{p}2\rho^{(d{\ell})^\alpha}.
\end{equation}

\item Starting from a point in the ball $B\left((d{\ell}+k[(d{\ell}%
)^\alpha])m, r_0\right)$, the random walk cannot exit the cone $\mathcal{C}$
in less than $2{\ell}$ steps.
\end{itemize}

We want now to apply Theorem \ref{TLL-R}; for all integers $\ell,i>0$ with $\ell\leq i\leq 2\ell$ we define $g^{\prime }_{\ell,i}:=g_{d%
{\ell}+k[(d{\ell})^\alpha]+i}-x-(d{\ell}+k[(d{\ell})^\alpha])m$ and we
notice that, if $\|x\|\leq r_0$, then 
\begin{equation*}
\|g^{\prime }_{\ell,i}-im\|=O\left((d{\ell}+k[(d{\ell})^\alpha]+i
)^\alpha\right)=O\left(i^\alpha\right). 
\end{equation*}
Since $\alpha<2/3$, Theorem \ref{TLL-R} gives us the existence of positive
constants $c$ and $C$ such that, for all $\ell$ large enough,
\begin{equation*}
\inf_{x : \|x\|\leq r_0 \text{ and } g^{\prime }_{\ell,i}\in G} \mathbb{P}%
_x[S_i=g^{\prime }_{\ell,i}]\geq C i^{-n/2}\exp\left(-\frac{c}{i
}\|g^{\prime }_{\ell,i}-im\|^2\right). 
\end{equation*}
Modifying $c$ and $C$ if necessary, we obtain 
\begin{equation}  \label{de}
\inf_{x : \|x\|\leq r_0 \text{ and } g^{\prime }_{\ell,i}\in G} \mathbb{P}
_x[S_i=g^{\prime }_{\ell,i}]\geq C {\ell}^{-n/2}\exp\left(-c {\ell}%
^{2\alpha-1}\right).
\end{equation}

We apply this estimate to a path of the random walk starting at time $d{\ell}%
+k[(d{\ell})^\alpha]$ from the point $x+(d{\ell}+k[(d{\ell})^\alpha])m$.
Using (\ref{un}) et (\ref{de}) we obtain 
\begin{multline*}
\mathbb{P}{\Large [S_1\in \mathcal{C},\ldots,S_{d{\ell}+k[(d{\ell}%
)^\alpha]}\in \mathcal{C}, \left\|S_{d{\ell}+k[(d{\ell})^\alpha]}-(d{\ell}%
+k[(d{\ell})^\alpha])m\right\|\leq r_0,} \\
S_{d{\ell}+k[(d{\ell})^\alpha]+i}=g_{d{\ell}+k[(d{\ell})^\alpha]+i}%
{\Large ]\geq \frac{p}2\rho^{(d{\ell})^\alpha}\times C {\ell}%
^{-n/2}\exp\left(-c {\ell}^{2\alpha-1}\right).}
\end{multline*}

Moreover, we remember that, by the choice of $d$, if $\left\|S_{d{\ell}+k[(d{%
\ell})^\alpha]}-(d{\ell}+k[(d{\ell})^\alpha])m\right\|\leq r_0$ then $S_{d{%
\ell}+k[(d{\ell})^\alpha]+j}\in \mathcal{C}$ for all $j$ between 1 and ${2\ell
}$. We conclude that 
\begin{equation*}
\mathbb{P}\left[S_1\in \mathcal{C},\ldots,S_{d{\ell}+k[(d{\ell})^\alpha]+{
i}}\in \mathcal{C}, S_{d{\ell}+k[(d{\ell})^\alpha]+i}=g_{d{\ell}+k[(d%
{\ell})^\alpha]+i}\right]\geq \frac{p}2\rho^{(d{\ell})^\alpha} C {\ell}%
^{-n/2}\exp\left(-c {\ell}^{2\alpha-1}\right). 
\end{equation*}
Since $\alpha<1$, the $(d{\ell}+k[(d{\ell})^\alpha]+i)^{-\alpha}$-th
power of the right hand side has a positive limit when ${\ell}$ goes to
infinity. This gives the announced conclusion, modulo the final easy following claim:
all large enough integer can be written under the form $d{\ell}+k[(d{\ell})^\alpha]+i$ with $0<\ell\leq i\leq 2\ell$
\end{proof}

\end{document}